\documentclass[11pt]{article}

\usepackage{cmlgc}
\usepackage{ucs}
\usepackage[utf8x]{inputenc}
\usepackage[T1]{fontenc}
\usepackage{color}
\usepackage[all]{xy}

\usepackage{enumitem,epsfig}
\usepackage{amsmath,amsfonts,amssymb,amsthm,stmaryrd}
\usepackage{mathrsfs} 

\definecolor{vertfonce}{rgb}{0.20, 0.46, 0.25}
\definecolor{rougefonce}{rgb}{0.64, 0.09, 0.20}

\usepackage[breaklinks=true,
            colorlinks=true,
            linkcolor=rougefonce,
            citecolor=vertfonce]{hyperref}

\newcommand{\DD}{\mathrm{d}}
\newcommand{\Cinf}{C^\infty}
\newcommand{\ham}[1]{\mathcal{X}_{#1}}
\newcommand{\phy}{\varphi}
\newcommand{\deriv}[2]{\frac{\partial #1}{\partial #2}}
\newcommand{\lie}{\mathcal{L}}
\newcommand{\pscal}[2]{\langle #1,#2\rangle}
\newcommand{\cqfd}{\hfill $\square$\par\vspace{1ex}}

\newcommand{\norm}[1]{\left\|#1\right\|}
\newcommand{\abs}[1]{\left|#1\right|}
\newcommand{\OP}{{\rm Op}_\hbar}
\newcommand{\h}{\hbar}
\newcommand{\C}{\mathcal{C}}
\renewcommand{\O}{\mathcal{O}}
\newcommand{\formel}[1]{[\![#1]\!]}

\newcommand{\RM}{\mathbb{R}}
\newcommand{\ZM}{\mathbb{Z}}

\newcommand{\NM}{\mathbb{N}}
\newcommand{\CM}{\mathbb{C}}

\newcommand{\T}{\mathbb{T}}

\newcommand{\f}{{\vec{f}}}
\newcommand{\cO}{{\mathcal{O}}}

\newcommand{\theor}{Theorem}
\newcommand{\defin}{Definition}
\newcommand{\lemma}{Lemma}
\newcommand{\remar}{Remark}
\newcommand{\exemp}{Example}
\newcommand{\corol}{Corollary}
\newcommand{\propo}{Proposition}
\newcommand{\demon}{Proof}
\newcommand{\probl}{Problem}
\newtheorem{theo}{\theor}[section]
\newtheorem{theo*}{\theor}
\newtheorem{defi}[theo]{\defin}
\newtheorem{defi*}[theo*]{\defin}
\newtheorem{prop}[theo]{\propo}
\newtheorem{lemm}[theo]{\lemma}
\newtheorem{coro}[theo]{\corol}
\newtheorem{conj}[theo]{Conjecture}

\newcommand{\demons}[1][$\!\!$]{\noindent\textbf{\demon\ }\textsl{#1}\textbf{.}~}
\newenvironment{rema}
{\par\vspace{1ex}\refstepcounter{theo}%
\noindent\textbf{\remar~\thetheo.} }
{~\hfill\mbox{$\triangle$}\par\vspace{1ex}}

\newenvironment{demo}[1][$\!\!$]
{\demons[#1]\ }
{\cqfd}

\newcommand{\R}{\mathbb{R}}
\newcommand{\Z}{\mathbb{Z}}
\newcommand{\is}{\left(M,\omega,F\right)}
\theoremstyle{definition}
\newtheorem{exm}[theo]{Example}

\newtheorem{exer}[theo]{Exercise}
\newtheorem*{claim}{Claim}

\title{Integrable systems, symmetries, and quantization}

\author{Daniele \textsc{Sepe}\footnote{Universidade Federal
    Fluminense, Instituto de Matem\'atica, Campus do Valonguinho CEP
    24020-140, Niter\'oi (Brazil), email: \texttt{danielesepe@id.uff.br}} \and \textsc{Vũ
    Ngọc} 
  San\footnote{IRMAR (UMR 6625), Universit{\'e} de Rennes 1, Campus de
    Beaulieu, 35042 Rennes cedex (France), email: \texttt{san.vu-ngoc@univ-rennes1.fr}}}

\begin{document}
\maketitle

\begin{footnotesize}
  \noindent \textbf{MS Classification:} 37J35, 53D05, 70H06, 81Q20.
\end{footnotesize}

\begin{abstract}
  These notes are an expanded version of a mini-course given at the
  Poisson 2016 conference in Geneva. Starting from classical
  integrable systems in the sense of Liouville, we explore the notion
  of non-degenerate singularities and expose recent research in
  connection with semi-toric systems. The quantum and semiclassical
  counterpart are also presented, in the viewpoint of the inverse
  question: from the quantum mechanical spectrum, can one recover the
  classical system?
\end{abstract}

\section{Foreword}

These notes, after a general introduction, are split into four parts:
\begin{itemize}[leftmargin=*]
\item Integrable systems and action-angle coordinates
  (Section~\ref{sec:lecture1}), where the basic notions about
  Liouville integrable systems are recalled.
\item Almost-toric singular fibers
  (Section~\ref{sec:almost-toric-sing}), where emphasis is laid on a
  Morse-like theory of integrable systems.
\item Semi-toric systems (Section~\ref{sec:lecture-3-semi}), where we
  introduce the recent semi-toric systems and their classification.
\item Quantum systems and the inverse problem
  (Section~\ref{sec:quantum}), where the geometric study is applied to
  the world of quantum mechanics.
\end{itemize}

The main objective is to serve as an introduction to recent research
in the field of classical and quantum integrable systems, in
particular in the relatively new and expanding theory of semi-toric
systems, in which the authors have taken an active part in the last
ten years (cf.\cite{san_polygon,pel_san_inv,pel_san_acta,hss} and
references therein).

We are also very glad to propose, for the tireless reader, a link to
some exercises that have been given during this Summer School for this
lecture, by Yohann Le Floch and Joseph Palmer, to whom we want to
express our gratitude for their excellent work
(cf. \cite{lf-p-s-vn_exercises}).

\section{Introduction}

\subsection{Classical mechanics}

One of the main motivations for studying integrable Hamiltonian
systems is classical mechanics. Recall Newton's equation, for the
position $q\in\RM^n$ of a particle of mass $m$ under the action of a
force $\vec F(q)$:
\begin{equation}
  m \ddot q = \vec F(q)
  \label{equ:newton}
\end{equation}
This ordinary differential equation can be easily solved locally,
either theoretically via Cauchy-Lipschitz, or numerically, if the
force $\vec F$ is sufficiently smooth. However, as is well-known, the
problem of understanding the behavior for long times can be quite
tricky, in the sense that the sensitivity to initial conditions can
prohibit both theoretical and numerical approaches to obtaining
relevant qualitative and quantitative description of the
trajectories. A natural way to deal with this issue is to discover (or
impose) \emph{conserved quantities}, as these reduce the
dimension of the space where the trajectories lie.

A particularly useful setting for finding conserved quantities is the
Hamiltonian formulation of classical mechanics.  We shall assume that
forces are conservative and thus derive from a smooth potential
function $V$:
\[
\vec F(q)= - \nabla V(q).
\]
Then Newton's equation becomes equivalent to the following so-called
Hamiltonian system:
\begin{equation}
  \label{equ:hamilton_system}
  \begin{cases}
    \dot q = & \displaystyle \deriv{H}{p} \\[1em]
    \dot p = & \displaystyle -\deriv{H}{q},
  \end{cases}
\end{equation}
where the \emph{Hamiltonian} (or Hamilton function) is
\begin{equation}
  H(q,p) = \frac{\norm{p}^2}{2m} + V(q),
  \label{equ:hamiltonianV}
\end{equation}
and the new variable $p$, called the impulsion or the \emph{momentum},
corresponds to $m\dot q$ in Newton's equation. The function $H$
defined in \eqref{equ:hamiltonianV} can be interpreted physically as
the {\em energy} of the system, where the first summand represents
kinetic energy, while the second is potential energy. The above is a
dynamical system in the space $\RM^{2n}$ of the variables $(q,p)$,
called the \emph{phase space}. The physical dimension $n$ is usually
called the \emph{number of degrees of freedom} of the system. A
fundamental feature of Hamilton's formulation is that it admits a
coordinate-free presentation, where the phase space is not restricted
to $\RM^{2n}$, but can be any symplectic manifold, see
Section~\ref{sec:lecture1}.

A Hamiltonian system, in general, is a dynamical system of the
form~\eqref{equ:hamilton_system} for some function $H$ defined on
phase space.  Hamiltonian systems are ubiquitous in mechanics. The
particular form of Equation~\eqref{equ:hamiltonianV} can serve to
obtain the motion of a massive particle subject to gravity, or a
charged particle subject to an electric field $V$. Electromagnetism
can easily enter the picture: the Lorentz force
$\vec F = e \dot q \wedge \vec B$, where $n=3$ and $\vec B$ is the
magnetic vector field in $\RM^3$, also possesses a Hamiltonian
formulation, as follows. Let $A(q)$ be a magnetic potential,
\emph{i.e.} $\textup{curl} A = \vec B$. Then the motion of a particle
of charge $e$ in the electro-magnetic field $(V,\vec B)$ is obtained
by the Hamiltonian:
\[
H(q,p) = \frac{\norm{p- eA(q)}^2}{2m} + V(q),
\]
Note that the Hamiltonian formulation is not limited to finite
dimensional systems; many evolution PDEs have a `formal' Hamiltonian
structure, which gives important insights via conservation laws
(cf. \cite{palais} and references therein for many beautiful
examples).

In a Hamiltonian system, the energy is conserved along any trajectory;
thus the motion is restricted to a hypersurface of constant energy.
Of course this is very important: a Hamiltonian system never
explores the whole phase space. However, in general one cannot say much
more. The motion can be ergodic (which is sometimes said to
be a form of `chaotic' behaviour) in the sense that the motion
explores a dense subset of the energy hypersurface.  Geodesic flows
on hyperbolic surfaces have this property (cf. \cite{anosov} and
\cite[Appendix]{ballmann}).

Because of energy conservation, Hamiltonian systems of one degree of
freedom ($n=1$) become very special. Indeed, the hypersurface is
(generically) a submanifold of dimension $1$, \emph{i.e.} a curve: it necessarly coincides with the trajectory itself! For such systems,
the geometry of the energy levels is tightly related to the dynamics
of the system. In particular, an immediate corollary of energy
conservation is the following: if the hypersurface is compact, the
motion is periodic. This is very strong indeed! (For many dynamical
systems, the mere question of finding a single periodic trajectory is
open and can lead to formidable mathematical developments.)  For a one
degree of freedom Hamiltonian system, not only are almost all
trajectories periodic, but we have much more: motion is symplectically
conjugate to rotation at a `constant' speed. It is important to remark
that, in general, the `constant' depends on the energy. This is
the content of the action-angle theorem, see Theorem \ref{thm:aa} in
Section \ref{sec:lecture1}.

\subsection{Quantum mechanics}

Another motivation for these lectures is quantum mechanics, as
developed by Heisenberg, Dirac, Schrödinger and others in the first
decades of the twentieth century.  This `old' quantum mechanics can
still produce very intriguing results ({\it e.g.} superposition
principle (Schrödinger's cat), intrication, quantum computing), and is
an active area of research in mathematics (or mathematical physics),
not to mention recent Nobel prizes in physics (in particular
Thouless--Haldane--Kosterlitz in 2016, Haroche--Wineland in
2012). Some of the results that we present here have been used by
quantum chemists when studying the light spectrum of simple molecules
({\it e.g.} water, $\textup{CO}_2$; cf. for
instance~\cite{child,child-tennyson, cushman-prl, joyeux-birkhoff,
  stewart}).

The starting point is Schrödinger's equation, which can be written as
follows:
\begin{equation}
  \label{equ:schrodinger}
  -i\hbar\partial_t \psi = \hat H \psi, \quad 
  \hat H := \frac{-\h^2}{2m}\Delta + V.
\end{equation}
where the unknown is the `wave function' $\psi\in L^2(\RM^n)$. The
solution of this infinite dimensional dynamical system is a trajectory
$t\mapsto \psi_t$ living in this Hilbert space.  This
equation bears strong similarities with the
$\textup{XVII}^{\textup{th}}$ century Newton
equation~\eqref{equ:newton}, and in particular its Hamiltonian
formulation~\eqref{equ:hamilton_system}--\eqref{equ:hamiltonianV}. Given
a normalized initial condition $\psi_0$ at $t=0$, it can be solved
formally by the evolution group $\psi_t = e^{it\hat H/\h}\psi_0$.
However the physical interpretation of $\psi_t$ remains quite
mysterious, even nowadays, for several reasons. The first one is that
$\psi_t$ does not provide the deterministic position of the quantum
particle. It can only give a probabilistic answer~:
$\abs{\psi_t(x)}^2 dx$ is the probability measure to find the particle
at time $t$ at the position $x$. The second oddity follows directly
from the linearity of the equation: if one finds two solutions, their
\emph{sum} is again a solution. This cannot have anything to do with
classical mechanics!

Since the coefficients of the Schrödinger evolution equation do not
depend on time, a first natural step is to perform a partial Fourier
transform with respect to $t$; this amounts to searching for solutions
of the form
\begin{equation}
  \psi_t(x) = e^{i\lambda t/\h} u(x);\label{equ:stationary}
\end{equation}
these are called `stationary solutions', because their modulus (and
hence the associated probability measure) does not change as time
varies. Thus, the new time-independent wave function $u$ satisfies the
stationary Schrödinger equation
\[
\lambda u = \hat H u.
\]
In other words, the initial evolution equation is transformed into a
classic eigenproblem: finding eigenvectors and eigenvalues of the
operator $\hat H$. In order for this problem to be well-posed, one
needs to specify the space where $u$ should live. This, both in terms
of functional analysis and from the physics viewpoint, is out of the
scope of these notes. Localized eigenfunctions correspond to spaces
requiring a decay at infinity, most often the popular $L^2(\RM^n)$
space, equipped with Lebesgue measure. On the other hand, scattering
problems typically involve `generalized' eigenfunctions, which do
not belong to the standard $L^2$ space, but can be interpreted as
elements of another $L^2$ space equipped with a suitable weight
rapidly decreasing at infinity. In both cases, one has to take care of
the fact that the spectrum of the operator need not contain only
eigenvalues; continuous spectrum can show up, and is the signature of
non-localized solutions. In these notes, we only deal with cases where
the quantum particule is `confined', which leads to purely discrete
spectra: isolated eigenvalues with finite multiplicity.

One of the goals of these notes is to emphasize some interplay between
the Hamiltonian dynamics of the classical Hamiltonian $H$, and the
structure of the discrete spectrum of the quantum Hamiltonian
$\hat H$.  In particular we are interested in the following inverse
spectral problem. \emph{Assume that one knows the spectrum of the
  Schrödinger operator }(or, a more general quantum
operator)\emph{. Can one determine the underlying classical
  mechanics?}

In order to have a rigorous link between quantum and classical
mechanics, one needs to introduce the so-called semiclassical limit.
Following a long tradition (cf. the Landau-Lifshitz book
\cite{landau-lifshitz}), we think of the noncommutative algebra of
quantum operators on $L^2(\RM)$ as a deformation (in the algebraic
sense) of the commutative algebra of functions on $\RM^n$. In order to
write down such a deformation, we need a formal parameter, that we
call $\h$ in honor of the Planck constant.

As can be guessed from Equations~\eqref{equ:schrodinger}
and~\eqref{equ:stationary}, the limit $\h\to 0$ is highly singular. It
should not be considered as a perturbation theory. Formally, letting
$\h=0$ in~\eqref{equ:schrodinger} kills the term
$\frac{-\h^2}{2m}\Delta$, which is the quantum kinetic energy, leaving
only the potential $V$; however, in the correct semiclassical limit,
we want to recover the full classical mechanics, \emph{i.e.} both
kinetic and potential energies must survive to leading order.

The way to understand the correct limiting procedure is to introduce
\emph{fast oscillations}, not only in time
(Equation~\eqref{equ:stationary}), but also in space. (And indeed, the
semiclassical theory applies to a wide range of problems involving
high frequencies, not necessarily emanating from quantum mechanical
problems.) If $u$ oscillates at a frequency proportional to $\h^{-1}$,
then each derivative of $u$ gets multiplied by $\h^{-1}$, and
$\frac{-\h^2}{2m}\Delta$ becomes of zero order, as required. We refer
to~\cite[Exercise 5]{lf-p-s-vn_exercises}, where the basic idea of the
oscillating WKB ansatz is worked out.

The inverse spectral problem that we want to solve has to be thought
of in the semiclassical limit as well. This means that we are
interested in recovering the geometry from the asymptotic beheviour of
the spectrum as $\h\to 0$. In the case of non-degenerate (or Morse)
Hamiltonians in one degree of freedom, a positive answer is given
in~\cite{san-inverse}.  Even more recently, the symplectic
classification of the so-called semi-toric systems has paved to way to
the solution of this `spectral conjecture' for this class of
systems. An important goal of these lectures is to present semi-toric
systems (see Section \ref{sec:lecture-3-semi}).

\subsection{Integrability}

Hamiltonian systems of only one degree of freedom are `integrable',
simply because the energy $H$ is conserved along the
trajectories. Thus, by a mere application of the implicit function
theorem (solving $H(x,\xi)=\textup{const}$), one can essentially solve
the dynamical system, up to time reparameterization. Of course, the
situation can be delicate at critical points of $H$, but even then,
the fact that the two-dimensional phase space is foliated by the
possibly singular curves $H(x,\xi)=\textup{const}$ can be seen as an
`integration' of the dynamical system.

What about higher dimensions? What are the situations where one can
`geometrically integrate' the dynamics?

The first, very natural idea, is to consider systems with symmetries.
Indeed, one can hope to \emph{reduce} the symmetry and descend to a
one degree of freedom system, which then is integrable. Such systems,
for which this can be done, are called {\em integrable}. The aim of
Section \ref{sec:lecture1} is to give a precise definition, which is more
general: the strength of the theory lies in the fact that one does not
need a true symmetry: an \emph{infinitesimal} symmetry is enough.

A good example of a two-degree of freedom integral system with a
global symmetry is the spherical pendulum, which dates back to
Huygens~\cite{huygens}, and was revived by Cushman and
Duistermaat~\cite{cushman-duist}; see
also~\cite[Chap. 3]{san-panoramas}.  Symplectic geometers who prefer
compact phase spaces may be more interested in the model introduced in
\cite{sad_zhi}. This describes, amongst others, the so-called
spin-orbit system, whose phase space is $S^2 \times S^2$, and contains
a rich geometry. Both admit a quantum version, see
Section~\ref{sec:quantum}. Properties of the spherical pendulum are
worked out in~\cite[Exercises 4 and 6]{lf-p-s-vn_exercises}.

\section{Integrable systems and action-angle coordinates}
\label{sec:lecture1}

\subsection{Hamiltonian systems on symplectic manifolds}


Throughout this section, fix a $2n$-dimensional symplectic manifold
$(M,\omega)$. Non-degeneracy of $\omega$ implies that, associated to
any $f\in \Cinf(M)$, there is a unique vector field
$\ham{f} \in \mathfrak{X}(M)$ defined by
\[
\omega(\ham{f},\cdot) = -\DD f.
\]
\begin{defi}\label{defn:ham_vf}
  Given any $f\in \Cinf(M)$, the vector field $\ham{f}$ is called the
  {\em Hamiltonian vector field} associated to $f$, while its flow is
  the {\em Hamiltonian flow} of $f$.
\end{defi}
Hamiltonian vector fields and their flows are symmetries of symplectic
manifolds.
\begin{lemm}\label{lemm:ham_symmetries}
  For any $f \in \Cinf(M)$, $\lie_{\ham{f}}\omega = 0$, {\it i.e.}
  $\ham{f}$ is an infinitesimal symmetry of $(M,\omega)$.
\end{lemm}
\begin{proof}
  Using Cartan's formula, obtain that
  \[
  \lie_{\ham{f}}\omega = \iota_{\ham{f}}\DD\omega +
  \DD(\iota_{\ham{f}}\omega) = 0 + \DD (-\DD f) =0.
  \]
\end{proof}

\begin{exm}\label{exm:ham_vf_can_coords}
  Recall that any symplectic manifold admits local Darboux
  coordinates, {\it i.e.} for any $p \in M$, there exists an open
  neighborhood $U \subset M$ of $p$ and a coordinate chart
  $\varphi: \left(V, \omega_{\mathrm{can}}\right) \to (U, \omega)$,
  where $V \subset \R^{2n}$ is an open neighbourhood of the origin and
  $$ \omega_{\mathrm{can}} := \sum\limits_{i=1}^n \DD \xi_i \wedge\DD x_i =: \DD
  \xi\wedge\DD x, $$
  \noindent
  such that $\varphi^*\omega = \omega_{\mathrm{can}}$
  (cf. \cite[Theorem 3.15]{mcduff-salamon}). In these coordinates, it
  is instructive to calculate $\ham{x_j}$ and $\ham{\xi_j}$. For
  instance, if $v \in \R^{2n}$, comparing
  \[
  \omega(\ham{\xi_j},v) = (\DD \xi\wedge\DD x)(\ham{\xi_j},v) =
  \pscal{\DD\xi(\ham{\xi_j})}{\DD x(v)} - \pscal{\DD\xi(v)}{\DD
    x(\ham{\xi_j})}
  \]
  and
  \[
  \omega(\ham{\xi_j},v)=-d\xi_j(v),
  \]
  we see that, for all $i$, $\DD\xi_i(\ham{\xi_j})=0$ and
  $\DD x_i(\ham{\xi_j})=\delta^i_j$. Therefore
  $\ham{\xi_j}=\deriv{}{x_j}$. Similarly,
  $\ham{x_j}=-\deriv{}{\xi_j}$.
\end{exm}

The symplectic form $\omega$ endows $\Cinf(M)$ with the following
algebraic structure, which is, in some sense, compatible with the
underlying smooth
structure. 

\begin{defi}\label{defi:poisson_symp}
  The {\em Poisson bracket} induced by $\omega$ on $\Cinf(M)$ is the
  $\R$-bilinear, skew-symmetric bracket
  $\{\cdot , \cdot \} : \Cinf(M) \times \Cinf(M) \to \Cinf(M)$ defined
  as
  \[
  \{f,g\} = \omega(\ham{f},\ham{g}),
  \]
  \noindent
  for any $f, g \in \Cinf(M)$.
\end{defi}
The following lemma, whose proof is left as an exercise to the reader
(cf. \cite[Section 3]{mcduff-salamon}), illustrates some fundamental
properties of the above Poisson bracket.

\begin{lemm}\label{lemm:properties_Poi}
  \mbox{}
  \begin{itemize}[leftmargin=*]
  \item The Poisson bracket $\{\cdot,\cdot\}$ makes $\Cinf(M)$ into a
    Lie algebra, called the algebra of {\em classical observables} or
    {\em Hamiltonians}.
  \item The Poisson bracket $\{\cdot,\cdot\}$ satisfies {\em the
      Leibniz identity}, {\it i.e.} for all $f,g,h \in \Cinf(M)$,
    \[
    \{f,gh\}=g\{f,h\}+\{f,g\}h.
    \]
  \item For any $f, g \in \Cinf(M)$ and any (possibly locally defined)
    smooth map $\chi: \R \to \R$, the Poisson bracket
    $\{\cdot,\cdot\}$ satisfies
    \[
    \{f,\chi(g)\} = \{f,g\}\chi'(g).
    \]
  \item The map
    \begin{equation*}
      \begin{split}
        \left(\Cinf(M), \{\cdot,\cdot\}\right) &\to
        \left(\mathfrak{X},[\cdot,\cdot]\right) \\
        f &\mapsto \ham{f},
      \end{split}
    \end{equation*}
    \noindent
    where $[\cdot,\cdot]$ denotes the standard Lie bracket on vector
    fields, is a Lie algebra homomorphism.
  \end{itemize}
\end{lemm}

Poisson brackets arose naturally in the study of Hamiltonian mechanics
(cf. \cite{weinstein_poisson} and references therein); to this end, it
is worthwhile mentioning that, for any $f \in \Cinf(M)$, $\ham{f}$, as
a derivation, is just the Poisson bracket by $f$; in other words the
evolution of a function $g$ under the flow of $\ham{f}$ is given by
the equation
\[
\dot{g}=\{f,g\}.
\]
Indeed, $\{f,g\}=\omega(\ham{f},\ham{g})=\DD g(\ham{f})$.

\begin{exm}\label{exm:Poisson_Darboux}
  In local Darboux coordinates $(x,\xi)$ (cf. Example
  \ref{exm:ham_vf_can_coords}), it can be shown that, for all
  $i,j = 1,\ldots,n$, $\{\xi_i,x_j\}=\delta_{ij}$,
  $\{\xi_i,\xi_j\}=0= \{x_i,x_j\}$. Thus, for any $f,g \in \Cinf(M)$,
  locally the Poisson bracket equals
  \[
  \{f,g\} = \DD g(\ham{f})=
  \sum\limits_{i=1}^n\left(\deriv{f}{\xi_i}\deriv{g}{x_i} -
    \deriv{f}{x_i}\deriv{g}{\xi_i}\right).
  \]
\end{exm}

\begin{rema}\label{rmk:natural}
  Definitions \ref{defn:ham_vf} and \ref{defi:poisson_symp} only
  depend on the symplectic structure and hence must behave naturally
  with respect to symplectomorphisms. For instance
  $\{f,g\}\circ\phy=\{f\circ\phy,g\circ\phy\}$ if
  $\phy : (M,\omega) \to (M,\omega)$ is a symplectomorphism (and this
  can even be taken as a characterisation of
  symplectomorphisms). Therefore, the Hamiltonian flow of
  $\phy^*f=f\circ\phy$ is mapped by $\phy$ to the Hamiltonian flow of
  $f$. In fact the naturality of the definitions yields, for any
  symplectomorphism $\phy$,
  \[
  \ham{\phy^*f}=\phy^*\ham{f}:=d\phy^{-1}\ham{f}\circ\phy.
  \]
\end{rema}

\subsection{Integrals of motion and Liouville integrability}

In the Hamiltonian formulation of classical mechanics, one of the most
general problems is to study the dynamics of the Hamiltonian vector
field $\ham{H}$ defined by a physically relevant smooth function
$H : (M,\omega) \to \R$, {\it e.g.} the sum of kinetic and potential
energy. 
Fix a smooth function $H : (M,\omega) \to \R$; the dynamics of
$\ham{H}$ are confined to level sets of the function $H$, since
$\DD H\left(\ham{H}\right)=\{H,H\} =0$. The simplest possible
(non-trivial) scenario is when it is possible to constrain the
dynamics of $\ham{H}$ `as much as possible' by using symmetries or,
using Noether's theorem, constants of motion.

\begin{defi}\label{defn:integral}
  An \emph{(first) integral} of a Hamiltonian $H\in\Cinf(M)$ is a
  function that is invariant under the flow of $\ham{H}$, {\it i.e.} a
  function $f \in \Cinf(M)$ such that $\{H,f\}=0$.
\end{defi}

Suppose that $H$ admits such a first integral $f_2$ (other that $H$
itself) on a symplectic manifold $M$ of dimension $2n$. Near any point
$m$ where $df_2\neq 0$, one can perform a reduction of the dynamics to
a new Hamiltonian system defined by the restriction of $H$ to the
space $M_{f_2,h}$ of local $f_2$-orbits near $m$ living in the level
set $\{f_2= f_2(m)=:c_2\}$:
\[
M_{f_2,h} := (f_2^{-1}(c_2), m)/\ham{f_2},
\]
where we use the manifold pair (or germ) notation to indicate that we
restrict to a sufficiently small neighborhood of $m$. We leave to the
reader to prove that this (local) $2n-2$ manifold is again symplectic
($\ham{f_2}$ is both tangent and symplectically orthogonal to
$f_2^{-1}(c_2)$). Now, if $f_3$ is a new first integral for $H$, which
is independent of both $H$ and $f_2$, and whose symmetry (in the sense
of Noether) commutes with that of $f_2$, then we can repeat the
process of reduction. After $n$ steps, the Hamiltonian system is
completely reduced on a zero-dimensional manifold and hence becomes
trivial. Unfolding the reduction steps back, we are in principle able
to `completely integrate' the original dynamics, which leads to the
following definition
(cf. Proposition \ref{prop:fibre}).

\begin{defi}\label{defn:ci}
  A Hamiltonian $H \in \Cinf(M)$ is \emph{completely integrable} if
  there exist $n-1$ independent functions $f_2,\dots,f_n$ which are
  integrals of $H$ and moreover \emph{pairwise Poisson commute}, {\it
    i.e.} for all $i, j = 2,\ldots, n$, $\{f_i,f_j\}=0$.
\end{defi}


While from a mechanical perspective, the function $H$ may be
significant, Definition \ref{defn:ci} implies that it does not have
any distinguished mathematical role from the other functions
$f_2,\dots,f_n$.  This is the perspective of these notes, which
motivates the following definition.

\begin{defi} \label{defi:CI} A \emph{completely integrable Hamiltonian
    system} is a triple $\left(M,\omega,F = (f_1,\ldots,f_n)\right)$,
  where $(M,\omega)$ is a $2n$-dimensional symplectic manifold and the
  components of $F : (M,\omega) \to \R^n$ are
  \begin{itemize}[leftmargin=*]
  \item pairwise in involution, {\it i.e.} for all
    $i,j = 1,\ldots, n$, $\{f_i,f_j\}=0$, and
  \item functionally independent almost everywhere, {\it i.e.} for
    almost all $p \in M$, $\DD_p F$ is onto.
  \end{itemize}
  The number $n$ denotes the {\em degrees of freedom} of
  $\left(M,\omega,F = (f_1,\ldots,f_n)\right)$
\end{defi}

Throughout these notes, a triple as in Definition \ref{defi:CI} is
simply referred to as an integrable system. In order to explain why
the dimension of the symplectic manifold in Definition \ref{defi:CI}
is twice the number of functions, it is worthwhile observing that
integrable systems are intimately linked to {\em Lagrangian
  foliations}. To make sense of this object, some more notions are
introduced (cf. \cite[Chapter 2]{mcduff-salamon} for a more thorough
treatment of the objects discussed below).


\begin{defi}\label{defn:symp_orth}
  Given a subspace $W$ of a symplectic vector space $(V,\omega)$, its
  {\em symplectic orthogonal} $W^{\omega}$ is the subspace defined by
  $$ W^{\omega} := \left\{ v \in V \mid \forall\, w \in W \,  \,
    \omega(v,w) = 0 \right\}.$$
\end{defi}

Given a symplectic vector space $\left(V,\omega\right)$ and a subspace
$W$, non-degeneracy of $\omega$ implies that
$ \dim W + \dim W^{\omega} = \dim V$. The following are important
types of subspaces of symplectic vector spaces.


\begin{defi}\label{defi:subspaces}
  A subspace $W$
  of a symplectic vector space $(V,\omega)$ is said to be
  \begin{itemize}[leftmargin=*]
  \item {\em isotropic} if $W \subset W^{\omega}$;
  \item {\em coisotropic} if $W^{\omega} \subset W$;
  \item {\em Lagrangian} if it is both isotropic and coisotropic, {\it
      i.e.} $W = W^{\omega}$.
  \end{itemize}
\end{defi}

The condition of $W$ being isotropic is equivalent to
$\omega|_W \equiv 0$; moreover, if $W$ is isotropic, then
$\dim W \leq \frac{1}{2} \dim V$ with equality if and only if it is
Lagrangian. Thus Lagrangian subspaces are precisely the maximally
isotropic ones.

\begin{defi}\label{defi:symp_subm}
  A submanifold $N$ of a symplectic manifold $(M,\omega)$ is said to
  be {\em isotropic} (respectively {\em coisotropic}, {\em
    Lagrangian}) if, for all $p \in N$, the subspace
  $T_pN \subset (T_p M,\omega_p)$ is isotropic (respectively
  coisotropic, Lagrangian).
\end{defi}

Lagrangian submanifolds are very important in the study of symplectic
topology\footnote{So much so that Weinstein once wrote `Everything is
  a Lagrangian submanifold!' (cf. \cite{weinstein}).}; for the purpose
at hand, `families' of Lagrangian submanifolds play a particularly
important r\^ole as they are given locally by integrable systems. The
following result, stated below without proof (as it is Exercise 9 in
\cite{lf-p-s-vn_exercises}), makes the above precise
(cf. \cite[Proposition 7.3]{weinstein_symp}).

\begin{prop}
  \label{prop:fibre}
  Let $(M,\omega)$ be a $2n$-dimensional symplectic manifold and let
  $F=(f_1,\dots,f_n):U\subset M\to\RM^n$ be a smooth map such that the
  differentials $\DD f_1,\dots, \DD f_n$ are linearly independent at
  each point of $U$. Then the connected components of the level sets
  $F^{-1}(c)$, $c\in\RM^n$ are Lagrangian if and only if, for all
  $i,j = 1, \ldots, n$, $\{f_i,f_j\} =0$. In this case, the
  Hamiltonian vector fields $\ham{f_j}$, $i=1,\dots,n$ span the
  tangent space of the leaves $F^{-1}(c)$.
\end{prop}

To conclude this section, we prove that integrable systems induce
infinitesimal Hamiltonian $\R^n$-actions.

\begin{lemm}\label{lemm:cihs_as_inf_act}
  Given an integrable system
  $\left(M,\omega, F = (f_1,\ldots,f_n)\right)$, the map
  \begin{equation}
    \label{eq:6}
    \begin{split}
      \R^n & \to \mathfrak{X}_{\mathrm{Ham}}(M) \\
      (t_1,\ldots,t_n) &\mapsto \sum\limits_{i=1}^n t_i \ham{f_i}
    \end{split}
  \end{equation}
  is a Lie algebra homomorphism, where
  $\mathfrak{X}_{\mathrm{Ham}}(M)$ denotes the subalgebra of
  Hamiltonian vector fields of $\mathfrak{X}(M)$.
\end{lemm}
\begin{proof}
  The map of equation \eqref{eq:6} is manifestly linear. Thus, to
  prove the result, it suffices to show that, for all
  $i,j=1,\ldots,n$, $[\ham{f_i},\ham{f_j}] = 0$. However, by Lemma
  \ref{lemm:properties_Poi}, for any $i,j=1,\ldots, n$,
  $[\ham{f_i},\ham{f_j}] = \ham{\{f_i,f_j\}} = \ham{0} = 0$.
\end{proof}

\subsection{Examples of integrable systems}\label{sec:exampl-integr-syst}
Before proceeding to prove further properties of integrable systems
(cf. Section \ref{sec:local-normal-forms}), we introduce a few
examples which are going to be used throughout the notes.

\begin{exm}[The case $n=1$]\label{exm:n=1}
  An integrable system on a 2-dimensional symplectic manifold is a
  smooth function whose differential is non-zero almost everywhere. As
  a family of more concrete examples, consider a closed, orientable
  surface of genus $g \geq 0$ embedded in $\R^3$ endowed with the
  natural area form; the height function defines an integrable system.
\end{exm}

Next, we mention a few examples from Hamiltonian mechanics.

\begin{exm}[The classical spherical pendulum]\label{exm:sp}
  Identify $T\R^3 \cong T^*\R^3$ using the standard Euclidean metric,
  so that $T\R^3$ inherits a symplectic form, which, in standard
  coordinates $(x,y)$, is given by
  $\Omega =\sum\limits_{i=1}^3 \DD y_i \wedge \DD x_i$. Consider
  furthermore the unit sphere $S^2 \hookrightarrow \R^3$; the
  restriction of $\Omega$ to the submanifold
  $TS^2 \hookrightarrow T \R^3$ defines a symplectic form on $TS^2$,
  henceforth denoted by $\omega$. The restrictions of the functions
  $H(x,y) = \frac{1}{2}\lVert y \rVert^2 + x_3$ and
  $J(x,y) = x_1y_2 - x_2 y_1$ to $TS^2$ define an integrable system
  known as the spherical pendulum, which has been extensively studied
  as it is one of the first integrable systems in which {\em
    Hamiltonian monodromy} was observed (cf. \cite{dui} and references
  therein).
\end{exm}
\begin{figure}[h]
  \centering
  \includegraphics[width=0.6\linewidth]{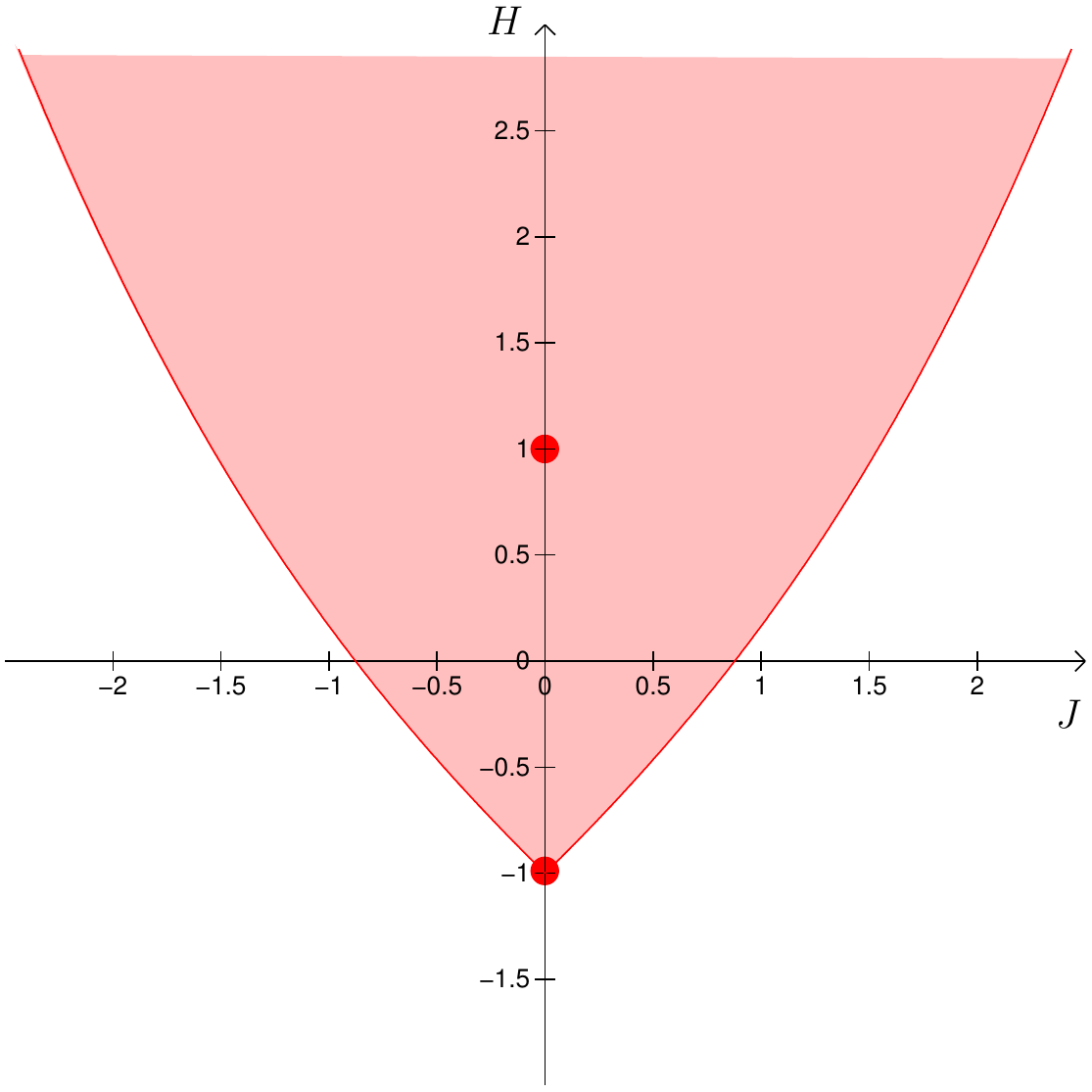}
  \caption{The image of the map $F=(J,H)$ of the spherical pendulum
    (Example~\ref{exm:sp}). The red dots are critical values of rank
    0, and the red curve is the set of critical values of rank 1, see
    Definition~\ref{defn:regular}.}
  \label{fig:classical-pendulum}
\end{figure}

\begin{exm}[Coupled angular momenta on $S^2 \times S^2$,
  cf. \cite{sad_zhi}]
  \label{exm:cou_ang_mom}
  Denote the spheres of radius $a,b > 0$ centered at the origin in
  $\R^3$ by $S^2_a$ and $S^2_b$. Each sphere is endowed with the
  standard area form which gives the corresponding sphere the expected
  area ($4\pi a^2$ and $4\pi b^2$ respectively); these are denoted by
  $\omega_a$ and $ \omega_b$ respectively. Consider the symplectic
  manifold $\left(S^2 \times S^2,\omega_{a,b}\right)$, where
  $\omega_{a,b} = \mathrm{pr}_1^* \omega_a + \mathrm{pr}_2^*
  \omega_b$,
  and, for $i=1,2$, $\mathrm{pr}_i$ denotes projection onto the $i$th
  component. For any $ 0 \leq t \leq 1$, the functions
  \[
  H_t := \frac{1-t}{a} y_3 + \frac{t}{ab} \langle x, y \rangle
  \qquad J:= y_3 + x_3,
  \]
  where $x = (x_1,x_2,x_3)$ and $y = (y_1,y_2,y_3)$ are coordinates on
  the ambient $\R^3$ for the first and second sphere respectively,
  define an integrable system on
  $\left(S^2 \times S^2, \omega_{a,b}\right)$. For all but two values
  of $t$, the corresponding integrable system is an example of a
  semi-toric system (cf. Section \ref{sec:lecture-3-semi}).
\end{exm}

The next family of examples comes from Hamiltonian group actions.

\begin{exm}[Symplectic toric manifolds]\label{exm:symp_toric}
  Suppose that an $n$-dimensional torus $\mathbb{T}^n$ acts on a
  $2n$-dimensional symplectic manifold $(M,\omega)$
  effectively\footnote{The only element that acts as the identity at
    all points of $M$ is the identity of the group.} by
  symplectomorphisms. If $\mathfrak{t}$ denotes the Lie algebra of
  $\mathbb{T}^n$, there is an induced homomorphism of Lie algebras
  $\mathfrak{t} \to \mathfrak{X}(M)$, sending $\eta$ to $X_{\eta}$,
  where, for $p \in M$,
  \[
  X_{\eta}(p):= \frac{\DD}{\DD t}\bigg|_{t=0} \exp(t \eta) \cdot p,
  \]
  where $\exp : \mathfrak{t} \to \mathbb{T}^n$ denotes the exponential
  map and $\cdot$ denotes the torus action. Since $\mathfrak{t}$ is
  abelian, it can be checked that the above map defines an
  infinitesimal Hamiltonian $\mathfrak{t} \cong \R^n$-action
  (cf. Lemma \ref{lemm:cihs_as_inf_act}). The group action is said to
  be {\em Hamiltonian} if there exists a map, called {\em moment map}
  $\mu : M \to \mathfrak{t}^*$, such that, for all
  $\eta \in \mathfrak{t}$,
  \[
  \omega(X_{\eta}, \cdot) = - \DD \langle \mu, \eta \rangle,
  \]
  where $\langle \cdot, \cdot \rangle$ denotes the canonical pairing
  between $\mathfrak{t}$ and $\mathfrak{t}^*$. If the above group
  action is Hamiltonian, the triple $(M,\omega,\mu)$ is known as a
  {\em symplectic toric manifold}. Identifying
  $\mathfrak{t}^* \cong \R^n$, a symplectic toric manifold defines an
  integrable system, which is referred to as being {\em toric}.
\end{exm}

The last two examples provide suitable {\em local normal forms} for
integrable systems, cf. Section \ref{sec:local-normal-forms}.

\begin{exm}[Local normal forms]\label{exm:local_normal_form_regular_point}
  \mbox{}
  \begin{enumerate}[label= \alph*), ref= (\alph*), leftmargin=*]
  \item \label{item:13} Consider the standard symplectic vector space
    $\left(\R^{2n},\omega_{\mathrm{can}}\right)$ with coordinates
    $(x,\xi)$ as in Example \ref{exm:ham_vf_can_coords}. The map
    $\xi:= (\xi_1,\ldots,\xi_n) : \R^{2n} \to \R^n$ defines an
    integrable system.
  \item \label{item:14} Consider the manifold
    $\R^n \times \mathbb{T}^n \cong T^* \mathbb{T}^n$ equipped with
    the canonical symplectic form
    $\omega_0 = \sum\limits_{i=1}^n \DD \xi_i \wedge \DD \theta_i$,
    where $\xi$ and $\theta$ are the standard coordinates on $\R^n$
    and $\mathbb{T}^n = \R^n/\Z^n$ respectively. The map
    $\xi:= (\xi_1,\ldots,\xi_n) : \left(\R^n \times
      \mathbb{T}^n,\omega_0\right) \to \R^n$
    defines an integrable system all of whose fibers are compact.
  \end{enumerate}
\end{exm}

\subsection{The classification problem}\label{sec:class-probl}
To prove the first non-trivial properties of integrable systems
(cf. Section \ref{sec:local-normal-forms}), some notion of {\em
  equivalence} of integrable systems has to be established. There are
several distinct notions of equivalence of integrable systems in the
literature (cf.~\cite[Definition 3.2.6]{san-panoramas}, \cite[Section
1.9]{bolsinov-fomenko-book} and \cite[Section 3.1]{wang} for thorough
reviews). We concentrate on notions, which, loosely speaking, yield
that two integrable systems are isomorphic if and only if they possess
`isomorphic singular Lagrangian foliations'. This is a subtle issue,
due to the presence of flat functions in the $\Cinf$ category. To make
formal sense of the above idea, we use an algebraic approach.

\begin{defi}\label{defn:comm}
  Let $\left(M,\omega, F=(f_1,\ldots,f_n)\right)$ be an integrable
  system and let $U \subset M$ be open. The {\em commutant} of
  $\left(M,\omega, F\right)$ in $U$ is
  $$ \C_F(U) := \left\{ g \in C^{\infty}(U) \mid \forall i=1,\ldots,n
    \, \left\{f_i,g\right\} = 0 \right\}. $$
\end{defi}

The Leibniz and Jacobi identities yield the following result.

\begin{coro}\label{cor:subalg}
  The commutant $\C_F(U)$ of $\left(M,\omega, F\right)$ in $U$ is a
  Poisson subalgebra of
  $\left(C^{\infty}(U),\left\{\cdot,\cdot\right\}|_U\right)$.
\end{coro}


Commutants provide the correct algebraic tool to state (local)
equivalence of integrable systems defined on a given symplectic
manifold.

\begin{defi}
  \label{defi:eq-faible}
  Let $\left(M,\omega,F\right)$ and $\left(M,\omega,G\right)$ be
  integrable systems. Say that $F$ and $G$ are {\em equivalent on an
    open subset $U \subset M$} if
  \begin{equation*}
    \C_{F}(U)=\C_{G}(U).
  \end{equation*}
  We denote it as $F\sim_U G$. If, in the above, $U = M$,
  $\left(M,\omega,F\right)$ and $\left(M,\omega,G\right)$ are said to
  be {\em equivalent}; this is denoted by $F \sim G$.
\end{defi}

The subscript $U$ is omitted whenever it is not ambiguous from the
context. It is possible to provide a {\em geometric} interpretation of
Definition \ref{defi:eq-faible}, by introducing the following object.

\begin{defi}\label{defn:leaf_space}
  Given an integrable system $\left(M,\omega,F\right)$,
  \begin{itemize}[leftmargin=*]
  \item a connected component of a fiber of $F$ is called a {\em
      leaf};
  \item its {\em leaf space} $\mathcal{L}$ is the quotient of $M$ by
    the equivalence relation which identifies points on the same leaf.
  \end{itemize}
\end{defi}

\begin{rema}\label{rmk:top_not_smooth}
  The above notion of leaf is {\em topological} in nature, for there
  is no guarantee that a leaf in the above sense be a(n immersed)
  submanifold of the ambient symplectic manifold.
\end{rema}

In general, the leaf space $\mathcal{L}$ of an integrable system
$\left(M,\omega,F\right)$ need not be a `nice' topological
space. However, continuous functions on it can be identified with
continuous functions on $M$ which are constant along the connected
components of the fibers of $F$. On the other hand, the commutant
$\C_F(M)$ is the algebra of {\em smooth} functions that are constant
on the connected components of the fibers of $F$. Thus it is possible
to declare that $\C_F(M)$ to be the space of smooth functions on
$\mathcal{L}$. Having done this, $F\sim G$ establishes a (local) {\em
  smooth} equivalence between the leaf spaces of the integrable
systems $\left(M,\omega,F\right)$ and $\left(M,\omega,G\right)$.

Finally, we can define a notion of symplectic equivalence between
integrable systems.

\begin{defi}\label{defn:equivalence}
  For $i=1,2$, let $\left(M_i,\omega_i, F_i\right)$ be an integrable
  system. Say that $\left(M_1,\omega_1, F_1\right)$ is {\em
    symplectically equivalent} to $\left(M_2,\omega_2, F_2\right)$ if
  there exists a symplectomorphism
  $\phy: \left(M_1,\omega_1\right)\to \left(M_2,\omega_2\right) $ such
  that $F_1 \sim \phy^*F_2:=F_2 \circ \phy$.
\end{defi}

When considering some special families of integrable systems whose
leaf spaces are particularly `nice' (as in the case of the semi-toric
systems considered in Section \ref{sec:lecture-3-semi}), Definition
\ref{defn:equivalence} is equivalent to the following notion of
equivalence of integrable systems, which is geometrically simpler, but
{\it a priori}, stronger.

\begin{defi}\label{defn:strong_equiv}
  Two integrable systems $\left(M_1,\omega_1, F_1\right)$ and
  $\left(M_2,\omega_2, F_2\right)$ are said to be {\em strongly
    symplectically equivalent} if there exists a pair
  $\left(\varphi, g \right)$, where
  $\varphi : (M_1,\omega_1) \to (M_2,\omega_2)$ is a symplectomorphism
  and $g: F_1(M_1) \to F_2(M_2)$ is a diffeomorphism\footnote{{\it A
      priori} there is no guarantee that $F_1(M_1)$ is an open subset
    of $\R^n$. Throughout these notes, if $A \subset \R^n$ is any
    subset, a map $H: A \to \R^m$ is said to be {\em smooth} if for
    every $x \in A$, there exists an open neighbourhood $W$ and a
    smooth map $H_W : W \to \R^m$ extending $H|_{A \cap W}$. A
    diffeomorphism is therefore a smooth map whose inverse is also
    smooth in the above sense.}, making the following diagram commute
  \begin{equation*}
    \xymatrix{(M_1,\omega_1) \ar[r]^-{\varphi} \ar[d]_-{F_1} &
      (M_2,\omega_2) \ar[d]^-{F_2} \\
      F_1(M_1) \ar[r]_-{g} & F_2(M_2).}
  \end{equation*}
\end{defi}

\begin{coro}\label{coro:stronger}
  Two strongly symplectically equivalent integrable systems are
  symplectically equivalent.
\end{coro}

\begin{rema}\label{rk:strictly_stronger}
  There are examples of symplectically equivalent integrable systems
  which fail to be strongly symplectically equivalent
  (cf. \cite[Appendix]{chaperon}).
\end{rema}
\begin{rema}
  In the setting of strong equivalence it is natural to consider a
  slightly stronger notion of \emph{orientation preserving} strong
  equivalence, where we demand that $g$ preserves the orientation of
  the base, which means $\det \DD g >0$. This often leads to
  simpler classification results.
\end{rema}
\subsection{Local normal forms: the case of regular points and
  leaves}\label{sec:local-normal-forms}
The aim of this section is to describe the local structure of an
integrable system near regular points and leaves (cf. Theorems
\ref{thm:dar_car} and \ref{thm:aa}). Intuitively, these results stem
from finding `integrations' of the infinitesimal $\R^n$-action
attached to an integrable system as in Lemma
\ref{lemm:cihs_as_inf_act}.

\begin{defi}\label{defn:regular}
  Let $(M,\omega,F)$ be an integrable system with $n$ degrees of
  freedom.
  \begin{itemize}[leftmargin=*]
  \item A point $m\in M$ is said to be {\em regular} for
    $(M,\omega,F)$ if $dF(m)$ has maximal rank, equal to
    $n$. Otherwise, $m$ is said to be {\em singular}.
  \item A leaf $\Lambda \subset M$ is said to be {\em regular} if all
    of its points are regular. Otherwise, it is said to be {\em
      singular}.
  \end{itemize}
\end{defi}

To simplify the statement of results regarding the local normal form
of integrable systems, it is useful to observe that they behave well
under restrictions.

\begin{defi}\label{defn:subsystem}
  Given an integrable system $\left(M,\omega,F\right)$ and an open
  subset $U \subset M$, its {\em subsystem relative to $U$} is the
  integrable system $\left(U,\omega|_U, F|_U\right)$.
\end{defi}

Fix an integrable system $(M,\omega,F)$ and suppose that $m \in M$ is
a regular point. Darboux's theorem (cf. \cite[Theorem
3.15]{mcduff-salamon}) states that, locally near $m$, the symplectic
form $\omega$ can be put in standard form; on the other hand, since
$F$ is a submersion at $m$, the local normal form for submersions
implies that, near $m$, $F$ is simply given by a projection. Thus a
natural question is to ask whether it is possible to attain the two
above local normal forms {\em at once}. This is the content of the
following well-known result.


\begin{theo}[Darboux-Carathéodory]\label{thm:dar_car}
  Let $\left(M,\omega,F\right)$ be an integrable system with $n$
  degrees of freedom and let $m \in M$ be regular. Then there exist
  open neighborhoods $U \subset M$, $V \subset \R^{2n}$ of $m$ and of
  the origin respectively, such that the subsystem of
  $\left(M,\omega,F\right)$ relative to $U$ is strongly symplectically
  equivalent to the subsystem of
  $\left(\R^{2n},\omega_{\mathrm{can}}, \xi\right)$ (cf. Example
  \ref{exm:local_normal_form_regular_point} \ref{item:13}) relative to
  $V$ via a pair of the form
  $\left(\varphi,\mathrm{id}\right)$. 
\end{theo}
\begin{proof}
  Using Darboux's theorem and the local normal form for submersions,
  it may be assumed, without loss of generality, that
  \begin{itemize}[leftmargin=*]
  \item
    $\left(M,\omega\right)
    =\left(\R^{2n},\omega_{\mathrm{can}}\right)$, $m = 0$, $F(0) = 0$;
  \item $F : \R^{2n} \to \R^n$ is a surjective submersion with
    connected fibers which admits a smooth section
    $\sigma : \R^n \to \R^{2n}$ with $\sigma(0) = 0$.
  \end{itemize}
  The infinitesimal action of Lemma \ref{lemm:cihs_as_inf_act} yields
  an action of the bundle of abelian Lie algebras $T^* \R^n \to \R^n$
  (\textit{i.e.} viewing each cotangent space as an abelian Lie
  algebra) on $F: \R^{2n} \to \R^n$, {\it i.e.} a Lie algebra
  homomorphism
  $\mathfrak{a}: \Gamma\left(T^*\R^n\right) =
  \Omega^1\left(\R^n\right) \to \mathfrak{X}\left(\R^{2n}\right)$
  with the property that, for all
  $\alpha \in \Omega^1\left(\R^n\right)$,
  $\mathfrak{a}\left(\alpha\right) \in \Gamma\left(\ker DF\right)$.
  Explicitly, if $\alpha \in \Omega^1(\R^n)$, then
  $\mathfrak{a}\left(\alpha\right) =
  \omega_{\mathrm{can}}^{-1}\left(F^*\alpha\right)$,
  {\it i.e.} the unique vector field on $\R^{2n}$ which, when
  contracted with $\omega$, equals $F^*\alpha$. To see the connection
  with the action of Lemma \ref{lemm:cihs_as_inf_act}, let
  $a = (a_1,\ldots, a_n)$ be the standard coordinates on $\R^n$ and
  write
  $\alpha = \sum\limits_{i=1}^n \alpha_i \DD a_i \in \Omega^1(\R^n)$.
  Then
  \begin{equation}
    \label{eq:8}
    \mathfrak{a}\left(\alpha\right)=
    \sum\limits_{i=1}^n\mathfrak{a}\left(\alpha_i \DD a_i\right) =
    \sum\limits_{i=1}^n \left(F^* \alpha_i\right)
    \mathfrak{a}\left(\DD a_i\right) = \sum\limits_{i=1}^n \left(F^*\alpha_i\right)
    \ham{f_i}.
  \end{equation}
  For any $\alpha \in \Omega^1(\R^n)$, let $\phi^t_{\alpha}$ denote
  the flow at time $t$ of $\mathfrak{a}\left(\alpha\right)$; observe
  that, whenever it is defined, $F \circ \phi^t_{\alpha} = F$. Just as
  in the case of actions of Lie algebras, such an action can be
  integrated to an action of a bundle of local abelian Lie groups,
  {\it i.e.} there exists an open neighborhood $W \subset T^*\R^n$ of
  the zero section such that the map
  \begin{equation}
    \label{eq:7}
    \begin{split}
      A : W \times_{\R^n} \R^{2n} &\to \R^{2n} \\
      \left(\alpha,p\right) &\mapsto \phi^{2\pi}_{\alpha}(p)
    \end{split}
  \end{equation}
  is an action\footnote{This means that all the standard axioms for
    actions are verified whenever the compositions are possible.}  of
  $W$ on $F: \R^{2n} \to \R^n$, where $\times_{\R^n}$ denotes fibered
  product over $\R^n$. This means that, for all $c \in \R^n$, there
  exists an open neighborhood $W_c \subset T_c^* \R^n$ of the origin
  and an action of the local abelian Lie group $W_c$ on
  $F^{-1}(c)$. Consider the map $\Psi_{\sigma} : W \to \R^{2n}$
  defined by
  $\Psi_{\sigma}(\alpha)
  :=A\left(\alpha,\sigma\left(\mathrm{pr}\left(\alpha\right)\right)\right)$,
  where $\mathrm{pr} : T^* \R^n \to \R^n$ is the natural projection;
  since near the zero section the differential $D \Psi_{\sigma}$ is
  invertible and by local freeness of the action of equation
  \eqref{eq:7}, it follows that $\Psi_{\sigma}$ is a diffeomorphism of
  some open neighborhood of the zero section in $W$ to some open
  neighborhood of $\sigma(\R^n) \subset \R^{2n}$. Moreover,
  $F \circ \Psi_{\sigma} = \mathrm{pr}$ by construction and
  $\Psi_{\sigma}^*\omega_{\mathrm{can}} = \omega_{\mathrm{can}} +
  \mathrm{pr}^* \sigma^* \omega_{\mathrm{can}}$
  (cf. \cite[Theorem 3.1]{daz_delz}). Therefore the integrable system
  $\left(\R^{2n},\omega_{\mathrm{can}},F\right)$ is strongly
  symplectically equivalent (locally near $0$) to
  $\left(T^*\R^{n},\omega_{\mathrm{can}} + \mathrm{pr}^* \beta,
    \mathrm{pr}\right)$,
  where $\beta = \sigma^* \omega_{\mathrm{can}}$, via a pair of the
  form $\left(\varphi,\mathrm{id}\right)$. Since $\R^n$ is
  contractible, $\beta = \DD \gamma$ for some 1-form $\gamma$; the
  section $-\gamma : \R^n \to T^*\R^n$ is Lagrangian for the
  symplectic form $\omega_{\mathrm{can}} + \mathrm{pr}^*\beta$.
  Repeating the above argument, we obtain that
  $\left(T^*\R^{n},\omega_{\mathrm{can}} + \mathrm{pr}^* \beta,
    \mathrm{pr}\right)$
  is strongly symplectically equivalent to
  $\left(T^*\R^{n},\omega_{\mathrm{can}}, \mathrm{pr}\right)$ via a
  pair of the form $\left(\varphi,\mathrm{id}\right)$.  Upon using the
  standard trivialization for $T^*\R^n$,
  $\left(T^*\R^{n},\omega_{\mathrm{can}}, \mathrm{pr}\right)$ can be
  identified with $\left(\R^{2n},\omega_{\mathrm{can}},
    \xi\right)$. This completes the proof.
\end{proof}

\begin{rema}\label{rmk:loc_coord}
  It is a useful exercise to unravel the above proof using local
  coordinates. If $F = \left(f_1,\ldots,f_n\right)$ and $m \in M$ is
  regular, then Theorem \ref{thm:dar_car} simply says that there exist
  smooth functions $\xi_1,\ldots,\xi_n$ defined locally near $m$ such
  that, $f_1-f_1(m),\ldots,f_n -f_n(m),\xi_1,\ldots,\xi_n$ are Darboux
  coordinates near $m$.
\end{rema}

A nice application of Theorem \ref{thm:dar_car} is the following
strengthening of Corollary \ref{cor:subalg}.

\begin{coro}\label{cor:comm_comm}
  Let $\left(M,\omega,F\right)$ be an integrable system and let
  $U \subset M$ be an open subset. The commutant $\C_F(U)$ is an
  abelian Poisson subalgebra of
  $\left(C^{\infty}(U),\left\{\cdot,\cdot\right\}|_U\right)$.
\end{coro}
\begin{proof}
  First we show that, if $U$ is a connected neighborhood of a regular
  point as in the statement of Theorem \ref{thm:dar_car}, then
  $\C_F(U)$ is abelian. Suppose that $U$ is such a neighborhood; using
  Corollary \ref{coro:stronger} and Theorem \ref{thm:dar_car}, it
  suffices to consider the case in which $U$ is an open neighborhood
  of the origin and
  $\left(M,\omega,F\right) = \left(\R^{2n},\omega_{\mathrm{can}},
    \xi\right)$
  (cf. Example \ref{exm:local_normal_form_regular_point}). The proof
  of Theorem \ref{thm:dar_car} implies that $\xi|_U$ has connected
  fibers. Then any element of $\mathcal{C}_{\xi}(U)$ is $\xi$-basic,
  {\it i.e.} of the form $\xi^*g$ for some smooth function $g$ defined
  on $\xi(U)$. This is because any element of $\mathcal{C}_\xi(U)$ is
  locally constant on the fibers of $\xi$, the fibers of $\xi$ are
  connected and $\xi$ is a submersion onto $\xi(U)$. The fact that the
  components of $\xi$ Poisson commute implies that, for any
  $g, h \in C^{\infty}\left(\xi(U)\right)$,
  $\left\{\xi^*g,\xi^*h\right\} = 0$. Thus, in this case,
  $\mathcal{C}_\xi(U)$ is abelian, as desired.

  In fact, the above argument implies that, if $U$ consists solely of
  regular points, then $\C_F(U)$ is abelian, for it suffices to
  restrict to open subsets as in the statement of Theorem
  \ref{thm:dar_car}. Finally, if $U$ contains singular points, observe
  that, by definition of integrable systems, there exists an open,
  dense subset $U' \subset U$ consisting of regular points. This
  reduces the problem to the previous case and, thus, completes the
  proof.
\end{proof}



Having established Theorem \ref{thm:dar_car}, we turn to the question
of describing the structure of an integrable system in a neighborhood
of a {\em compact} regular leaf. This is the content of the following
result, which is is usually associated with the names of Liouville,
Mineur and Arnol'd.

\begin{theo}[Action-angle variables]\label{thm:aa}
  Let $\left(M,\omega,F\right)$ be an integrable system with $n$
  degrees of freedom and suppose that $\Lambda_c \subset F^{-1}(c)$ is
  a compact regular leaf. Then there exist open neighborhoods
  $U \subset M$,
  $V \subset T^* \mathbb{T}^n \cong \R^n \times \mathbb{T}^n$ of
  $\Lambda_c$ and of the zero section respectively, saturated with
  respect to the maps $F$ and
  $\mathrm{pr} : \R^n \times \mathbb{T}^n \to \R^n$ respectively, such
  that the subsystem of $\left(M,\omega,F\right)$ relative to $U$ is
  strongly symplectically equivalent to the subsystem of
  $\left(\R^n \times \mathbb{T}^n, \omega_0, \xi\right)$ (cf. Example
  \ref{exm:local_normal_form_regular_point}) relative to $V$.
\end{theo}

An immediate consequence of Theorem \ref{thm:aa} is the following
result, which is usually stated as part of the existence of
action-angle variables.

\begin{coro}\label{cor:aa}
  Let $\left(M,\omega,F\right)$ be an integrable system with $n$
  degrees of freedom and suppose that $\Lambda_c \subset F^{-1}(c)$ is
  a compact regular leaf. Then $\Lambda_c \cong \mathbb{T}^n$ and
  there exists an $F$-saturated, open neighborhood $U \subset M$ such
  that $F|_U$ is locally trivial.
\end{coro}

After the pioneer work of Mineur~\cite{mineur-action-angle}, one can
find several proofs of slight variations of the statement of Theorem
\ref{thm:aa} in the literature (cf. \cite{dui,bates_snia},
\cite[Appendix A2]{hofer-zehnder}, \cite[Section 44]{gui_ste_book}
amongst others). In the proof presented below we use an argument to
reduce tideas from the problem to the one considered in
\cite[Section 44]{gui_ste_book}.

\begin{proof}
  The first step is to show that, by restricting to a suitable open
  neighborhood of $\Lambda_c$, it may be assumed, without loss of
  generality, that $F$ is a proper submersion with connected
  fibers. This is a consequence of the following differential
  topological fact. 
  
  \begin{claim}
    Let $F :M \to N$ be a smooth map and suppose that
    $\Lambda \subset F^{-1}(q)$ is a compact, connected component
    consisting solely of regular points. Then there exists an
    $F$-saturated open neighborhood $U \subset M$ of $\Lambda$ such
    that $F|_U$ is a proper submersion with connected fibers.
  \end{claim}


  Thus suppose that $F: M \to \R^n$ is a proper submersion whose
  fibers are connected; without loss of generality, it may be assumed
  to be onto. Henceforth, the argument is as in \cite[Section
  44]{gui_ste_book} and is outlined below for completeness. As in the
  proof of Theorem \ref{thm:dar_car}, consider the infinitesimal
  action $\mathfrak{a} : \Omega^1(\R^n) \to \mathfrak{X}(M)$ of
  $T^* \R^n \to \R^n$ on $F : M \to \R^n$ given by equation
  \eqref{eq:8}. Compactness of the fibers of $F$ imply that the flow
  $\phi^t_{\alpha}$ of $\mathfrak{a}\left(\alpha\right)$ exists for
  all $t \in \R$, where $\alpha \in \Omega^1(\R^n)$. Therefore, the
  infinitesimal action $\mathfrak{a}$ integrates to an action of the
  bundle of abelian Lie groups $T^* \R^n \to \R^n$ on $F : M \to \R^n$
  which is given by equation \eqref{eq:7}. (The only difference being
  that this integrated action is defined on the whole of $T^*\R^n$!)
  For each $c \in \R^n$, the abelian Lie group $T^*_c \R^n$ acts on
  $F^{-1}(c)$; for $p \in F^{-1}(c)$, let
  $A_p:=A(-,p) : T^*_c \R^n \to F^{-1}(c)$ be the smooth map induced
  by the action $A$. Since, for all $p \in F^{-1}(c)$, $D_0A_p$ is an
  isomorphism, connectedness of $F^{-1}(c)$ implies that the action of
  $T^*_c \R^n$ on $F^{-1}(c)$ is transitive (cf. \cite[Page
  351]{gui_ste_book}). Fix $c \in \R^n$; for $p \in F^{-1}(c)$,
  consider the isotropy at $p$ of the action of $T^*_c \R^n$ on
  $F^{-1}(c)$,
  \[
  \Sigma_p : = \left\{ \alpha \in T^*_c \R^n \mid
    \phi^{2\pi}_{\alpha}(p) = p \right\}.
  \]
  Since $T_c^*\R^n$ is abelian, the isotropy subgroups of any two
  points $p,p' \in F^{-1}(c)$ are {\em canonically} isomorphic; thus
  obtain a well-defined subgroup
  \[
  \Sigma_c := \left\{ \alpha \in T^*_c \R^n \mid \exists \, p \in
    F^{-1}(c) \,, \, \phi^{2\pi}_{\alpha}(p) = p \right\},
  \]
  and $F^{-1}(c) \cong T^*_c\R^n/\Sigma_c$. Since $T^*_c\R^n$ is
  abelian and has the same dimension as the compact submanifold
  $F^{-1}(c)$, it follows that $\Sigma_c \cong \Z^n$ and that
  $F^{-1}(c) \cong \mathbb{T}^n$. Set
  \begin{equation}
    \label{eq:9}
    \Sigma := \bigcup\limits_{c \in \R^n} \Sigma_c \subset T^* \R^n;
  \end{equation}
  it is a {\em smooth Lagrangian} submanifold and the projection
  $\Sigma \to \R^n$ is a {\em $\Z^n$-bundle}, {\it i.e.} it is a fiber
  bundle with fiber isomorphic to $\Z^n$ and whose structure group is
  $\mathrm{GL}(n;\Z)$. (It is important to observe that smoothness of
  $\Sigma$ is not trivial to prove, cf. \cite[Theorem
  44.4]{gui_ste_book} and references therein.) Since $\R^n$ is
  contractible, the map $F: M \to \R^n$ admits a globally defined
  smooth section $\sigma : \R^n \to M$. As in the proof of Theorem
  \ref{thm:dar_car}, consider the smooth map
  $\Psi_{\sigma} : T^*\R^n \to M$ induced by the above action $A$,
  {\it i.e.}
  $\Psi_{\sigma}(\alpha) = A\left(\alpha,
    \sigma\left(\mathrm{pr}(\alpha)\right)\right)$,
  where $\mathrm{pr} : T^*\R^n \to \R^n$ denotes the standard
  projection. This map descends to a diffeomorphism
  $\hat{\Psi}_{\sigma} : T^*\R^n/\Sigma \to M$ which sends the zero
  section to $\sigma$, since $\Sigma$ is precisely the isotropy of the
  action $A$. Since $\Sigma$ is Lagrangian, $T^*\R^n/\Sigma$ inherits
  a symplectic form $\omega_0$ and
  $\hat{\Psi}_{\sigma}^*\omega = \omega_0 +
  \mathrm{pr}^*\sigma^*\omega$
  (cf. \cite[Theorem 3.1]{daz_delz}). Thus the integrable system
  $(M,\omega,F)$ is strongly symplectically equivalent to
  $\left(T^*\R^n/\Sigma,\omega_0 + \mathrm{pr}^*\beta,
    \mathrm{pr}\right)$,
  where $\beta$ is a closed 2-form and, by abuse of notation,
  $\mathrm{pr} : T^*\R^n/\Sigma \to \R^n$ denotes the induced
  projection. In fact, arguing as in the proof of Theorem
  \ref{thm:dar_car}, we can show that
  $\left(T^*\R^n/\Sigma,\omega_0 + \mathrm{pr}^*\beta,
    \mathrm{pr}\right)$
  is strongly symplectically equivalent to
  $\left(T^*\R^n/\Sigma,\omega_0, \mathrm{pr}\right)$: the form
  $\beta$ is exact and any of its antiderivatives induces a Lagrangian
  section. 
  Since $\R^n$ is contractible, the $\Z^n$-bundle $\Sigma \to \R^n$ is
  trivializable. Fixing a trivialization allows to identify
  $\left(T^*\R^n/\Sigma,\omega_0, \mathrm{pr}\right)$ with the
  required integrable system
  $\left(\R^n \times \mathbb{T}^n, \omega_0, \xi\right)$, thus
  completing the proof.
\end{proof}

\begin{rema}\label{rmk:aa}
  Theorem \ref{thm:aa} can be interpreted as saying that, as a
  Lagrangian foliation, the only invariant of an integrable system
  $(M,\omega,F)$ in a neighborhood of a regular, compact connected
  component of a fiber of $F$ is the number of degrees of freedom.
  However, a closer look at the proof of Theorem \ref{thm:aa} yields
  {\em symplectic} invariants of the map $F$ as follows.  Observe that
  the {\em bundle of periods} $\Sigma \subset T^*\R^n$ being
  Lagrangian follows from the fact that any (local) section of
  $\Sigma \to \R^n$ is a {\em closed} form (cf. \cite[Section
  44]{gui_ste_book} and references therein). Let
  $\alpha_1,\ldots,\alpha_n$ denote a frame for $\Sigma \to \R^n$;
  $\alpha_1,\ldots,\alpha_n$ are symplectic invariants of $F$ itself
  since they are determined by periods of periodic trajectories of the
  initial system. Moreover, (locally) these are exact forms, so there
  exist functions $g_1,\ldots,g_n$ such that, for all $i=1,\ldots,n$,
  $\alpha_i = \DD g_i$. Observe that the composition
  $g \circ F : M \to \R^n$, where $g:=\left(g_1,\ldots,g_n\right)$ can
  be viewed as the moment map of an effective Hamiltonian
  $\mathbb{T}^n$-action, {\it i.e.} for each $i$,
  $\ham{\left(g \circ F\right)_i}$ has flow which is periodic with
  period 1. Finally, it is worth mentioning that there is an explicit
  formula for the functions $g_1,\ldots,g_n$. Let
  $\gamma_1,\ldots,\gamma_n: \R^n \to \mathrm{H}_1(\mathbb{T}^n;\Z)$
  be a smooth map associating to each $c \in \R^n$ a base of
  $\mathrm{H}_1\left(F^{-1}(c);\Z\right) \cong
  \mathrm{H}_1(\mathbb{T}^n;\Z)$
  (locally such a map always exists). Fix $c_0 \in \R^n$; then, in an
  $F$-saturated neighborhood $W \subset M$ of $F^{-1}(c_0)$, the
  symplectic form $\omega$ is exact, say equal to $\DD \sigma$ . Then,
  for $i=1,\ldots, n$,
  \begin{equation}
    \label{eq:12}
    g_i(c):= \frac{1}{2\pi} \int\limits_{\gamma_i(c)} \sigma,
  \end{equation}
  (cf. \cite{dui}).
\end{rema}

\begin{rema}\label{rmk:generalize}
  The advantage of the proofs of Theorems \ref{thm:dar_car} and
  \ref{thm:aa} as presented above is that they easily generalize to
  more general settings of Hamiltonian integrability (the so-called
  {\em non-commutative} case, cf. \cite{daz_delz}), and to more
  general geometric structures than symplectic forms ({\it e.g.}
  almost-symplectic, contact, Poisson structures,
  cf. \cite{jovanovic,sans_sepe,sal_sepe,flgv}).
\end{rema}

\subsection{The global counterpart of action-angle coordinates:
  integral affine structures}\label{sec:glob-count-acti}

While Theorem \ref{thm:aa} establishes the existence of {\em local}
action-angle variables, there are topological obstructions to the
existence of {\em global} ones, as first observed in
\cite{dui}. However, even if an integrable system does not admit
global action-angle variables, there is a globally defined geometric
structure, invariant under strong symplectic equivalence
(cf. Definition \ref{defn:strong_equiv}), which encodes all possible
local action variables, an observation which is also due to
Duistermaat in \cite{dui}. Throughout this section, fix an integrable
system $(M,\omega,F)$ all of whose fibers are compact and let
$\mathcal{L}$ denote its leaf space (cf. Definition
\ref{defn:leaf_space}).

\begin{defi}\label{defn:regular_leaf}
  The subset $\mathcal{L}_{\mathrm{reg}} \subset \mathcal{L}$
  corresponding to regular leaves is called the {\em regular leaf
    space} of $(M,\omega,F)$.
\end{defi}

The pair $\left(\mathcal{L},\mathcal{L}_{\mathrm{reg}}\right)$ is an
invariant of a $(M,\omega,F)$ in the following sense.

\begin{coro}\label{cor:invariant}
  If two integrable systems $(M_1,\omega_1,F_1)$, $(M_2,\omega_2,F_2)$
  are strongly symplectically equivalent, then their pairs of leaf and
  regular leaf spaces are homeomorphic as pairs.
\end{coro}

Under the assumption that all fibers are compact, regular leaf spaces
are well-behaved topologically (cf. \cite[Section 2.4]{moerd_mrcun}).

\begin{lemm}\label{lemma:reg_ls}
  If all fibers are compact, the regular leaf space
  $\mathcal{L}_{\mathrm{reg}}$ of $(M,\omega,F)$ is open in
  $\mathcal{L}$, is Hausdorff, locally compact and second countable.
\end{lemm}

In fact, Theorem \ref{thm:aa} can be used to endow
$\mathcal{L}_{\mathrm{reg}}$ with the structure of a smooth manifold
whose changes of coordinates are `rigid'.

\begin{defi}\label{defn:int_aff}
  An {\em integral affine structure} on a Hausdorff, second countable,
  locally compact topological space $N$ is a smooth atlas
  $\mathcal{A}=\left\{\left(W_i,\chi_i\right)\right\}$ such that, for
  all $i,j$ with $W_i \cap W_j \neq \emptyset$, the restriction of
  $\chi_j \circ \chi_i^{-1}:\chi_i\left(W_i \cap W_j\right) \subset
  \R^n \to \chi_j\left(W_i \cap W_j\right) \subset \R^n$
  to each connected component of $\chi_i\left(W_i \cap W_j\right)$ is
  (the restriction of) an element of the group
  $$ \mathrm{AGL}(n;\Z) := \mathrm{GL}(n;\Z) \ltimes \R^n. $$
  A pair $\left(N,\mathcal{A}\right)$ where $N$ and $\mathcal{A}$ are
  as above is said to be an {\em integral affine manifold}.
\end{defi}

\begin{coro}\label{cor:ia_reg_ls}
  Let $\left(M,\omega,F\right)$ be an integrable system all of whose
  fibers are compact. Then its regular leaf space
  $\mathcal{L}_{\mathrm{reg}}$ inherits an integral affine structure
  $\mathcal{A}$ uniquely defined by the property that
  $\chi : W \subset \mathcal{L}_{\mathrm{reg}} \to \R^n$ is an
  integral affine coordinate chart if and only if
  $q \circ \chi : \left(q^{-1}(W),\omega|_{q^{-1}(W)}\right) \to \R^n$
  is the moment map of an effective Hamiltonian $\mathbb{T}^n$-action,
  where $q : M \to \mathcal{L}$ is the quotient map.
\end{coro}

\begin{proof}
  Let $[p] \in \mathcal{L}_{\mathrm{reg}}$; this corresponds to a
  compact, regular leaf $\Lambda$. Unravelling the definitions,
  Theorem \ref{thm:aa} guarantees the existence of an open
  neighborhood $W \subset \mathcal{L}_{\mathrm{reg}}$ of $[p]$
  together with a map $\chi: W \to \R^n$ with the property that
  $q \circ \chi$ is the moment map of an effective Hamiltonian
  $\mathbb{T}^n$-action (cf. Remark \ref{rmk:aa}). This defines the
  atlas $\mathcal{A}$. To check that it induces an integral affine
  structure, it suffices to observe that the differential of the
  components of $\chi$ locally generate the bundle of periods; thus if
  $\chi_i,\chi_j$ are coordinate maps defined on a connected,
  non-empty open set, their differentials are related by an element of
  $\mathrm{GL}(n;\Z)$ since this is the structure group of the bundle
  of periods. To obtain the required result, integrate this
  equality. Checking that this integral affine structure is uniquely
  defined by the above property is left as an exercise to the reader.
\end{proof}

\begin{rema}\label{rmk:connected_fibers}
  It is useful to unravel the above constructions for an integrable
  system $\left(M,\omega,F\right)$ all of whose fibers are compact and
  connected. In this case, $\mathcal{L}_{\mathrm{reg}}$ can be
  identified with the open subset of $\R^n$ given by the intersection
  of $F(M)$ and the set of regular values of $F$. If
  $q :M \to \mathcal{L}$ denotes the quotient map, setting
  $M_{\mathrm{reg}}:= q^{-1}\left(\mathcal{L}_{\mathrm{reg}}\right)$,
  we have that the above open subset of $\R^n$ equals
  $F\left(M_{\mathrm{reg}}\right)$. The subsystem of
  $\left(M,\omega,F\right)$ relative to $M_{\mathrm{reg}}$ has compact
  and connected fibers and contains no singular point by
  construction. Therefore, proceeding as in the proof of Theorem
  \ref{thm:aa}, it is possible to associate to it a bundle of periods
  $\Sigma \to F\left(M_{\mathrm{reg}}\right)$. If
  $\DD g_1,\ldots, \DD g_n$ denotes a local frame for $\Sigma$, then
  the map $g = \left(g_1,\ldots,g_n\right)$ is an integral affine
  chart for the atlas $\mathcal{A}$ constructed as in Corollary
  \ref{cor:ia_reg_ls}.
\end{rema}


The integral affine manifold
$\left(\mathcal{L}_{\mathrm{reg}},\mathcal{A}\right)$ is an invariant
of integrable systems up to strong symplectic equivalence. Instead of
showing this fact in general, we prove it only for the family of
integrable systems with compact and connected fibers.

\begin{coro}\label{cor:ias_invt}
  For $i=1,2$, let $\left(M_i,\omega_i,F_i\right)$ be an integrable
  system all of whose fibers are compact and connected, and suppose
  that $\left(M_1,\omega_1,F_1\right)$ and
  $\left(M_2,\omega_2,F_2\right)$ are strongly symplectically
  equivalent via $\left(\phy,g\right)$. If, for $i=1,2$, $\Sigma_i$
  denotes the bundle of periods constructed as in Remark
  \ref{rmk:connected_fibers}, then $g^* \Sigma_2 = \Sigma_1$.
\end{coro}

\begin{demo}
  For $i=1,2$, let $M_{\mathrm{reg},i} \subset M_i$ denote the open
  subset of $M$ equal to
  $q_i^{-1}\left(\mathcal{L}_{\mathrm{reg},i}\right)$ as in Remark
  \ref{rmk:connected_fibers}. The strong symplectic equivalence
  $\left(\varphi,g\right)$ between $\left(M_1,\omega_1,F_1\right)$ and
  $\left(M_2,\omega_2,F_2\right)$ restricts to a strong symplectic
  equivalence between the subsystems of
  $\left(M_1,\omega_1,F_1\right)$ and $\left(M_2,\omega_2,F_2\right)$
  relative to $M_{\mathrm{reg},1}$ and $M_{\mathrm{reg},2}$
  respectively. By abuse of notation, denote this induced strong
  symplectic equivalence also by $\left(\varphi,g\right)$. Since
  $\left(\varphi,g\right)$ is a strong symplectic equivalence,
  $g^*\Sigma_2 \subset \Sigma_1$, for a strong symplectic equivalence
  intertwines the actions of $T^* \R^n$ considered in the proof of
  Theorem \ref{thm:aa}. Reversing the roles of $\Sigma_1$ and
  $\Sigma_2$ and using the inverse strong symplectic equivalence
  $\left(\phy^{-1},g^{-1}\right)$, we have that
  $\left(g^{-1}\right)^* \Sigma_1 \subset \Sigma_2$. Thus
  \[
  \Sigma_1 = g^*\left(g^{-1}\right)^* \Sigma_1 \subset g^* \Sigma_2
  \subset \Sigma_1,
  \]
  which implies the desired equality.
\end{demo}


\begin{rema}\label{rmk:int_aff_elsewhere}
  Integral affine structures appear naturally in other problems
  related to Poisson geometry: for instance, in the study of Poisson
  manifolds of compact types (cf.
  \cite{pmct2,pmct1,zung_proper}). 
\end{rema}



\section{Almost-toric singular fibers}\label{sec:almost-toric-sing}
Throughout this section, any integrable system is assumed to have
compact fibers unless otherwise stated. The aim of this section is to
study a family of singular leaves/fibers of integrable systems; the
starting point is the following result, stated without proof.

\begin{lemm}\label{lemma:cihs_ham_act}
  Given an integrable system
  $\left(M,\omega, F=(f_1,\ldots,f_n)\right)$ with compact fibers, for
  each $i=1,\ldots, n$, the Hamiltonian vector field $\ham{f_i}$ is
  complete. In particular, the map
  \begin{equation*}
    \begin{split}
      \R^n \times M &\to M \\
      \left((t_1,\ldots,t_n),p\right) &\mapsto \phi^{t_1}_1\circ
      \ldots \circ \phi^{t_n}_n(p),
    \end{split}
  \end{equation*}
  where, for each $i=1,\ldots, n$, $\phi^s_i$ is the flow of
  $\ham{f_i}$ at time $s$, defines a Hamiltonian $\R^n$-action on
  $(M,\omega)$ one of whose moment maps is $F$.
\end{lemm}

Lemma \ref{lemma:cihs_ham_act} motivates referring to $F$ as the {\em
  moment map} of the integrable system $\left(M,\omega,F\right)$. With
the above Hamiltonian $\R^n$-action at hand, it is possible to extend
the notions introduced in Definition \ref{defn:regular}.

\begin{defi}\label{defn:singular}
  Let $\left(M,\omega,F\right)$ be a completely integrable system.
  \begin{itemize} [leftmargin=*]
  \item The {\em rank} of a point $p \in M$ is defined to be
    $\mathrm{rk}\, p:= \mathrm{rk}\, D_pF$.
  \item If $p \in M$ is singular (respectively regular), the
    $\R^n$-orbit $\mathcal{O}_p$ through $p$ is said to be {\em
      singular} (respectively {\em regular}).
  \end{itemize}
\end{defi}


\begin{coro}\label{coro:sing}
  Let $\left(M,\omega,F\right)$ be a completely integrable system. The
  rank of a point $p \in M$ equals the dimension of its orbit
  $\mathcal{O}_p$ which is diffeomorphic to
  $\mathbb{T}^{c(p)} \times \R^{\mathrm{rk}\, p - c(p)}$ for some
  non-negative integer $c(p)$.
\end{coro}


When studying the topology or (symplectic) geometry near a singular
point of an integrable system, there are several `scales' that can be
adopted, namely near the point, near the orbit through the point, or
near the leaf containing the point. The first two `scales' are often
referred to in the literature as {\em local}, while the last one is
known as {\em semi-local}. The first difference between the regular
and singular cases is that in the latter it is not necessarily true
that an orbit need coincide with a leaf. In full generality, there are
no results characterizing neighborhoods of singular
points/orbits/leaves: intuitively, this is because, given a completely
integrable system, the underlying Hamiltonian $\R^n$-action is not
necessarily proper. Thus we need to introduce a `suitable' restriction
on the types of singular points/orbits that allows for a geometric
treatment. Here, `suitable' means that it appears naturally in many
physical problems while also being generic in an appropriate
mathematical sense.

\subsection{Non-degeneracy and Eliasson's
  theorem}\label{sec:defin-char-non}


The property of singular points/orbits introduced below can be thought
of as a `symplectic Morse-Bott condition' (cf. \cite[Exercise
2]{lf-p-s-vn_exercises}).
In what follows, the definition of non-degenerate singular points is
split in two cases: first, fixed points (\textit{i.e.}  of rank 0) are
dealt with and then the general case is considered. Fix an integrable
system $\left(M,\omega,F = (f_1,\ldots,f_n) \right)$ and suppose that
$p \in M$ is a singular point of rank 0. This means that, for all
$i = 1,\ldots,n$, $\ham{f_i}(p) = 0$. For
$t =(t_1,\ldots,t_n) \in \R^n$, set
$\phi^{t}_{F}:=\phi^{t_1}_1 \circ \ldots \circ \phi^{t_n}_n$. The
condition on $p$ having rank equal to 0 means that, for all
$t \in \R^n$, $\phi^t_{F}(p) = p$. Thus, for all $t \in \R^n$,
$D_p\phi^t_{F} : T_p M \to T_pM$ is an isomorphism; moreover, since
$\phi^{t}_{F}$ is a symplectomorphism of $(M,\omega)$, it follows that
$D_p \phi^t_{F}$ is a linear symplectomorphism, {\it i.e.}
$D_p \phi^t_{F} \in \mathrm{Sp}\left(T_pM,\omega_p\right)$. By linear
algebra,
$\mathrm{Sp}\left(T_pM,\omega_p\right) =
\mathrm{Sp}\left(\R^{2n},\omega_{\mathrm{can}}\right)$,
where $\omega_{\mathrm{can}}$ denotes the canonical symplectic
form. Thus we obtain a Lie group morphism
$ \R^n \to \mathrm{Sp}\left(\R^{2n},\omega_{\mathrm{can}} \right)$,
which induces a Lie algebra morphism $\R^n \to \mathfrak{sp}(2n;\R) $
whose image is henceforth denoted by $\mathfrak{h}_p$.

\begin{defi}\label{defn:non-deg_fixed_point}
  A singular point $p \in M$ of rank 0 is said to be {\em
    non-degenerate} if $\mathfrak{h}_p$ is a Cartan subalgebra of
  $\mathfrak{sp}(2n;\R)$.
\end{defi}

\begin{rema}\label{rmk:cartan}
  By definition, $\mathfrak{h}_p \subset \mathfrak{sp}(2n;\R)$ is an
  abelian subalgebra. It is Cartan if it has maximal dimension equal
  to $n$ and if it is {\em self-normalising}, {\it i.e.} if
  $\mathfrak{h}_p = \left\{A \in \mathfrak{sp}(2n;\R) \mid \forall\, B
    \in \mathfrak{h}_p\quad [A,B] \in \mathfrak{h}_p \right\}. $
\end{rema}

\begin{rema}\label{label:linearisation}
  The above construction can be viewed equivalently as follows. Since
  the Hamiltonian vector fields $\ham{f_1}(p),\ldots,\ham{f_n}(p)$ all
  vanish, it is possible to consider their linearisations at $p$,
  $\ham{f_1}^{\mathrm{lin}}(p),\ldots,\ham{f_n}^{\mathrm{lin}}(p) \in
  \mathfrak{gl}\left(T_pM\right)$.
  These linear operators pairwise commute, since for all $i,j$,
  $[\ham{f_i},\ham{f_j}]= 0$, and are, in fact, symplectic
  . Therefore, obtain a representation
  $\R^n \to \mathfrak{sp}\left(T_pM,\omega_p\right) =
  \mathfrak{sp}\left(2n,\R\right)$,
  which is precisely the above homomorphism of Lie algebras.
\end{rema}

To deal with the general case, suppose that $p \in M$ is a singular
point of rank $k$. Without loss of generality, it may be assumed that
for all $i=1,\ldots,n-k$, $\DD_p f_i = 0$, so that
$\DD_p f_{n-k+1} \wedge \ldots \wedge \DD_p f_n \neq 0$ . Let
$\mathcal{O}_p \subset M$ denote the $\R^n$-orbit through $p$. In
analogy with the argument in the rank 0 case, get a Lie algebra
morphism $\R^{n-k} \to \mathfrak{sp}(2n;\R)$ whose image is denoted by
$\tilde{\mathfrak{h}}_p$. However, it can be shown that if
$A \in \tilde{\mathfrak{h}}_p$, then
$T_p \mathcal{O}_p \subset \ker A$ and that $\mathrm{im}\, A$ is
contained in the symplectic orthogonal
$\left(T_p \mathcal{O}_p\right)^{\Omega}$. Therefore, obtain a Lie
algebra homomorphism
$\R^{n-k} \to \mathfrak{sp}\left(2(n-k);\R\right)$, whose image is
denoted by $\mathfrak{h}_p$.

\begin{defi}\label{defn:non-deg}
  A singular point $p \in M$ of rank $k$ is said to be {\em
    non-degenerate} if $\mathfrak{h}_p$ is a Cartan subalgebra of
  $\mathfrak{sp}\left(2(n-k);\R\right)$.
\end{defi}

The above notion of non-degeneracy is {\em infinitesimal} (or {\em
  linear}), which makes it (theoretically) easy to check. Moreover,
Cartan subalgebras of $\mathfrak{sp}(2n;\R)$ have been classified up
to conjugation in \cite{williamson,williamson40}; this result is
recalled below without proof. In the statement of Theorem
\ref{thm:williamson}, the isomorphism of Lie algebras between
$\mathfrak{sp}(2n;\R)$ and $\mathrm{Sym}(2n;\R)$ of symmetric bilinear
form on the canonical symplectic vector space
$(\R^{2n},\omega_{\mathrm{can}})$ is used (cf.~\cite[Exercise
13]{lf-p-s-vn_exercises}).

\begin{theo}[Williamson, \cite{williamson}]\label{thm:williamson}
  \mbox{}
  \begin{itemize}[leftmargin=*]
  \item Let $\mathfrak{h} \subset \mathrm{Sym}(2n;\R)$ be a Cartan
    subalgebra. Then there exists canonical coordinates $x_j,y_j$ for
    $\R^{2n}$, a triple $(k_e,k_h,k_{ff}) \in \mathbb{Z}^3_{\geq 0}$
    with $k_e + k_h + 2k_{ff} = n$, and a basis $f_1,\ldots,f_n$ of
    $\mathfrak{h}$ such that
    \begin{equation}
      \label{eq:5}
      f_i =
      \begin{cases}
        \frac{x^2_i + y^2_i}{2} & \text{ if } i=1,\ldots,k_e \\
        x_iy_i & \text{ if } i=k_e+ 1,\ldots,k_e+k_h
      \end{cases}
    \end{equation}
    and if $i=k_e+ k_h+ 2l-1$, where $l = 1,\ldots, k_{ff}$, then
    $f_i = x_i y_{i+1} - x_{i+1}y_i $ and
    $f_{i+1} = x_iy_i + x_{i+1}y_{i+1} $.
  \item Two Cartan subalgebras
    $\mathfrak{h}, \mathfrak{h}' \subset \mathrm{Sym}(2n;\R)$ are
    conjugate if and only if their corresponding triples are equal.
  \end{itemize}
\end{theo}

\begin{defi}\label{defn:e_h_ff}
  Elements of the basis of a Cartan subalgebra
  $\mathfrak{h} \subset \mathrm{Sym}(2n;\R)$ as in Theorem
  \ref{thm:williamson} are said to be {\em elliptic}, {\em hyperbolic}
  or of {\em focus-focus} type according to whether they are of the
  form $\frac{x^2_i + y^2_i}{2}$, $x_iy_i$ or a pair
  $ x_i y_{i+1} - x_{i+1}y_i$, $x_iy_i + x_{i+1}y_{i+1}$ respectively.
\end{defi}
One way of phrasing Theorem \ref{thm:williamson} is that a Cartan
subalgebra of $\mathfrak{sp}(2n;\R)$ is completely determined by a
triple $(k_e,k_h,k_{ff}) \in \mathbb{Z}^3_{\geq 0}$ satisfying
$k_e + k_h+k_{ff} = n$, where $k_e$ (respectively $k_h,k_{ff}$)
denotes the number of elliptic components (respectively hyperbolic
components and focus-focus pairs) appearing in the basis constructed
above.

\begin{defi}[Zung, \cite{zung-st1}]\label{defn:williamson_type}
  Let $p \in M$ be a non-degenerate singular point of rank
  $0 \leq k < n$. Its {\em Williamson type} is a quadruple
  $(k,k_e,k_h,k_{ff}) \in \mathbb{Z}^4_{\geq 0}$ satisfying
  $k + k_e +k_h +2k_{ff} = n$, where $(k_e,k_h,k_{ff})$ is the triple
  associated to the Cartan subalgebra
  $\mathfrak{h}_p \subset \mathfrak{sp}\left(2(n-k);\R\right)$.
\end{defi}

The above notion is invariant along orbits, {\it i.e.} if $p$ is
non-degenerate, then all points on $\mathcal{O}_p$ are non-degenerate
of the same Williamson type (cf. \cite[Exercise
15]{lf-p-s-vn_exercises}).
Therefore, it makes sense to talk about the Williamson type of a
non-degenerate {\em orbit}.

\begin{rema}\label{rmk:Wiliamson_fibre}
  In general, it does not make sense to talk about the Williamson type
  of a leaf. For instance, a singular leaf may contain a regular orbit
  (cf. Remark \ref{rmk:local_model_will}). Moreover, the distinct
  singular orbits lying on a singular leaf need not all have the same
  Williamson type. However, for a fixed singular leaf, all singular
  points whose rank is minimal have equal Williamson type
  (cf. \cite[Proposition 3.6]{zung-st1}).
\end{rema}

In fact, the Williamson type of a non-degenerate point is an invariant
under taking (strong) symplectic equivalences. This is the content of
the following well-known result, stated below without proof (cf. for
instance~\cite[Theorem 3.1.24]{wang}; in fact, the result is stronger,
for the Williamson type is actually a topological invariant of the
foliation, see~\cite[Theorems 6.1, 7.3]{zung-st1}).
\begin{theo}\label{thm:will_preserved}
  Local symplectic equivalences preserve non-degenerate points and
  their Williamson types.
\end{theo}

Any integrable system can be {\em linearized} in a neighborhood of a
non-degenerate singular point: this is the content of Theorem
\ref{thm:Eliasson}. Let $\left(M,\omega,F\right)$ be an integrable
system and suppose that $p \in M$ is a non-degenerate singular point
of Williamson type $(k,k_e,k_h,k_f)$. The following definition
associates to such a quadruple a `model' integrable system.

\begin{defi}\label{defn:local_model}
  Given a quadruple
  $\mathbf{k}=(k,k_e,k_h,k_{ff}) \in \mathbb{Z}^4_{\geq 0}$ satisfying
  $k + k_e+k_h + 2k_{ff} = n$, the {\em local model of a singular
    point of Williamson type $\mathbf{k}$} is the integrable system
  $\left(\R^{2n},\omega_{\mathrm{can}}, Q_{\mathbf{k}} =
    (q_1,\ldots,q_n)\right)$,
  where
  $\omega_{\mathrm{can}}=\sum\limits_{i=1}^n \DD x_i \wedge \DD y_i$
  is the canonical symplectic form, and
  \begin{equation}
    \label{eq:2}
    q_i(x_1,y_1,\ldots,x_n,y_n) =
    \begin{cases}
      y_i & \text{ if } i=1,\ldots,k, \\
      \frac{x^2_i + y^2_i}{2} & \text{ if } i=k+1,\ldots,k+k_e, \\
      x_iy_i & \text{ if } i=k+k_e+1,\ldots,k+k_e+k_h,
    \end{cases}
  \end{equation}
  and the remaining ones are focus-focus pairs in the coordinates
  $(x_i,y_i, x_{i+1},y_{i+1})$, for $i=k+k_e+2l-1$, where
  $l =1,\ldots,k_{ff}$.
\end{defi}

For any local model of a singular point of Williamson type
$\mathbf{k}$, it is easily checked that all singular points are
non-degenerate and that the origin has Williamson type $\mathbf{k}$
(cf. \cite[Exercise 7]{lf-p-s-vn_exercises}).

\begin{rema}\label{rmk:local_model_will}
  Given a local model of Williamson type
  $\mathbf{k}=(k,k_e,k_h,k_{ff})$ with $k+k_e+k_h+2k_{ff} = n$, the
  following hold:
  \begin{itemize}[leftmargin=*]
  \item There is a Hamiltonian $\R^n$-action on
    $\left(\R^{2n},\omega_{\mathrm{can}}\right)$ one of whose moment
    maps is $Q_{\mathbf{k}}$.
  \item The fibres of $Q_{\mathbf{k}}$ are not connected if and only
    if $k_h >0$.
  \item There exist leaves which contain more than one orbit of the
    above Hamiltonian $\R^n$-action if and only if $k_h + k_{ff} > 0$.
  \end{itemize}
\end{rema}

A natural question to ask is whether a sufficiently small
neighbourhood of a non-degenerate singular point of Williamson type
$\mathbf{k}$ is equivalent to the local model of Williamson type
$\mathbf{k}$. This is the content of the following result.

\begin{theo}[Eliasson, \cite{eliasson_thesis}]\label{thm:Eliasson}
  Suppose that $\left(M,\omega,F\right) $ is an integrable system and
  let $p \in M$ be a non-degenerate singular point of Williamson type
  $\mathbf{k} = (k,k_e,k_h,k_{ff})$. Then there exist open
  neighbourhoods $U \subset M$ of $p$ and $V \subset \R^{2n}$ of the
  origin such that the subsystems of $\left(M,\omega,F\right) $ and of
  $\left(\R^{2n},\omega_{\mathrm{can}},Q_{\mathbf{k}}\right)$ relative
  to $U$ and $V$ respectively are symplectically equivalent. If
  $k_h =0$, the above subsystems are strongly symplectically
  equivalent.
\end{theo}

\begin{rema}\label{rmk:history}
  While Theorem \ref{thm:Eliasson} is generally referred to as
  Eliasson's theorem, it is worthwhile mentioning that it was first
  proven by Vey~\cite{vey} in the analytic case. We are not aware of a
  reference in the literature
  that contains a complete, self-contained proof of the $\Cinf$ case
  (cf. \cite{vey,russmann,cdvv,eliasson_thesis,duf_mol,eliasson,
    cdv_san,san_wacheux,chaperon,
    miranda_san,miranda_thesis,san-panoramas,wang} amongst others for
  various partial results and attempts to understand the general
  picture).
\end{rema}



\begin{rema}\label{rmk:hyperbolic_Eliasson}
  In the real analytic setting, Theorem \ref{thm:Eliasson} can be
  strengthened to make the above subsystems strongly symplectically
  equivalence (cf. \cite{russmann} for the case of two degrees of
  freedom and \cite{vey} for the general case). The issue with the
  smooth case is the presence of flat functions, which, when
  $k_h \neq 0$, may prevent strong equivalence. To illustrate the
  situation if $k_h \neq 0$, consider the following two results. In
  the case of one degree of freedom, strong symplectic equivalence can
  be attained {\it even} if $k_h =1$ (cf. \cite{cdvv}); on the other
  hand, starting with $n \geq 2$ degrees of freedom, there are
  counterexamples to strong symplectic equivalence
  (cf. \cite[Appendix]{chaperon}).
\end{rema}

While singular orbits with $k_h >0$ arise naturally in many
mathematical and physical problems ({\it e.g.} the height function of
a torus, the mathematical pendulum), in what follows, only singular
orbits with $k_h = 0$ are considered. Before restricting further to
the case of 2 degrees of freedom (cf. Section
\ref{sec:almost-toric-systems}), we mention a useful extension of
Theorem \ref{thm:Eliasson} to the case of {\em compact} orbits whose
Williamson type is of the form $(k,k_e,0,0)$, whose proof can be found
in \cite{duf_mol}. Such orbits are henceforth referred to as {\em
  purely elliptic}. In analogy with Definition \ref{defn:local_model},
we introduce the following local model in a neighborhood of such a
compact orbit\footnote{It is possible to introduce a local model for a
  compact orbit of {\em any} Williamson type
  (cf. \cite{miranda_zung}), but this is beyond the scope of these
  notes.}.

\begin{defi}\label{defn:local_model_cpct}
  Given a quadruple $\mathbf{k}=(k,k_e,0,0) \in \mathbb{Z}^4_{\geq 0}$
  satisfying $k + k_e = n$, the {\em local model of a compact singular
    orbit of Williamson type $(k,k_e,0,0)$} is the integrable system
  \[
  \left(T^* \mathbb{T}^k \times \R^{2(n-k)},\omega = \omega_0 \oplus
    \omega_{\mathrm{can}}, Q_{\mathbf{k}} = (q_1,\ldots,q_n)\right),
  \]
  where
  $\Omega_{\mathrm{can}} = \sum\limits_{j=1}^k \DD \xi_i \wedge \DD
  \theta_i$,
  $\omega_{\mathrm{can}}=\sum\limits_{i=1}^{n-k} \DD x_i \wedge \DD
  y_i$ are the canonical symplectic forms, and
  \begin{equation}
    \label{eq:3}
    q_i(\theta_1,\ldots,\theta_k,\xi_1,\ldots,\xi_k,x_1,\ldots, x_{n-k},y_1,\ldots,y_{n-k}) =
    \begin{cases}
      \xi_i & \text{ if } i=1,\ldots,k, \\
      \frac{x^2_i + y^2_i}{2} & \text{ if } i=k+1,\ldots,n.
    \end{cases}
  \end{equation}
\end{defi}

\begin{theo}[Dufour and Molino, \cite{duf_mol}]\label{thm:toric}
  Suppose that $(M,\omega,F)$ is an integrable system and that
  $\mathcal{O} \subset M$ is a non-degenerate compact orbit of
  Williamson type $\left(k,k_e,0,0\right)$, where $k+k_e = n$. Then
  there exist open neighborhoods $U \subset M$,
  $V \subset T^* \mathbb{T}^k \times \R^{2(n-k)}$ of $\mathcal{O}$ and
  of $\mathbb{T}^k \times \{0\}$ respectively, such that the
  subsystems of $(M,\omega,F)$ and of
  \[
  \left(T^* \mathbb{T}^k \times \R^{2(n-k)},\,\omega = \omega_0 \oplus
    \omega_{\mathrm{can}},\, Q_{\mathbf{k}} = (q_1,\ldots,q_n)\right)
  \]
  with respect to $U$ and $V$ respectively, are strongly
  symplectically equivalent. Moreover, $V$ can be taken to be
  $Q_{\mathbf{k}}$-saturated.
\end{theo}

\begin{rema}\label{rmk:purely_elliptic}
  Compact purely elliptic orbits are leaves of the system; this should
  be compared with the regular case (cf. Theorem \ref{thm:aa}), and
  contrasted with the general non-degenerate case (cf. Remark
  \ref{rmk:local_model_will} and Theorem
  \ref{thm:Eliasson}). Furthermore, Theorem \ref{thm:toric} implies
  that, locally near any compact purely elliptic orbit, the
  Hamiltonian $\R^n$-action descends to a Hamiltonian
  $\mathbb{T}^n$-action, which, again, is reminiscent of the regular
  case (cf. Remark \ref{rmk:aa}). In fact, it is possible to extend
  the notion of regular leaf space as in Definition
  \ref{defn:regular_leaf} to include compact, purely elliptic orbits;
  this gives rise to a subset of the leaf space
  $\mathcal{L}_{\mathrm{lt}}$ which, in \cite{hsss}, is called the
  {\em locally toric} leaf space. 
  Like the regular leaf space, it inherits an integral affine
  structure which, in fact, extends that on the regular leaf space;
  however, unlike the regular leaf space, in general, this structure
  makes it into a smooth manifold {\em with corners}. (This should be
  compared with the structure of orbit spaces of symplectic toric
  manifolds, cf. \cite{kar_ler}.)
\end{rema}

\subsection{Almost-toric systems}\label{sec:almost-toric-systems}
Henceforth we restrict our attention to integrable systems with two
degrees of freedom all of whose singular orbits are non-degenerate,
unless otherwise stated. Motivated by notions which first appeared in
\cite{symington,san_polygon}, we introduce the following family of
integrable systems.

\begin{defi}\label{defn:at}
  An integrable system $\left(M,\omega,F\right)$ with two degrees of
  freedom and with compact fibers is said to be {\em almost-toric} if
  all of its singular orbits are non-degenerate without hyperbolic
  blocks, {\it i.e.} for any such orbit $k_h = 0$.
\end{defi}

The Williamson type of a singular orbit of an almost-toric system is
very constrained, for it can be one of three types: $(0,2,0,0)$,
$(1,1,0,0)$ or $(0,0,0,1)$. The first two are purely elliptic and are
known as {\em elliptic-elliptic} and {\em elliptic-regular} orbits,
while the last is a {\em focus-focus} point. Furthermore, the absence
of hyperbolic blocks has important consequences when describing
neighborhoods of singular orbits. First, Theorem \ref{thm:Eliasson}
implies that a neighborhood of any singular point of an almost-toric
system is strongly symplectically equivalent to the corresponding
local model. Second, together with compactness, it ensures that all
singular orbits are compact. (For a more general statement and a
proof, cf.  \cite[Proposition 3.5]{zung-st1}).

\begin{coro}\label{cor:sing_orb_compact}
  Any singular orbit of an almost-toric system is compact.
\end{coro}

\begin{rema}\label{rmk:no_true}
  Corollary \ref{cor:sing_orb_compact} does not hold if hyperbolic
  blocks are allowed. Consider, for instance,
  $\left(S^2 \times \mathbb{T}^2, \omega, F = (f_1,f_2)\right)$, where
  $\omega = \omega_{S^2} \oplus \omega_{\mathbb{T}^2}$ is the sum of
  the standard symplectic forms on $S^2$ and $\mathbb{T}^2$
  respectively, and $f_1 :S^2 \times \mathbb{T}^2 \to \R$ is the
  pullback of the height function on $S^2$, while
  $f_2 : S^2 \times \mathbb{T}^2 \to \R$ is the pullback of a Morse
  function on $\mathbb{T}^2$ which possesses a point of index 0, two
  points of index 1 and a point of index 2. It can be checked directly
  that all singular orbits of the integrable system
  $\left(S^2 \times \mathbb{T}^2, \omega, F = (f_1,f_2)\right)$ are
  non-degenerate, but there are elliptic-regular orbits which are not
  compact.
\end{rema}

Corollary \ref{cor:sing_orb_compact} implies that Theorem
\ref{thm:toric} can be used to provide a local normal form for
neighborhoods of elliptic-regular orbits. In particular, this also
shows that a purely elliptic orbit of an almost-toric system is a
leaf. As a consequence, we have the following result.

\begin{coro}\label{cor:ff}
  If $\left(M,\omega,F\right)$ is an almost-toric system and $ p\in M$
  is a focus-focus point, then all singular points in the leaf
  containing $p$ are also focus-focus.
\end{coro}

Corollary \ref{cor:ff} motivates the introduction of the following
notion.

\begin{defi}\label{defn:ff_fiber}
  A {\em focus-focus} leaf of an almost-toric system is a leaf
  containing a focus-focus point.
\end{defi}

When the fibers of an almost-toric system are connected (as in the
case of semi-toric systems, cf. Theorem \ref{thm:connected}), we talk
about focus-focus {\em fibers}. In what follows, we describe the
topology of focus-focus leaves in almost-toric systems
(cf. \cite[Section 9.8.1]{bolsinov-fomenko-book} and \cite[Section
6.3]{san_bohr} for proofs). The starting point is the following
topological statement, which hinges on compactness of the leaf.


\begin{lemm}\label{lemma:finite}
  Any focus-focus leaf of an almost-toric system contains finitely
  many focus-focus points.
\end{lemm}


\begin{defi}\label{defn:mult}
  The {\em multiplicity} of a focus-focus leaf $\Lambda$ in an
  almost-toric system $\left(M,\omega,F\right)$ is the number of
  focus-focus points contained in $\Lambda$.
\end{defi}

Thinking of a focus-focus leaf as a point in the leaf space of an
almost-toric system, its multiplicity defines a natural number that
can be attached to this point and is invariant under symplectic
equivalence of neighborhoods of the leaves.

To prove further properties about focus-focus leaves in an
almost-toric system, we recall the local model for a focus-focus
point.  The integrable system under consideration is
$\left(\R^4, \omega_{\mathrm{can}},(q_1,q_2)\right)$, where
$\omega_{\mathrm{can}} = \DD x_1 \wedge \DD y_1 + \DD x_2 \wedge \DD
y_2$
and $q_1 = x_1y_2 - x_2y_1$, $q_2 = x_1y_1+x_2y_2$, and the
focus-focus point is the origin. The following result, whose proof is
left to reader, summarizes some important properties of this model
system.

\begin{prop}\label{prop:ff}
  Given the above integrable system
  $\left(\R^4, \omega_{\mathrm{can}},(q_1,q_2)\right)$,
  \begin{enumerate}[label=(\arabic*), ref = (\arabic*), leftmargin=*]
  \item \label{item:20} the flow of $\ham{q_1}$ is periodic. In
    particular, the function $q_1$ can be thought as the moment map of
    an effective Hamiltonian $S^1$-action. The $S^1$-action is, up to
    sign, unique;
  \item \label{item:21} the Hamiltonian $\R^2$-action descends to a
    Hamiltonian $S^1 \times \R$-action;
  \item \label{item:22} any open neighborhood of the origin contains a
    smaller open neighborhood which is saturated with respect to the
    above $S^1$-action, {\it i.e.} the latter contains whole orbits of
    the $S^1$-action;
  \item \label{item:23} the singular fiber $(q_1,q_2)^{-1}(0,0)$
    consists of the union of two Lagrangian disks intersecting
    transversally. These disks are the stable and unstable manifold
    for the flow of $\ham{q_2}$ restricted to $(q_1,q_2)^{-1}(0,0)$.
  \end{enumerate}
\end{prop}

Using Theorem \ref{thm:Eliasson}, Proposition \ref{prop:ff} provides a
description of the topology of a focus-focus leaf near a focus-focus
point, as well as giving some information about the $S^1 \times \R$
action defined along the leaf. (In fact, it also provides information
about the Hamiltonian action in a neighborhood of the leaf,
cf. \cite[Section 9.8.1]{bolsinov-fomenko-book}.) This
characterization can be used to infer further information about
focus-focus leaves in almost-toric systems.

\begin{coro}\label{cor:reg_ff}
  Any focus-focus leaf in an almost-toric system contains at least one
  regular orbit of the Hamiltonian $\R^2$-action. Moreover, each such
  orbit is diffeomorphic to $S^1 \times \R$.
\end{coro}

Finally, we can achieve a complete description of the topology of
focus-focus fibers in almost-toric systems.

\begin{theo}\label{theo:topology_ff}
  A focus-focus leaf of multiplicity $r \geq 1$ in an almost-toric
  system is homeomorphic to a torus with $r$ pinches.
\end{theo}

In fact, it is possible to give a `smooth' description of focus-focus
leaves, using the full strength of Proposition \ref{prop:ff}.

\begin{coro}\label{cor:smooth_ff}
  Let $\Lambda$ be a focus-focus leaf of multiplicity $r \geq 1$ in an
  almost-toric system. If $r =1$, $\Lambda$ is given by an immersed
  Lagrangian sphere with a single double point. If $r \geq 2$,
  $\Lambda$ is given by a chain of $r$ Lagrangian spheres with the any
  two of which either have empty intersection, or intersect
  transversally in a single point, or are equal.
\end{coro}
\begin{proof}
  The result follows from Proposition \ref{prop:ff}, which gives a
  {\em smooth} characterization of focus-focus leaves near a
  focus-focus point, since, locally, these are given by the
  transversal intersection of Lagrangian disks.
\end{proof}

\begin{rema}\label{rmk:obst}
  Corollary \ref{cor:smooth_ff} can be used to obtain topological {\em
    obstructions} for a symplectic manifold to support an almost-toric
  system with a focus-focus fiber with multiplicity $r \geq 2$
  (cf. \cite{smirnov}). 
\end{rema}

\subsection{Neighborhoods of focus-focus fibers}
\label{sec:monodr-focus-focus}

Theorem \ref{theo:topology_ff} and Corollary \ref{cor:smooth_ff}
characterize focus-focus leaves in almost-toric systems, showing that
they are determined by their multiplicities. The next natural question
is to study the (symplectic) geometry of sufficiently small saturated
neighborhoods of focus-focus leaves. In general, the multiplicity of a
focus-focus leaf determines a sufficiently small neighborhood thereof
only up to a suitable notion of homeomorphism preserving the foliation
(cf. \cite[Theorem 9.10]{bolsinov-fomenko-book}). If the multiplicity
is at least two, there are non-trivial obstructions for the existence
of a diffeomorphism preserving the foliation between sufficiently
small neighborhoods of focus-focus leaves (cf. \cite[Theorems A and
B]{bols_izo}). In the case of multiplicity one, \cite{san_invt} showed
that there is an obstruction to the existence of a strong symplectic
equivalence between sufficiently small foliated neighborhoods of
focus-focus leaves. The aim of this section is to present this
invariant, with an emphasis on some orientation issues that were not
considered in \cite{san_invt}.

To start, we introduce formally the type of integrable systems that we
are interested in.

\begin{defi}\label{defn:sat_neigh_ff}
  A {\em neighborhood of a focus-focus fiber of multiplicity one} is
  an almost-toric system $\left(M,\omega,F\right)$ such that:

  \begin{itemize}[leftmargin=*]
  \item All fibers are connected.
  \item There is only one singular fiber, $F^{-1}(c_0)$, which is of
    focus-focus type and has multiplicity one.
  \end{itemize}
\end{defi}

Examples of neighborhoods of focus-focus fibers of multiplicity one
can be constructed by performing self-plumbing of the unit disk bundle
of $T^* S^2$ (cf. \cite[Section 4.2]{symington} and \cite[Section
3]{zung_ff} for details). With notation as in Definition
\ref{defn:sat_neigh_ff}, we refer to $c_0$ as the {\em focus-focus
  value} of $\left(M,\omega,F\right)$.

\begin{rema}\label{rmk:sat_neigh}
  If $\left(M,\omega,F\right)$ is a neighborhood of a focus-focus
  fiber of multiplicity one, let $p_0 \in M$ denote the focus-focus
  point. Then
  \begin{enumerate}[label=(\alph*), ref = (\alph*), leftmargin=*]
  \item \label{item:15} $F$ is an open map. This is because, for any
    $U \subset M$ open,
    $F\left(U\right) = F \left(U \smallsetminus
      \left\{p_0\right\}\right)$,
    and $F|_{M \smallsetminus \left\{p_0\right\}}$ is open as it is a
    submersion.
  \item \label{item:16} $F$ is proper onto its image. This is a
    consequence of \cite[Theorem 3.3]{mrcun}. Using this fact, the
    leaf space $\mathcal{L}$ of such an integrable system can be
    identified topologically with the moment map image $F(M)$
    (cf. \cite[Lemma 2.15 and Proposition 3.20]{hsss}).
  \end{enumerate}
\end{rema}

The following result, stated below without proof
(cf. \cite[Proposition 6.3]{san_bohr}), establishes an important
property of neighborhoods of focus-focus fibers of multiplicity one
using, in a crucial fashion, Theorem \ref{thm:Eliasson}.

\begin{lemm}\label{lemma:connected}
  Let $\left(M,\omega,F\right)$ be neighborhood of a focus-focus fiber
  of multiplicity one whose focus-focus point is $p_0 \in M$. The
  subsystem of $\left(M,\omega,F\right)$ relative to
  $M \smallsetminus \left\{p_0\right\}$ has connected fibers.
\end{lemm}

While Theorem \ref{thm:Eliasson} gives a \emph{local} classification
of integrable systems near a focus-focus point, in this section we are
interested in a \emph{semiglobal} study, which means that we
investigate sufficiently small neighborhoods of the focus-focus
fiber. Therefore the object under study is really the germ of
subsystems in saturated neighborhoods of the focus-focus fiber: If
$\left(M,\omega,F\right)$ is a neighborhood of focus-focus fibers of
multiplicity one, then two subsystems of it,
$\left(M_1,\omega|_{M_1},F|_{M_1}\right)$ and
$\left(M_2,\omega|_{M_2},F|_{M_2}\right)$, have the \emph{same germ}
if they admit a common subsystem
$\left(M',\omega|_{M'},F|_{M'}\right)$ on which they agree, where
$M'=F^{-1}(W)$ for some open neighborhood $W$ of the focus-focus value
$c_0$.

Accordingly, we shall say that two neighborhoods of focus-focus fiber
of multiplicity one are \emph{equivalent} if they admit germs that are
equivalent in the sense of Definition~\ref{defi:eq-faible}.
Similarly, these two neighborhoods will be called \emph{isomorphic}
when they admit germs that are symplectically equivalent in the sense
of Definition~\ref{defn:equivalence}.

As a consequence of Theorem~\ref{thm:Eliasson}, there is no need to
distinguish here between the usual equivalence and the strong
equivalence of Definition~\ref{defn:strong_equiv}. This result is
mentioned already in~\cite{eliasson_thesis}; see
also~\cite{san-panoramas} and~\cite[Lemme 2.5]{san-fn}.
\begin{prop}
  Two neighborhoods of a focus-focus fiber of multiplicity one are
  isomorphic if and only if they admit germs that are strongly
  symplectically equivalent.
\end{prop}
\begin{proof}
  In view of Corollary \ref{coro:stronger}, we just need to prove that
  two symplectically equivalent neighborhoods of a focus-focus fiber
  of multiplicity one $\left(M_i,\omega_i,F_i\right)$, $i=1,2$, have
  germs that are strongly symplectically equivalent. By definition,
  there is a symplectomorphism
  $\varphi : \left(M_1,\omega_1\right) \to \left(M_2,\omega_2\right)$
  such that $F_1 \sim \varphi^* F_2$, {\it i.e.}
  \begin{equation}
    \C_{F_1}\left(M_1\right) = \C_{F_2 \circ \varphi}\left(M_1\right) =
    \varphi^*\left(\C_{F_2}\left(M_2\right)\right).
    \label{equ:equiv-commutant} 
  \end{equation}
  Observe that, since, for $i=1,2$, $p_i$ is the only non-degenerate
  singular point of $\left(M_i,\omega_i,F_i\right)$, Theorem
  \ref{thm:will_preserved} implies that $\varphi(p_1) = p_2$.  By
  Lemma \ref{lemma:connected}, the subsystems of
  $\left(M_1,\omega_1,F_1\right)$ and $\left(M_2,\omega_2,F_2\right)$
  relative to $M_1 \smallsetminus \left\{p_1\right\}$ and
  $M_2 \smallsetminus \left\{p_2\right\}$ have connected fibers and,
  by Definition \ref{defn:sat_neigh_ff}, contain no singular
  points. Therefore, for $i=1,2$,
  $F_i|_{M_i \smallsetminus \left\{p_i\right\}}$ is a submersion. This
  fact, together with connectedness of the fibers of
  $F_i|_{M_i \smallsetminus \left\{p_i\right\}}$, imply that
  $\C_{F_i}\left(M_i \smallsetminus \left\{p_i\right\}\right) =
  F_i^*\left(C^{\infty}\left(F_i\left(M_i \smallsetminus
        \left\{p_i\right\}\right)\right)\right)$.
  By the density of $M_i\smallsetminus \left\{p_1\right\}$ in $M_i$,
  we have the equality:
  $\Cinf(M_i)\cap \C_{F_i}(M_i \smallsetminus \left\{p_i\right\})=
  \C_{F_i}(M_i)$.
  On the other hand,
  $F_i(M_i \smallsetminus \left\{p_i\right\}) = F_i(M_i)$ (see
  property \ref{item:15} of Remark \ref{rmk:sat_neigh}). Thus,
  \begin{equation}
    \label{equ:weak-strong}
    \C_{F_i}(M_i) = F_i^*(\Cinf(F_i(M_i))).
  \end{equation}

  From equations~\eqref{equ:equiv-commutant}
  and~\eqref{equ:weak-strong}, there exists a smooth map
  $g : F_1\left(M_1\right) \to \R^2$ such that
  $g \circ F_1 = F_2 \circ \varphi$ on $M_1$. To complete the proof,
  it suffices to show that, restricted to smaller neighborhoods of
  $F_j(p_j)$ if necessary, $g$ is a diffeomorphism. This follows
  directly from observing that $\DD g(F_1(p_1))$ sends the Hessian of
  $F_1$ at $p_1$ to the Hessian of $F_2\circ\phy$ at $p_2$, and these
  Hessians have maximal rank due to the non-degeneracy of the
  singularity.
\end{proof}

Throughout we denote by $[(M,\omega,F)]$ the isomorphism class of the
germ of the system $(M,\omega,F)$ at the focus-focus fiber. The
problem that we wish to solve is determining the moduli space
$\mathcal{G}_{\mathrm{ff}}$ of these isomorphism classes, or,
equivalently, to construct sharp invariants that determine the
isomorphism class of a germ of a neighborhood of a focus-focus fiber
of multiplicity one. This is carried out in the subsections below,
following~\cite{san_invt}. Throughout the following subsections, the
integrable system
$\left(\R^4,\omega_{\mathrm{can}},q:=(q_1,q_2)\right)$ denotes the
local model of a singular point of focus-focus type (see
Definition~\ref{defn:e_h_ff} and
Section~\ref{sec:almost-toric-systems}):
\begin{equation}
  \label{equ:ff-dim4}
  q_1(x_1,y_1,x_2,y_2) = x_1 y_2 - x_2 y_1, \qquad 
  q_2(x_1,y_1,x_2,y_2) = x_1 y_1 + x_2 y_2.
\end{equation}

\subsubsection{Normalized neighborhoods of focus-focus
  fibers}\label{sec:norm-repr}
The first step is to show that any element of
$\mathcal{G}_{\mathrm{ff}}$ has a representative that is, locally near
the focus-focus point, symplectically conjugate to the linear
model. This fact is used in Section \ref{sec:calc-bundle-peri} to show
that its bundle of periods of the above representative exhibits a
universal asymptotic behaviour at the critical value.

Fix $\left(M,\omega,F\right)$, a neighborhood of a focus-focus fiber
of multiplicity one, and let $p_0 \in F^{-1}(c_0)$ denote the
focus-focus point. By Theorem \ref{thm:Eliasson}, there exist open
neighborhoods $U \subset M$ and $V \subset \R^4$ of $p_0$ and the
origin respectively, and a pair $\left(\varphi,g\right)$ consisting of
a symplectomorphism
$\varphi : \left(U,\omega|_U\right) \to
\left(V,\omega_{\mathrm{can}}|_V\right)$
and of a diffeomorphism $g: F(U) \to q(V)$ with $g(c_0) = (0,0)$ such
that
\[
g \circ F = q \circ \phy \qquad \text{ on } U.
\]

\begin{defi}\label{defn:eliasson_iso_diffeo}
  A pair $\left(\varphi,g\right)$ as above is called an {\em Eliasson
    isomorphism} for $\is$, while $\varphi$ (respectively $g$) is
  referred to as an {\em Eliasson symplectomorphism} (respectively
  {\em diffeomorphism}) for $\is$.
\end{defi}

\begin{rema}\label{rmk:more_than_one_eliasson}
  Given $\is$, there may be more than one Eliasson isomorphism for
  $\is$ (see Section \ref{sec:taylor-series-orbit}).
\end{rema}

Next we make precise the notion of representative that we are after.

\begin{defi}\label{defn:normalized}
  A neighborhood of a focus-focus fiber of multiplicity one $\is$ is
  said to be {\em normalizable} if one of its Eliasson
  symplectomorphisms is of the form
  $\left(\varphi,\mathrm{id}\right)$.
 
  A normalizable $\is$ together with a choice of Eliasson
  symplectomorphism $\left(\varphi,\mathrm{id}\right)$ is called {\em
    normalized} and is denoted by $\left(\is,\varphi\right)$.
\end{defi}

A normalizable neighborhood inherits the $S^1$-invariance of the
focus-focus local model, as follows.
\begin{lemm}\label{lemma:norm}
  Let $\is$ be a normalizable neighborhood of a focus-focus fiber of
  multiplicity one. Then there exists an open neighborhood $W$ of
  $c_0$ such that the first component of $F$ is the moment map of an
  effective Hamiltonian $S^1$-action on $F^{-1}(W)$.
\end{lemm}
\begin{proof}
  Let $\left(\varphi: U\to V,\mathrm{id}\right)$ be an Eliasson
  symplectomorphism for $\is$.  The map $q$ is open; let
  $W:=q(V)=F(U)$.  Observe that the first component of $F|_U$ is
  $q_1\circ\phy$, and hence is the moment map of an effective
  Hamiltonian $S^1$-action on $U$. However, since $U$ intersects all
  fibers of $F$ and these are connected, the result follows
  (cf. \cite{zung_ff} for details).
\end{proof}

\begin{lemm}\label{lemma:nice_rep}
  Any $\left[\is\right] \in \mathcal{G}_{\mathrm{ff}}$ has a
  normalizable representative.
\end{lemm}

\begin{proof}
  Fix $\left[\is\right] \in \mathcal{G}_{\mathrm{ff}}$ and fix an
  Eliasson isomorphism $\left(\varphi: U \to V,g\right)$ for $\is$.
  Then the system
  $\left((F^{-1}(F(U)), \omega|_{F^{-1}(F(U))}, g\circ F\right),
  \phy)$ is normalized.

\end{proof}


Combining Lemmata \ref{lemma:norm} and \ref{lemma:nice_rep}, we show
an important property of germs of neighborhoods of focus-focus fibers
of multiplicity one. To this end, we recall the following notion,
which plays an important role in \cite{zung_st2}.

\begin{defi}\label{defn:sys_pres}
  Given an integrable system $\is$ and an open subset $U \subset M$,
  an $S^1$-action on $U$ is said to be {\em locally system-preserving}
  if for all $\theta \in S^1$ and for all $p \in U$,
  $F\left(\theta \cdot p\right) = F(p)$.
\end{defi}

\begin{rema}\label{rmk:ham_sys_pres}
  Given $\is$, any moment map of a Hamiltonian local system-preserving
  $S^1$-action on an open subset $U \subset M$ is an element of the
  commutant $\C_{F}\left(U\right)$, {\it i.e.} it commutes with every
  component of $F|_U$.
\end{rema}

The following result, first proved in \cite[Proposition 4]{zung_ff},
shows that germs of neighborhoods of focus-focus fibers are naturally
endowed with a unique system-preserving Hamiltonian $S^1$-action.

\begin{coro}\label{cor:unique_sys_pres}
  Any sufficiently small neighborhood of a focus-focus fiber of
  multiplicity one possesses a unique (up to sign) effective Hamiltonian
  system-preserving $S^1$-action. 
\end{coro}
\begin{proof}
  Let $\is$ be a neighborhood of a focus-focus fiber of multiplicity
  one that is isomorphic to a normalizable one, denoted by
  $\left(M',\omega', F'\right)$. Lemma \ref{lemma:norm} shows that
  $\is$ has an effective system-preserving Hamiltonian $S^1$-action induced by
  pulling back the Hamiltonian $S^1$-action one of whose moment maps
  is the first component of $F'$. It remains to show that this action
  is unique up to sign. To this end, let $U \subset M$ be the domain
  of an Eliasson symplectomorphism for $\is$. Any effective system-preserving
  Hamiltonian $S^1$-actions on $\is$ restricts to a local
  system-preserving effective Hamiltonian $S^1$-action on $U$. This induces, via the
  given Eliasson isomorphism, an effective local system-preserving Hamiltonian
  $S^1$-action on an open neighborhood of the origin in the local
  model for a singular point of focus-focus type. It can be shown that
  this action is unique up to sign (cf. \cite[Proposition
  3.9]{hss}). By item \ref{item:20} in Proposition \ref{prop:ff}, it
  equals the Hamiltonian local system-preserving $S^1$-action induced
  by the first component of the moment map of the model. Thus, up to
  sign, any effective system-preserving Hamiltonian $S^1$-action on $\is$ is
  completely determined when restricted to the domain of an Eliasson
  symplectomorphism. Arguing as in the proof of Lemma
  \ref{lemma:norm}, this implies that it is uniquely determined up to
  sign on $M$ (cf. \cite{zung_ff} for details).
\end{proof}

\subsubsection{Regularized actions of normalized neighborhoods of
  focus-focus fibers of multiplicity one}\label{sec:calc-bundle-peri}

In this subsection we assume that $\left(M,\omega,F=(f_1,f_2)\right)$
is normalized in the sense of Definition~\ref{defn:normalized}, which
means that there exists a local symplectomorphism
$\phy:U\to V\subset\RM^4$ defined on an open neighborhood $U$ of the
focus-focus point $p_0\in M$, such that
\begin{equation}
  F  = q \circ \phy \qquad \text{ on } U.
  \label{equ:Fq}
\end{equation}
In particular, the focus-focus value of $\left(M,\omega,F\right)$ is
the origin in $\R^2$. Our goal is to construct the `regularized
action' for such a neighborhood of a focus-focus fiber of multiplicity
one (cf. \cite[Section 3]{san_invt}). In view of
Lemma~\ref{lemma:norm} and passing to a subsystem if necessary, we
assume that $M=F^{-1}(W)$, with $W=F(U)$ being simply connected.

\paragraph{Calculating the bundle of periods.}
Using the above data we calculate the {\em bundle of periods}
associated to the subsystem of $\left(M,\omega,F\right)$ (cf. Remark
\ref{rmk:aa}). Looking at the proof of Theorem \ref{thm:aa}, one way
to obtain this bundle (locally) is to fix a (local) Lagrangian section
$\sigma$ and calculate which closed 1-forms $\alpha$ satisfy
$\phi^{2\pi}_{\alpha} \circ \sigma = \sigma$, where $\phi^t_{\alpha}$
is the flow at time $t$ of the vector field $\ham{\alpha}$.  Lemma
\ref{lemma:norm} gives that the flow of $\ham{f_1}$ has period equal
to $2\pi$. Thus, if $\Sigma \to F(M) \smallsetminus \left\{0\right\}$
denotes the period bundle and $(a,b)$ denote the standard coordinates
on $\R^2$, we have that $\Z \langle \DD a \rangle \subset \Sigma$ and
$\DD a$ is primitive, {\it i.e.} for each point
$c \in F(M) \smallsetminus \left\{0\right\}$, the quotient
$\Sigma_c / \Z \langle \DD_c a \rangle$ has no torsion.

To complete the calculation of $\Sigma$, first we prove another
important consequence of Theorem \ref{thm:Eliasson}.

\begin{lemm}[\cite{san_invt}]\label{lemma:sections}
  Let $\left(\R^4,\omega_{\mathrm{can}},q\right)$ denote the local
  model for focus-focus points. Then there exist smooth Lagrangian
  sections $\sigma_s, \sigma_u : \R^2 \to \R^4$ with the property that
  $\sigma_s(0)$ (respectively $\sigma_u(0)$) intersects the stable
  (respectively unstable) manifold of the flow of $\ham{q_2}$.
\end{lemm}
\begin{proof}
  Fix the following complex coordinates on $\R^4 \cong \mathbb{C}^2$
  and on $\R^2 \cong \mathbb{C}$: $z_1=x_1 + i x_2$,
  $z_2 = y_2 + i y_1$, $w = a+ i b$. Then the map $(q_1,q_2)$ can be
  written as $\left(z_1,z_2\right) \mapsto z_1z_2$, {\it i.e.} as a
  complex hyperbolic map. It can be checked that the stable and
  unstable manifolds of the flow of $\ham{q_2}$ are given by
  $\left\{z_1 = 0\right\}$ and $\left\{z_2 = 0 \right\}$. Fix
  $\epsilon > 0$ and consider the sections
  $\sigma_s(w):=\left(\frac{w}{\epsilon}, \epsilon\right)$ and
  $\sigma_u(w):=\left(\epsilon,\frac{w}{\epsilon}\right)$. These
  sections are the desired ones.
\end{proof}

By abuse of notation, denote the restrictions of the smooth Lagrangian
sections of Lemma \ref{lemma:sections} to the open set $q(V)$ by
$\sigma_s$ and $\sigma_u$. From~\eqref{equ:Fq},
$\varsigma_s := \varphi^{-1}\circ \sigma_s$ and
$\varsigma_u:= \varphi^{-1}\circ \sigma_u$ are smooth Lagrangian
sections of $F$ such that $\varsigma_s(0)$ and $\varsigma_u(0)$ `lie
on opposite sides of' the focus-focus point $p_0$. These sections can
help us to determine $\Sigma$ using the method of \cite[Section
3]{san_invt}, as follows.

The aim is to construct a $\ZM$-basis of $\Sigma$ of the form
$(\DD a, \alpha)$ where $\alpha=\tau_1 \DD a + \tau_2 \DD b$ with
$\tau_2\neq 0$. Because of monodromy around $c_0$ (see Corollary
\ref{cor:monodromy} below) this basis cannot be globally defined on
$W\smallsetminus 0$, but its pull-back to the universal cover of
$W \smallsetminus 0$ can be. In the fact, the above basis is going to
be well-defined on $W\smallsetminus 0$ modulo
$\ZM \langle \DD a \rangle$. Let $U'\subset U$ be an open neighborhood
of $p_0\in M$ such that $M\smallsetminus U'$ contains the images of
$\varsigma_u$ and $\varsigma_s$ and $U'$ intersects all fibers of $F$
(thus $F(U')=F(U)=W$; such a $U'$ can be explicitly constructed in the
normal form). Since the action of
$T^* \left(W \smallsetminus \left\{0\right\}\right)$ on $M$ is
abelian, we may split $\alpha$ as $\alpha = \alpha_0 + \alpha_1$,
where $\alpha_0$ is the solution to
$\phi^{2\pi}_{\alpha_0} \circ \varsigma_u = \varsigma_s$, while
$\alpha_1$ is the solution to
$\phi^{2\pi}_{\alpha_1} \circ \varsigma_s = \varsigma_u$ with the
following properties:
\begin{enumerate}[label=(\arabic*.), leftmargin=*]
\item
  $\forall t\in [0,2\pi], \quad \forall w\in W, \quad
  \phi^t_{\alpha_0}\circ \varsigma_u(w)\in U$;
\item
  $\forall t\in [0,2\pi], \quad \forall w\in W, \quad
  \phi^t_{\alpha_1}\circ \varsigma_s(w)\in M\smallsetminus U'$.
\end{enumerate}
The upshot of this method, which justifies the passage to a normalized
system~\eqref{equ:Fq}, is that $\alpha_0$ can be calculated
explicitly. Near any $w\in W\smallsetminus 0$, $\alpha_0$ is given by
the closed form that satisfies
$\Phi^{2\pi}_{\alpha_0} \sigma_s = \sigma_u$, where
$\Phi^t_{\alpha_0}$ is the flow of the vector field corresponding to
$q^*\alpha_0$ using $\omega_{\mathrm{can}}$. Using the complex
coordinates of the proof of Lemma \ref{lemma:sections}, it can be
checked that, modulo $\Z \langle \DD a\rangle $,
\begin{equation}
  \label{eq:4}
  \alpha_0 := \frac{1}{2\pi} \mathrm{Im}\left(\DD \left( w \log w -
      w\right)\right) - \left(2 \ln \epsilon \right)\DD b, 
\end{equation}
where $\mathrm{Im}$ denotes the imaginary part and $\log$ is any
determination of the complex logarithm that is smooth near $w$
(cf. \cite[proof of Proposition 3.1]{san_invt}, taking into account
that, there, the convention for $q=(q_1,q_2)$ is to set the
$S^1$-moment map on the second component). Since
$\textup{Im}(w \log w) = a \arg w + b \log \abs{w}$, $\alpha_0$ is
smooth on the universal cover of $W\smallsetminus 0$, and well-defined
on $W\smallsetminus 0$ modulo $\ZM\langle \DD a \rangle$.

In order to define $\alpha_1$ we observe that the fibration $F$
restricted to $M\smallsetminus U'$ is trivial above the whole of $W$;
in fact the fibers of this restriction of $F$ are cylinders, see
Theorem~\ref{theo:topology_ff} and Lemma \ref{lemma:connected}.  Hence
the equation $\phi^{2\pi}_{\alpha_1} \circ \varsigma_s = \varsigma_u$
admits a unique smooth solution on $W$ for which the flow stays in
$M\smallsetminus U'$, as required by item (2.) above.  Moreover, since
$\varsigma_u, \varsigma_s$ are smooth Lagrangian sections, $\alpha_1$
is closed. Finally, writing
$\alpha = \alpha_0 + \alpha_1=\tau_1 \DD a + \tau_2 \DD b$ we see that
$\tau_2 \sim \frac{1}{2\pi}\log \abs{w}$ as $w\to 0$ and hence
$\tau_2\neq 0$ if $w$ is small enough, which proves that $\alpha$ is
linearly independent from $\DD a$.


Since $W$ simply connected, $\alpha_1$ is exact. This discussion
proves the following result (cf. \cite[Proposition 3.1]{san_invt}).
\begin{lemm}\label{lemma:period_ff}
  The bundle of periods associated a normalized neighborhood of a
  focus-focus fiber of multiplicity one is given by
  \begin{equation}
    \label{eq:10}
    \Sigma := \Z \left\langle \DD a, 
      \frac{1}{2\pi} \mathrm{Im}\left(\DD \left( w \log w -
          w\right)\right) + \DD h \right\rangle, 
  \end{equation}
  where $\log$ is any determination of the complex logarithm and
  $h : F(M) \to \R$ is a smooth function.
\end{lemm}

\paragraph{Intermezzo: Hamiltonian monodromy of neighborhoods of
  focus-focus fibers of multiplicity one.}
As remarked in~\cite{san_bohr}, Lemma \ref{lemma:period_ff} allows to
describe the Hamiltonian monodromy of any neighborhood of a
focus-focus fiber of multiplicity 1, a result that was first proved in
\cite{zou,matveev,zung_ff}. To see this, observe that Hamiltonian
monodromy measures precisely the non-triviality of the bundle of first
homology groups of the fibers (cf. \cite{dui,zung_ff}); and that this
bundle is dual to the bundle of periods via equation
\eqref{eq:12}. Furthermore, the latter is an invariant of the
isomorphism type of (a germ of) an integrable system (at a singular
fiber). Thus it suffices to consider normalized neighborhoods of
focus-focus fibers to calculate this invariant.

First, observe that, while the expression for the period bundle
$\Sigma$ of $\is$ of \eqref{eq:10} makes sense, the bundle
$\Sigma \to F(M) \smallsetminus \left\{(0,0)\right\}$ is {\em not
  trivial}, for the function $\log$ is multivalued. This is what gives
rise to non-trivial Hamiltonian monodromy in the presence of
focus-focus fibers (cf. \cite{matveev,zung_ff}); using Lemma
\ref{lemma:period_ff}, this can be formulated as follows. Set
$ l := \left\{ (a,b) \in \R^2 \mid a = 0, b \geq 0 \right\}$ and
identify in the following $\R^2 \cong \mathbb{C}$ using the complex
coordinate $w = a + i b$ of Lemma \ref{lemma:sections}. Let
$\mathrm{Log} : \mathbb{C} \smallsetminus l \to \mathbb{C}$ denote the
branch of the complex logarithm defined by setting, for all $b >0$,
\begin{equation}
  \label{eq:22}
  \lim\limits_{\stackrel{(a,b) \to (0,b)}{a > 0}} \arg(a+ib) =
  \frac{\pi}{2} \qquad \text{and} \qquad \lim\limits_{\stackrel{(a,b) \to (0,b)}{a < 0}} \arg(a+ib) =
  -\frac{3\pi}{2},
\end{equation}
where $\arg : \mathbb{C} \smallsetminus l \to \R$ is the induced
argument function. Moreover, set $\beta_1 := \DD a$,
$\beta_2 := \frac{1}{2\pi} \mathrm{Im}\left(\DD \left( w \log w -
    w\right)\right) + \DD h$.

\begin{coro}[Linear holonomy of a focus-focus fiber with
  multiplicity one]\label{cor:monodromy}
  For any $b > 0$, the following equality holds
  \begin{equation}
    \label{eq:11}
    \lim\limits_{\stackrel{(a,b) \to (0,b)}{a > 0}}
    \begin{pmatrix}
      \beta_1 (a,b)\\
      \beta_2 (a,b)
    \end{pmatrix} =
    \begin{pmatrix}
      1 & 0 \\
      1 & 1
    \end{pmatrix}
    \lim\limits_{\stackrel{(a,b) \to (0,b)}{a < 0}}
    \begin{pmatrix}
      \beta_1 (a,b)\\
      \beta_2 (a,b)
    \end{pmatrix}.
  \end{equation}
\end{coro}
\begin{proof}
  Fix $b > 0$. Since $\beta_1$ is globally defined on $\R^2$, it
  follows that
  \[
  \lim\limits_{\stackrel{(a,b) \to (0,b)}{a > 0}}\beta_1(a,b) =
  \lim\limits_{\stackrel{(a,b) \to (0,b)}{a < 0}}\beta_1(a,b).
  \]

  On the other hand, with the above choice of complex logarithm, have
  that
  $$ \lim\limits_{\stackrel{(a,b) \to (0,b)}{a > 0}}\beta_2(a,b) -
  \lim\limits_{\stackrel{(a,b) \to (0,b)}{a < 0}}\beta_2(a,b) =
  \beta_1(0,b);$$
  this uses the fact that the smooth function $h$ in Lemma
  \ref{lemma:period_ff} is continuous at the origin.
\end{proof}


\paragraph{The regularized action.}
Lemma \ref{lemma:period_ff} contains data of the normalized system
that is finer than the Hamiltonian monodromy. Fix the above
identification $\R^2 \cong \mathbb{C}$ and the complex logarithm
$\mathrm{Log} : \mathbb{C} \smallsetminus l \to \mathbb{C}$ determined
by equation \eqref{eq:22}, and consider the restriction of the bundle
of periods $\Sigma|_{F(M) \smallsetminus l}$. First, observe that the
smooth function $h \in C^{\infty}\left(F(M)\right)$ is not uniquely
defined by \eqref{eq:10}. This is because if
$h_0 \in C^{\infty}\left(F(M)\right)$ is a smooth function satisfying
\[
\Sigma|_{F(M) \smallsetminus l}= \Z \left\langle \DD a, \frac{1}{2\pi}
  \mathrm{Im}\left(\DD \left( w \, \mathrm{Log}\, w - w\right)\right)
  + \DD h_0 \right\rangle,
\]
then, for any $k \in \Z$ and $c \in \R$, the function
$h:= h_0 + ka + c$ is another function satisfying the above
equality. However, there is a unique smooth function
$\mathsf{h} \in C^{\infty}\left(F(M)\right)$ satisfying the above
equality and the conditions
\begin{equation}
  \label{eq:23}
  \mathsf{h}(0,0) = 0 \qquad \text{and} \qquad \frac{\partial
    \mathsf{h}}{\partial a}(0,0) \in [0,1[.
\end{equation}
Following \cite[Remark 3.2]{san_invt}, we introduce the following
terminology.

\begin{defi}\label{defn:reg_action}
  Let $\left(\is,\varphi\right)$ be a normalized neighborhood of a
  focus-focus fiber of multiplicity one. The function
  $\mathsf{h}\left(\is,\varphi\right) \in C^{\infty}\left(F(M)\right)$
  constructed above is called the {\em regularized action of
    $\left(\is,\varphi\right)$}.
\end{defi}

\begin{rema}\label{rmk:choices}
  Strictly speaking, the above construction of the regularized action
  of $\left(\is,\varphi\right)$ depends on the following other
  choices:
  \begin{itemize}[leftmargin=*]
  \item smooth Lagrangian sections of the local model for a singular
    point of focus-focus type which satisfy the conditions of Lemma
    \ref{lemma:sections};
  \item a branch of the complex logarithm;
  \item the conditions \eqref{eq:23}.
  \end{itemize}
  While there are ways to adapt the above construction so that it
  becomes independent of the above choices, we follow the blueprint of
  \cite[Section 3]{san_invt} and fix the above choices as there is no
  loss in generality in doing so.
\end{rema}

\subsubsection{The symplectic invariant}\label{sec:taylor-series-orbit}
The aim of this subsection is to construct an invariant of the
isomorphism class of the germ of a neighborhood of a focus-focus fiber
of multiplicity one starting from the Taylor series at the origin of
the regularized action of Definition \ref{defn:reg_action}, following
the approach of \cite{san_invt}. In what follows, given a smooth
function $h$ defined near the origin, we denote the Taylor series of
$h$ at $(0,0)$ by $\left(h\right)^{\infty}$.

To start, observe that given any
$\left[\is\right] \in \mathcal{G}_{\mathrm{ff}}$, the proof of Lemma
\ref{lemma:nice_rep} constructs a normalized
$\left(\left(M',\omega',F'\right),\varphi\right)$ such that
$\left(M',\omega',F'\right) \in \left[\is\right]$. This construction
depends on two choices, namely:
\begin{enumerate}[label=(\Roman*), ref = (\Roman*), leftmargin=*]
\item \label{item:19} a representative $\is \in \left[\is\right]$;
\item \label{item:26} an Eliasson symplectomorphism
  $\left(\varphi,g\right)$ for the representative $\is$.
\end{enumerate}
Section \ref{sec:calc-bundle-peri} associates to
$\left(\left(M',\omega',F'\right),\varphi\right)$ its regularized
action $\mathsf{h}\left(\left(M',\omega',F'\right),\varphi\right)$
(see Definition \ref{defn:reg_action}). Thus, in order to define an
invariant of $\left[\is\right]$ starting from
$\left(\mathsf{h}\left(\left(M',\omega',F'\right),\varphi\right)\right)^{\infty}$,
it suffices to construct an object that does not depend on the choices
\ref{item:19} and \ref{item:26}. This is achieved below dealing first
with the case in which the choice \ref{item:19} is fixed and then with
the general case. However, in order to define the invariant, we first
need to understand strong symplectic equivalences of the local model
for a singular point of focus-focus type relative to subsets
containing the origin as well as a group action on Taylor series of
smooth functions at the origin. This is achieved in the following two
intermezzos.

\paragraph{First intermezzo: germs of automorphisms of the local model
  for a singular point of focus-focus type.}


Let $\mathrm{Aut}\left((\R^4,0),\omega_{\mathrm{can}},q\right)$ be the
set of germs of symplectomorphisms $\phy$ defined in a neighborhood of
the origin in $\RM^4$ that act on the strong equivalence class of $q$,
\emph{i.e.} for which there exists a (germ of) local diffeomorphism
$g$ of $(\RM^2,0)$ such that
\begin{equation}
  q \circ \phy = g \circ q.
  \label{equ:automorphism}
\end{equation}

One question that we need to address in order to define the desired
symplectic invariant is: what are the possible $g$'s obtained in this
fashion?

As mentioned in \cite[proof of Lemma 4.1]{san_invt}, if
$\left(\varphi,g \right)$ satisfies~\eqref{equ:automorphism}, then the
germ of $g$ at the origin in $\R^2$ is determined by that of $\varphi$
at the origin in $\R^4$, as the fibers of $q$ are locally connected
near the origin. Hence we may identify
$\mathrm{Aut}\left((\R^4,0),\omega_{\mathrm{can}},q\right)$ with the
set of pair $(\phy,g)$
satisfying~\eqref{equ:automorphism}. Furthermore, if $(\phy',g')$ is
another pair of germs satisfying~\eqref{equ:automorphism}, then
$q\circ \phy \circ \phy' = g \circ g ' \circ q$; this shows that the
group structure of
$\mathrm{Aut}\left((\R^4,0),\omega_{\mathrm{can}},q\right)$ extends to
pairs.

The analysis performed in~\cite{san_invt} gives the following answer
to the above question.  Let $G$ be the group of germs of local
diffeomorphisms of $(\RM^2,0)$ of the form
\begin{equation}
  \label{equ:g-group}
  g(a,b) = \left(\epsilon_1 a,\epsilon_2 b +
    \mathrm{O}\left(\infty\right)\right), 
\end{equation}
where $a,b$ are the standard coordinates on $\R^2$, for $i=1,2$,
$\epsilon_i \in \{ \pm 1\}$, and $\mathrm{O}\left(\infty\right)$
denotes a flat function in the variables $a$ and $b$ defined near
$(0,0)$ (this is by definition a smooth function $h$ with
$(h)^\infty=0$).
\begin{prop}
  \label{prop:germs_auto}
  There exists a pair $(\phy,g)$ satisfying~\eqref{equ:automorphism}
  if and only of $g\in G$. (In other words, the natural group
  homomorphism
  $\mathrm{Aut}\left((\R^4,0),\omega_{\mathrm{can}},q\right) \to G$ is
  onto.)
\end{prop}

\begin{proof}
  The necessity $g\in G$ is proved in \cite[Lemma 4.1]{san_invt}.  In
  order to prove the converse, let
  \begin{equation}
    \vec\epsilon:G\to (\ZM/2\ZM)\times (\ZM/2\ZM)
    \label{equ:signs}
  \end{equation}
  be the homomorphism that assigns to $g\in G$ of the
  form~\eqref{equ:g-group} the `signs'
  $(\mathsf{s}(\epsilon_1),\mathsf{s}(\epsilon_2))$, where
  $\mathsf{s} : \left\{+1,-1\right\} \to \ZM/2\ZM$ is the natural
  isomorphism. 
  If $\vec\epsilon(g)=(0,0)$, then the existence a pair
  $\left(\phy,g\right)$ satisfying \eqref{equ:automorphism} is
  established in \cite[Lemma 5.1, step (2) of the
  proof]{san_invt}. The general case is briefly mentioned in
  \cite[proof of Lemma 4.1]{san_invt} and we recall it because it
  plays an important role.

  Notice that the homomorphism $\vec\epsilon$ is onto, for, given any
  $j=(j_1,j_2)\in (\ZM/2\ZM)\times (\ZM/2\ZM)$, the linear
  diffeomorphism $g_j(a,b):=(\epsilon_1 a, \epsilon_2 b)$, where
  $\epsilon_k:=\mathsf{s}^{-1}(j_k)$, satisfies
  $\vec\epsilon(g_j) = j$.  Consider the two linear symplectomorphisms
  $A_{1,0}$ and $A_{0,1}$ in $\textup{Sp}(4,\RM)$ given by
  \begin{equation}
    \label{equ:A}
    \begin{split}
      A_{1,0}\left(x_1,y_1,x_2,y_2\right) &=
      \left(x_2,y_2,x_1,y_1\right) \\
      A_{0,1}\left(x_1,y_1,x_2,y_2\right) &=
      \left(y_1,-x_1,y_2,-x_2\right),
    \end{split}
  \end{equation}
  and notice, in view of~\eqref{equ:ff-dim4}, that
  $q\circ A_{1,0} = (-q_1,q_2)$ and $q\circ A_{0,1} = (q_1,-q_2)$. In
  other words, the pairs $(A_{1,0}, g_{1,0})$ and $(A_{0,1}, g_{0,1})$
  both satisfy~\eqref{equ:automorphism}. Since $A_{1,0}$ and $A_{0,1}$
  commute, setting $A_{1,1}:=A_{1,0}\circ A_{0,1}$, we obtain an
  injective homomorphism from $\ZM/2\ZM\times\ZM/2\ZM$ to
  $\textup{Sp}(4,\RM)$ mapping $j=(j_1,j_2)$ to $A_j$. This, in turn,
  defines an injective homomorphism from $\ZM/2\ZM\times \ZM/2\ZM$ to
  $\mathrm{Aut}\left((\R^4,0),\omega_{\mathrm{can}},q\right)$, sending
  $j$ to $(A_j, g_j)$.

  Hence, if a general $g\in G$ is given, we are reduced to the case of
  vanishing $\vec\epsilon$, for, once we find
  $(\phy,g\circ g_j)\in
  \mathrm{Aut}\left((\R^4,0),\omega_{\mathrm{can}},q\right)$,
  with $j=\vec\epsilon (g)$, then
  $(\tilde\phy,g)\in
  \mathrm{Aut}\left((\R^4,0),\omega_{\mathrm{can}},q\right)$
  with $\tilde\phy := \phy\circ A_j$.

\end{proof}

\begin{rema}
  \label{rema:orientation-preserving}
  Depending on the type of equivalence of neighborhoods of focus-focus
  fibers of multiplicity one under consideration, we may restrict our
  attention to specific subgroups of
  $\mathrm{Aut}\left((\R^4,0),\omega_{\mathrm{can}},q\right)$. For
  instance, if we consider only strong symplectic equivalences whose
  underlying diffeomorphisms preserve orientation, then the
  corresponding subgroup of
  $\mathrm{Aut}\left((\R^4,0),\omega_{\mathrm{can}},q\right)$ is the
  preimage of the subgroup $G_+ \subset G$ consisting of
  orientation-preserving germs. Observe that $G_+$ surjects onto the
  diagonal $\ZM/2\ZM \subset (\ZM/2\ZM)\times (\ZM/2\ZM)$ under
  $\vec\epsilon$.

  Using the notion of isomorphism of semi-toric systems (see
  Definition \ref{defn:equiv_st}), the corresponding subgroup of
  $\mathrm{Aut}\left((\R^4,0),\omega_{\mathrm{can}},q\right)$ is the
  preimage of $G_0 =\ker \vec\epsilon \subset G$, which consists of
  elements which are, up to a flat term, equal to the identity.
    
\end{rema}

\paragraph{Second intermezzo: actions of $K_4:=\Z/2\Z \times \Z/2\Z$
  on (equivalence classes of) Taylor series.}
As shown above the Klein group $K_4=\Z/2\Z \times \Z/2\Z$ appears
naturally when studying (germs of) automorphisms of the local model
for a singular focus-focus point. The proof of Proposition
\ref{prop:germs_auto} shows that $K_4$ acts on $\RM^4$ by linear
symplectomorphisms and on $\RM^2$ by linear automorphisms.  In order
to define the desired symplectic invariant, we introduce an action of
$K_4$ on (some quotient of) the space of Taylor series in two
variables, which is \emph{not} the natural action induced by the
linear automorphisms of $\RM^2$, but which is natural for the problem
under consideration, see Lemma \ref{lemm:K4-equivariant} and Theorem
\ref{thm:well_defined}.

By Borel's lemma, the space of Taylor series (at the origin) of smooth
functions on $(\RM^2,0)$ is identified with the space
$\RM\formel{a,b}$ of formal power series in two variables. Let
\begin{equation}
  \RM\formel{a,b}_\textup{ff} := \RM\formel{a,b} / (\RM \oplus \ZM a)
  \label{equ:taylor-ff}
\end{equation}
be the space of formal power series $T(a,b)$ with no constant term and
for which the $a$ coefficient is taken modulo $\ZM$. The following
lemma, whose elementary proof is left to the reader, introduces the
desired $K_4$-action on $\RM\formel{a,b}_\textup{ff}$.

\begin{lemm}\label{lemm:new_action}
  The map
  $\RM\formel{a,b}_\textup{ff} \times K_4 \to
  \RM\formel{a,b}_\textup{ff}$
  given by
  $\left(T (a,b),(j_1,j_2)\right) \mapsto T (a,b) \star (j_1,j_2)$,
  where
  \begin{equation}
    \label{equ:j-action}
    T(a,b) \star
    (j_1,j_2) := \mathsf{s}^{-1}(j_2) T(g_{j_1,j_2}(a,b)) + \frac{j_1 a}{2}
  \end{equation}
  defines an effective right $K_4$-action on
  $\RM\formel{a,b}_\textup{ff}$.
\end{lemm}
Henceforth, if $h\in\Cinf(\RM,0)$, we denote by $[h]^\infty$ the class
of its Taylor series $(h)^\infty$ in $\RM\formel{a,b}_\textup{ff}$;
moreover, the orbit of $T = T (a,b) \in \RM\formel{a,b}_\textup{ff}$
under the action of Lemma \ref{lemm:new_action} is denoted by
$\mathcal{O}(T)$. The symplectic invariant of isomorphism classes of a
germ $\left[\left(M,\omega,F\right)\right]$ of a neighborhood of a
focus-focus fiber of multiplicity one is precisely
$\mathcal{O}\left([\mathsf{h}]^\infty\right)$, where
$\mathsf{h} =
\mathsf{h}\left(\left(M',\omega',F'\right),\varphi\right)$
is the regularized action of any normalized
$\left(\left(M',\omega',F'\right),\varphi\right)$ with
$\left(M',\omega',F'\right) \in \left[\left(M,\omega,F\right)\right]$
(see Definition \ref{defn:symp_invariant} and Theorem
\ref{thm:taylor_sharp}).


Before showing the above claim (which is carried out in the
subsections below), we describe the orbit space
$\RM\formel{a,b}_\textup{ff}/K_4$. Consider the leading order
coefficients of $ T \in \RM\formel{a,b}_\textup{ff}$: the coefficients
of the monomials $a$ and $b$. It is clear from \eqref{equ:j-action}
that the coefficient of $b$ does not depend on the representative in
$\mathcal{O}\left(T\right)$. This number is directly related to the
dynamics of the radial hyperbolic vector field (which, in normalized
coordinates, is given by the Hamiltonian $q_2$, see also \cite[Step 2
of Section 5.2]{pel_san_inv}), and has been explicitly calculated in a
number of famous examples (cf. \cite{pel_san_spin, dullin13,
  papadopoulos-dullin13}). On the other hand, the $K_4$-action on the
coefficient of $a$ is non-trivial and can be used to classify almost
all equivalence classes.

\begin{prop}
  \label{prop:K4-domain}
  Identify the $a$-coefficient of an element in
  $\RM\formel{a,b}_\textup{ff}$ with a point in the unit circle $U(1)$
  via $t\mapsto e^{i2\pi t}$. Then
  \begin{enumerate}[label = \arabic*., ref= \arabic*, leftmargin=*]
  \item \label{item:17} The $K_4$-action on
    $\RM\formel{a,b}_\textup{ff}$ induces the standard $K_4$-action on
    $U(1)$, \emph{i.e.} $(1,0)$ acts as reflection in the imaginary
    axis, while $(0,1)$ acts as reflection in the real axis.
  \item \label{item:18} A fundamental domain for the $K_4$-action on
    $\RM\formel{a,b}_\textup{ff}$ is given by the subset whose
    elements are formal series
    \begin{equation}
      T = C_{1,0}a + C_{0,1}b + \sum_{i_1+i_2\geq 2} C_{i_1,i_2}a^{i_1}
      b^{i_2}\label{equ:formal-series}
    \end{equation}
    such that:
    \begin{enumerate}[label=(\roman*), ref = (\roman*)]
    \item \label{item:25} either
      $C_{1,0}\in \left]0,\tfrac{1}{4}\right[ \mod 1$;
    \item \label{item:27} or $C_{1,0}=0 \mod 1$ and, if
      $i_{b}:=\inf\{(i_1,2k) \mid (i_1,k)\in \NM\times \NM \text{ and
      } C_{i_1,2k}\neq 0\} \neq \infty$,
      where the $\inf$ is taken for the lexicographic order,
      $C_{i_{b}}>0$;
    \item \label{item:28} or $C_{1,0}=\tfrac{1}{4} \mod 1$ and, if
      $i_{a}:=\inf\{(2k+1,i_2) \mid (k,i_2)\in\NM\times \NM \text{ and
      } C_{2k+1,i_2}\neq 0\}\neq \infty$, $C_{i_{a}}>0$.
    \end{enumerate}
  \item \label{item:24} Let $\ZM_{1,1}\subset K_4$ denote the diagonal
    subgroup (isomorphic to $\ZM/2\ZM$). A fundamental domain for the
    $\ZM_{1,1}$-action on $\RM\formel{a,b}_\textup{ff}$ is given by
    elements~\eqref{equ:formal-series} for which
    $C_{1,0}\in \left[0,\frac{1}{2} \right[ \mod
    1$.
  \end{enumerate}
\end{prop}
\begin{demo}
  The first item is immediate from \eqref{equ:j-action}.  Consider
  $T\in\RM\formel{a,b}_\textup{ff} $ as in~\eqref{equ:formal-series}
  and its $K_4$-orbit $\mathcal{O}(T)$. If $4C_{1,0}\neq 0 \mod 1$,
  then item \ref{item:17} implies that there is a unique
  representative in $\mathcal{O}(T)$ whose $a$-coefficient belongs to
  $\left]0,\tfrac{1}{4}\right[ \mod 1$, thus establishing
  \ref{item:25}. 
  Suppose that $4C_{1,0} = 0 \mod 1$; using the $K_4$-action, there is
  no loss in generality in assuming that $C_{1,0}=0$ or
  $C_{1,0}=\tfrac{1}{4}$ (modulo $1$). However, in either case there
  are two elements of $\mathcal{O}(T)$ that satisfy the corresponding
  condition. Consider first the case $C_{1,0}=0$. The two elements of
  $\mathcal{O}(T)$ in question are $T$ and $T\star (0,1)$, and they
  are equal if and only if $C_{i_1,2k}=0$ for all $i_1,k\geq 0$.
  Therefore we may distinguish them by requiring that the first
  nonzero such coefficient is positive. The case
  $C_{1,0}=\tfrac{1}{4}$ can be treated similarly, thus establishing
  \ref{item:27} and \ref{item:28}, and completing the proof of
  \eqref{equ:formal-series}.

  Concerning \ref{item:24}, observe that $\ZM_{1,1}$ acts on the
  coefficient $C_{1,0}$ by the antipody of the circle $\RM/\ZM$. Hence
  there exists a unique representative of a class in
  $\RM\formel{a,b}_\textup{ff}/\ZM_{1,1}$ with
  $C_{1,0}\in \left[0,\frac{1}{2}\right[$.
\end{demo}

\begin{rema} 
  Item \ref{item:24} in Proposition~\ref {prop:K4-domain} implies
  that, instead of considering the $K_4$-orbit $\mathcal{O}(T)$, which
  generally consists of $4$ elements, one can always pick a
  representative for which $C_{1,0} \in [0,1/2[ \mod 1$ and consider
  its $\ZM/2\ZM$ orbit by the action of $(1,0)$.

  Even more interestingly, we know that, in order to define the
  desired symplectic invariant under the stricter notion of
  isomorphism that allows only for symplectic equivalences whose
  underlying diffeomorphisms preserve orientation, we only need to
  consider the action of the subgroup $\ZM_{1,1}$ (see Remark
  \ref{rema:orientation-preserving}). In this case, item \ref{item:24}
  of Proposition~\ref {prop:K4-domain} shows that the symplectic
  invariant can be identified with a Taylor series whose
  $a$-coefficient belongs to $\left[0,\frac{1}{2}\right[$.
\end{rema}

\paragraph{The Taylor series orbit of a germ of a neighborhood of a
  focus-focus fiber.}

If $((M,\omega,F), \phy)$ is a normalized neighborhood of a
focus-focus fiber of multiplicity one (see
Definition~\ref{defn:normalized}), and
$(\psi,g)\in
\mathrm{Aut}\left((\R^4,0),\omega_{\mathrm{can}},q\right)$,
then the subsystem of $\left(M,\omega,g^{-1}\circ F\right)$ relative
to a sufficiently small saturated neighborhood of the focus-focus
fiber is normalized by $\psi^{-1}\phy$. This defines an action of
$\mathrm{Aut}\left((\R^4,0),\omega_{\mathrm{can}},q\right)$ on the set
of 
germs of normalized neighborhoods.  Thus, using~\eqref{equ:A}, we
obtain an action of the group $K_4=\Z/2\Z \times \Z/2\Z$ on the same
set, using the map $j\mapsto (A_j,g_j)$.

On the other hand, the regularized action of
Definition~\ref{defn:reg_action} can be interpreted as a map
$\mathsf{h}$ from the set of normalized neighborhoods to
$\Cinf(\RM^2,0)$. We denote by $[\mathsf{h}]^\infty$ the left
composition of this map by the projection onto
$\RM\formel{a,b}_\textup{ff}$ (see \eqref{equ:taylor-ff}); the latter
is endowed with the $K_4$-action of Lemma~\ref{lemm:new_action}.

\begin{lemm}
  \label{lemm:K4-equivariant}
  The map $[\mathsf{h}]^\infty$ is $K_4$-equivariant.
\end{lemm}
\begin{proof}
  We need to show that, for any $j\in K_4$,
  \begin{equation}
    \label{eq:36}
    [\mathsf{h}]^\infty((M,\omega,g_j^{-1}\circ F), A_j^{-1}\phy) =
    [\mathsf{h}]^\infty((M,\omega,F), \phy) \star j.
  \end{equation}
  Notice that, since any element of $K_4$ has order two,
  $A_j^{-1} = A_j$ and $g_j^{-1} = g_j$. Moreover the choice of a
  determination of the complex logarithm has no influence on
  $[\mathsf{h}]^\infty$, because different choices of the former
  modify $\mathsf{h}$ by adding integer multiples of $\DD a$,
  see~\eqref{eq:10}.

  The key point to prove the result is Corollary \ref{cor:ias_invt},
  {\it i.e.} the fact that the bundles of periods associated to
  strongly symplectically equivalent systems are isomorphic. Fix the
  germ of $((M,\omega,F), \phy)$ of a normalized neighborhood of a
  focus-focus fiber of multiplicity one and let $\Sigma$ denote the
  bundle of periods associated to $((M,\omega,F), \phy)$ as in Lemma
  \ref{lemma:period_ff}. For any $j \in K_4$, let $\Sigma_j$ denote
  the bundle of periods associated to
  $\left(\left(M,\omega,g^{-1}_j\circ F\right), A^{-1}_j \phy\right)$.
  Since $\left(\mathrm{id},g_j\right)$ is a strong symplectic
  equivalence between $\left(M,\omega,g^{-1}_j\circ F\right)$ and
  $(M,\omega,F) $, Corollary \ref{cor:ias_invt} implies that
  \begin{equation}
    \label{eq:37}
    g_j^* \Sigma = \Sigma_j.
  \end{equation}
  By Lemma \ref{lemma:period_ff}, there exists a smooth function $h$
  defined near the origin in $\R^2$ such that
  \[
  \Sigma= \Z \left\langle \DD a, \frac{1}{2\pi} \mathrm{Im}\left(\DD
      \left( w \log w - w\right)\right) + \DD h\right\rangle.
  \]
  Fix any such function $h$. If $j = (0,0)$, \eqref{eq:36} is
  trivially satisfied; the remaining cases are dealt with separately.

\paragraph{Case $j=(1,0)$:}
By \eqref{eq:37}, we obtain that
\[
\Sigma(g_j\circ F) = \Z \left\langle g_j^* \DD a, \frac{1}{2\pi}
  \mathrm{Im}\left(g_j^* \DD \left( w \log w - w\right)\right) + \DD
  h\circ g_j \right\rangle.
\]
Since $j=(1,0)$, $g_j(w)=-\bar w$. Writing
$\mathrm{Im} (\DD (w \log w - w)) = \arg w\, \DD a + \log \abs{w} \DD
b$, we get
\begin{align*}
  \mathrm{Im}\left(g_j^* \DD \left( w \log w - w\right)\right)
  & = (\pi - \arg w) g_j^* \DD a + \log \abs{w} g_j^*\DD b \quad 
    \mod 2\pi\ZM \langle \DD a \rangle \\
  & = (\arg w - \pi)\DD a + \log \abs{w}\DD b \quad
    \mod 2\pi\ZM \langle \DD a \rangle\\
  & = \mathrm{Im}\left(\DD \left( w \log w - w\right)\right) - \pi \DD a \quad
    \mod 2\pi\ZM \langle \DD a \rangle.
\end{align*}
Hence
\[
\Sigma_j = \Z \left\langle \DD a, \frac{1}{2\pi} \mathrm{Im}\left( \DD
    \left( w \log w - w\right)\right) + \DD \tilde h_j \right\rangle,
\]
where $\tilde h_j := h\circ g_j - \pi a $. Thus,
$[\tilde h_j]^\infty = [h \circ g_j - \pi a]^\infty = [h]^\infty \star
j$,
see Lemma \ref{lemm:new_action}. Hence, in view of
Definition~\ref{defn:reg_action}, we get the result.

\paragraph{Case $j=(0,1)$.}
In this case, $g_j(w)=\bar w$, and we obtain
\begin{align*}
  \mathrm{Im}\left(g_j^* \DD \left( w \log w - w\right)\right)
  & = - \arg w \, g_j^* \DD a + \log \abs{w} g_j^*\DD b \quad 
    \mod 2\pi\ZM \langle \DD a \rangle \\
  & = - \arg w \, \DD a - \log \abs{w}\DD b \quad
    \mod 2\pi\ZM \langle \DD a \rangle\\
  & = - \mathrm{Im}\left(\DD \left( w \log w - w\right)\right)  \quad
    \mod 2\pi\ZM \langle \DD a \rangle.
\end{align*}
Hence,
$$
\Sigma_j = \Z \left\langle \DD a, \frac{1}{2\pi} \mathrm{Im}\left( \DD
    \left( w \log w - w\right)\right) + \DD \tilde h_j \right\rangle,
$$
where $\tilde h_j := - h\circ g_j$. Thus,
$[\tilde h_j]^\infty = [-h \circ g_j]^\infty = [h]^\infty \star j$,
(again, see Lemma \ref{lemm:new_action}), thus yielding the result in
this case.

Finally, by the group homomorphism property, the case $j=(1,1)$
follows from the previous cases, thus completing the proof.
\end{proof}

Using Lemma \ref{lemm:K4-equivariant}, we construct the symplectic
invariant for any given representative $\is \in \left[\is\right]$ (in
other words, we fix choice \ref{item:19}). For $i=1,2$, let
$\left(\varphi_i,g_i\right)$ be an Eliasson symplectomorphism of $\is$
and denote by $\left(\left(M_i,\omega_i,F_i\right),\varphi_i\right)$
the normalized neighborhood constructed as in the proof of Lemma
\ref{lemma:nice_rep} starting from $\is$ and
$\left(\varphi_i,g_i\right)$. For $i=1,2$, denote the regularized
action of $\left(\left(M_i,\omega_i,F_i\right),\varphi_i\right)$ by
$\mathsf{h}_i$.

\begin{theo}\label{thm:well_defined}
  With the above notation,
  $\mathcal{O}\left([\mathsf{h}_1]^\infty\right) =
  \mathcal{O}\left([\mathsf{h}_2]^\infty\right)$.
\end{theo}
\begin{proof}
  Fix data as above. Begin by observing that, for $i=1,2$, since
  $\mathcal{O}\left([\mathsf{h}_i]^\infty\right)$ does not depend on
  the germs of $\left(M_i,\omega_i,F_i\right)$ and of $\varphi_i$,
  there is no loss in generality in assuming that the domains of
  $\varphi_1$ and $\varphi_2$ are equal to the open subset
  $U \subset M$, and that $M=F^{-1}\left(F(U)\right)$ (and hence
  $\omega_1=\omega_2$ and $F_i = g_i \circ F$).  Setting, for $i=1,2$,
  $V_i := \varphi_i(U)$, $g:= g_2 \circ g_1^{-1}$, and
  $\varphi:= \varphi_2 \circ \varphi^{-1}_1$, the following diagram
  commutes:
  \begin{equation}
    \label{eq:26}
    \xymatrix{V_1 \ar[dd]_-{q} \ar@/^2pc/[rr]^-{\varphi} & U \ar[l]_-{\varphi_1}
      \ar[r]^-{\varphi_2} \ar@{^{(}->}[d] & V_2 \ar[dd]^-{q} \\
      & \left(M,\omega\right) \ar[dl]_-{F_1} \ar[d]^-{F} \ar[dr]^-{F_2}
      & \\
      q(V_1) = F_1(M) \ar@/_2pc/[rr]_-{g}& F(M) = F(U) \ar[r]_-{g_2} \ar[l]^-{g_1} & F_2(M)
      = q(V_2).
    }
  \end{equation}
  In particular, $\left(\varphi,g\right)$ and
  $\left(\mathrm{id},g\right)$ are strong symplectic equivalences
  between the subsystems of
  $\left(\R^4,\omega_{\mathrm{can}},q\right)$ relative to $V_1$ and
  $V_2$, and between $\left(M,\omega,F_1\right)$ and
  $\left(M,\omega,F_2\right)$ respectively. As in
  Proposition~\ref{prop:germs_auto}, set $j:=\vec\epsilon(g)$. Acting
  on the normalized neighborhood $(M,\omega,F_2)$ by $j$, we obtain a
  new normalized neighborhood
  for which the regularized action belongs to the same orbit
  $\mathcal{O}\left([\mathsf{h}_2]^\infty\right)$ by to
  Lemma~\ref{lemm:K4-equivariant}.  Replacing
  $\left(M,\omega,F_2\right)$ by the system constructed above, we
  obtain the corresponding diagram~\eqref{eq:26} above with the
  property that $\vec\epsilon(g) = \left(0,0\right)\in K_4$. Thus,
  Proposition~\ref{prop:germs_auto} implies that
  \begin{equation}
    \label{equ:g-identity}
    g(a,b)=(a,b+\textup{O}(\infty)).
  \end{equation}

  Arguing as in the proof of Lemma~\ref{lemm:K4-equivariant}, we may
  write $g^*\Sigma_2$ and observe that the smooth function $h$ from
  the computation of $\Sigma_1$ is modified by a term in
  $\textup{O}(\infty)$. This shows that
  $[\mathsf{h}_1]^\infty= [\mathsf{h}_2]^\infty$, hence proving the
  theorem.

\end{proof}

Theorem \ref{thm:well_defined} gives that the following notion is
well-defined.

\begin{defi}\label{defn:tso}
  The {\em Taylor series orbit} of the germ of $\is$, a neighborhood of a
  focus-focus fiber of multiplicity one, is the $K_4$-orbit 
  \[
  \mathcal{O}\is:=\mathcal{O}\left(\left[\mathsf{h}\left(\left(M',\omega',
          F'\right),\varphi\right)\right]^{\infty}\right) \subset
  \RM\formel{a,b}_\textup{ff},
  \]
  where $\left(\left(M',\omega', F'\right),\varphi\right)$ is any
  normalized neighborhood of a focus-focus fiber of multiplicity one
  constructed from $\is$ as in the proof of Lemma
  \ref{lemma:nice_rep}, and
  $\mathsf{h}\left(\left(M',\omega', F'\right),\varphi\right)$ is its
  regularized action.
\end{defi}



\subsubsection*{Independence of the representative of the isomorphism
  class of the germ}
The aim of this subsection is to show that, in fact, the Taylor series
orbit of the germ of $\is$ defines an invariant of
$\left[\is\right]$. This is the content of the following result.

\begin{prop}\label{prop:indep}
  For $i=1,2$, let $\left(M_i,\omega_i,F_i\right) \in
  \left[\is\right]$. Then
  $$\mathcal{O}\left(M_1,\omega_1,F_1\right) =
  \mathcal{O}\left(M_2,\omega_2,F_2\right).$$ 
\end{prop}
\begin{proof}
  Since, for $i=1,2$, $\left(M_i,\omega_i,F_i\right) \in
  \left[\is\right]$, there
  is no loss in generality in assuming that $\left(M_1,\omega_1,F_1\right)$
  and $\left(M_2,\omega_2,F_2\right)$ are strongly symplectically
  equivalent. Let $\left(\varphi,g\right)$ denote this equivalence and
  choose an Eliasson symplectomorphism $\left(\varphi_2,g_2\right)$
  for $\left(M_2,\omega_2,F_2\right)$. Arguing as above, there is no loss
  in generality in assuming that the domain of $\varphi_2$ intersects
  all fibers of $F_2$. Then
  $\left(\varphi_1,g_1\right):= \left(\varphi_2 \circ
    \varphi,g_2 \circ g\right)$ is an Eliasson symplectomorphism for
  $\left(M_1,\omega_1,F_1\right)$ with the property that the domain of
  $\varphi_1$ intersects all fibers of $F_1$. Moreover, for $i=1,2$, denote by
  $\left(\left(M_i,\omega_i, g_i \circ F_i\right),\varphi_i\right)$ the
  normalized neighborhood of a focus-focus fiber of multiplicity one
  obtained from $\left(M_i,\omega_i,F_i\right)$ and
  $\left(\varphi_i,g_i\right)$ as in the proof of Lemma
  \ref{lemma:nice_rep}. 

  First, observe that 
  \begin{equation*}
    \begin{split}
      \mathcal{O}\left(M_1,\omega_1,F_1\right)
      & =
      \mathcal{O}\left(\left[\mathsf{h}\left(\left(M_1,\omega_1,g_1
              \circ F_1\right),\varphi_1\right)\right]^{\infty}\right) \\
      & =
      \mathcal{O}\left(\left[\mathsf{h}\left(\left(M_1,\omega_1,g_2
              \circ \left(g
                \circ F_1\right)\right),\varphi_1\right)\right]^{\infty}\right) \\
      & = \mathcal{O}\left(M_1,\omega_1,g \circ F_1\right),
    \end{split}
  \end{equation*}
  where the last equality follows from the fact that
  $\left(\varphi_1,g_2\right)$ is, by construction, an Eliasson
  diffeomorphism for $\left(M_1,\omega_1,g \circ F_1\right)$. Thus
  it suffices to prove the result under the assumption that $g =
  \mathrm{id}$; under
  this assumption, $g_1 = g_2$. For $i=1,2$, set $\mathsf{h}_i:=\mathsf{h}\left(\left(M_i,\omega_i,g_i
      \circ F_i\right),\varphi_i\right)$ and let $\Sigma_i \to F_i(M)
  \smallsetminus \left\{c_{0,i}\right\}$ denote the bundle of periods
  of the subsystem of $\left(M_i,\omega_i, F_i\right)$ relative to
  $M_i \smallsetminus F^{-1}_i\left(c_{0,i}\right)$. Since
  $\left(\varphi,\mathrm{id}\right)$ is a strong symplectic
  equivalence between $\left(M_1,\omega_1, F_1\right)$ and
  $\left(M_2,\omega_2, F_2\right)$, $F_1(M_1) = F_2(M_2)$ and $\Sigma_1 = \Sigma_2$. Moreover,
  since $g_1 = g_2$, $g_1\left(F_1(M)\right) =
  g_2\left(F_2(M_2)\right)$ and $\left(g^{-1}_1\right)^* \Sigma_1 =
  \left(g^{-1}_2\right)^* \Sigma_2$. As above, set $l = \left\{(a,b)
    \mid a =0 , b \geq 0\right\}$, and fix the standard identification
  $\R^2 \cong \mathbb{C}$ and the choice of complex
  logarithm $\mathrm{Log} : \mathbb{C} \smallsetminus l \to
  \mathbb{C}$ determined by \eqref{eq:22}. Then, for $i=1,2$,
  \[
  \left(\left(g^{-1}_i\right)^*
    \Sigma_i\right)\bigg|_{g_i\left(F_i(M_i)\right) \smallsetminus l}
  = \Z \left\langle \DD a, \frac{1}{2\pi} \mathrm{Im}\left(\DD \left(
        w \, \mathrm{Log}\, w - w\right)\right) + \DD \mathsf{h}_i
  \right\rangle.
  \]
  Arguing as in the proof of Theorem \ref{thm:well_defined} and using
  the defining conditions \eqref{eq:23}, the above equalities imply
  that $\mathsf{h}_1 = \mathsf{h}_2$. Since, for $i=1,2$,
  $\mathcal{O}\left(\left[\mathsf{h}_i\right]^{\infty}\right) =
  \mathcal{O}\left(M_i,\omega_i,F_i\right)$,
  the desired equality follows.
\end{proof}

In light of Proposition \ref{prop:indep}, the following definition
makes sense.

\begin{defi}\label{defn:symp_invariant}
  The {\em Taylor series orbit} of $\left[\is\right]$, the isomorphism
  class of a germ of a neighborhood of a focus-focus fiber of
  multiplicity one, is the Taylor series orbit of any of its
  representatives, and is denoted by $\mathcal{O}\left[\is\right]$. 
\end{defi}

\subsubsection{The classification result}\label{sec:class-result}
The Taylor series orbit determines completely the isomorphism class of
the germ of a neighborhood of a focus-focus fiber of multiplicity
one. This is the content of the following result, which is a precised
version of the results in \cite{san_invt}, and whose proof, sketched
below, relies both on the arguments from [\emph{op. cit.}, Sections 5
and 6] and on the analysis of the $K_4$-action from the previous
paragraphs.
\begin{theo}\label{thm:taylor_sharp}
  The map
  \begin{equation*}
    \begin{split}
      \mathcal{G}_{\mathrm{ff}} &\to
      \RM\formel{a,b}_\textup{ff}/K_4 \\
      \left[\is\right] &\mapsto \mathcal{O}\left[\is\right]
    \end{split}
  \end{equation*}
  is a bijection. In other words, two neighborhoods of focus-focus
  fibers of multiplicity one have isomorphic germs if and only if
  their Taylor series orbits are equal. Moreover, for any
  $\mathcal{O}\left(\left[\mathsf{h}\right]^{\infty}\right) \in
  \RM\formel{a,b}_\textup{ff}/K_4 $,
  there exists a neighborhood of a focus-focus fiber of multiplicity
  one whose Taylor series orbit equals
  $\mathcal{O}\left(\left[\mathsf{h}\right]^{\infty}\right)$.
\end{theo}
\begin{proof}[Sketch of proof]
  The map in the statement is well-defined in light of Theorem
  \ref{thm:well_defined} and Proposition \ref{prop:indep}. To show
  that it is surjective, it suffices that, given any formal power
  series in two variables
  $\sum\limits_{i,j =0}^{\infty} t_{ij} X^i Y^j$ with $t_{00} = 0$ and
  $t_{10} \in \left[0,1\right[$, there exists a normalized
  neighborhood of a focus-focus fiber of multiplicity one whose
  regularized action has Taylor series at $(0,0)$ equal to
  $\sum\limits_{i,j =0}^{\infty} t_{ij} X^i Y^j$. This is proved in
  \cite[Section 6]{san_invt}.

  Thus it remains to show that the map is injective. Suppose that
  \begin{equation}
    \label{eq:32}
    \mathcal{O}\left[M_1,\omega_1,F_1\right] =
    \mathcal{O}\left[M_2,\omega_2,F_2\right].
  \end{equation}
  Then, without loss of generality, it may be assumed that, for
  $i=1,2$, $\left(M_i,\omega_i,F_i\right)$ is normalizable. Let
  $\left(\varphi_i,\mathrm{id}\right)$ be an Eliasson
  symplectomorphism for $\left(M_i,\omega_i,F_i\right)$, so that
  $\left(\left(M_i,\omega_i,F_i\right), \varphi_i\right)$ is
  normalized. If, for $i=1,2$, $\mathsf{h}_i$ denotes the regularized
  action of $\left(\left(M_i,\omega_i,F_i\right), \varphi_i\right)$,
  \eqref{eq:32} implies that there exists
  $j \in K_4 $ such
  that $\left[\mathsf{h}_1\right]^{\infty} =
  \left[\mathsf{h}_2\right]^{\infty} \star
  j$. Let
  $\left(A_j,g_j\right)$
  be the automorphism of the local model for a singular point of
  focus-focus type determined by \eqref{equ:A}. Then
  $\left(\left(M_1,\omega_1,g_j\circ F_1\right),
    \varphi_j \circ \varphi_1\right)$ is
  normalized and
  $\left[\left(M_1,\omega_1,g_j\circ
      F_1\right)\right] = \left[\left(M_1,\omega_1,
      F_1\right)\right]$. If $\bar{\mathsf{h}}_1$ denotes the
  regularized action of $\left(\left(M_1,\omega_1,g_{\epsilon_1,\epsilon_2}\circ F_1\right),
    \varphi_{\epsilon_1,\epsilon_2} \circ \varphi_1\right)$, Lemma
  \ref{lemm:K4-equivariant} implies that 
  $\left[\mathsf{h}_1\right]^{\infty} =
  \left[\bar{\mathsf{h}}_1\right]^{\infty} \star
  j$. It follows that $
  \left[\bar{\mathsf{h}}_1\right]^{\infty} =
  \left[\mathsf{h}_2\right]^{\infty}$. Therefore, without loss of
  generality, it may be assumed that the normalized
  $\left(\left(M_1,\omega_1,F_1\right), \varphi_1\right)$ and
  $\left(\left(M_2,\omega_2,F_2\right), \varphi_2\right)$ are such
  that $\left[\mathsf{h}_1\right]^{\infty} =
  \left[\mathsf{h}_2\right]^{\infty}$. This equality, together with
  the defining conditions \eqref{eq:23} imply that
  $\left(\mathsf{h}_1\right)^{\infty} =
  \left(\mathsf{h}_2\right)^{\infty}$, {\it i.e.} the regularized
  actions have equal Taylor series at the origin. This is precisely
  the case considered in \cite[Section 5]{san_invt}, which proves that
  $\left(M_1,\omega_1,F_1\right)$ and $\left(M_2,\omega_2,F_2\right)$
  have isomorphic germs. 
\end{proof}

\begin{rema}\label{rmk:higher_mult}
  The analog of Theorem \ref{thm:taylor_sharp} for focus-focus fibers
  with higher multiplicity is sketched in \cite[Section 7]{san_invt}
  (without taking into account the analog of the above
  $K_4$-action). It is worthwhile remarking that the classification is
  expected to be more involved in the presence of several focus-focus points (cf. the
  forthcoming \cite{pel-tang-ff}).
\end{rema}



\section{Semi-toric systems}\label{sec:lecture-3-semi}
The aim of this section is to describe properties of {\em semi-toric}
systems, which have been intensively studied in the last few years
from both the classical and quantum perspectives
(cf. \cite{san_polygon,pel_san_inv,pel_san_acta,pel_rat_san_affine,
  hss,hsss,pel_san_spin, san-lefloch-pelayo:jc} amongst others). Not
only are they, in some sense, the simplest non-trivial family of
almost-toric systems admitting focus-focus leaves, but also they are
extremely useful in physics: several physical systems can be modeled
using semi-toric systems
(cf. \cite{jaynes_cummings,cummings,sad_zhi}).

\subsection{Definition, Examples and First properties}
\label{sec:defin-exampl-first}
We begin with introducing semi-toric systems, first defined in
\cite{san_polygon}.

\begin{defi}\label{defn:semi-toric}
  An integrable system $\left(M,\omega, F= (J,H)\right)$ is said to be
  {\em semi-toric} if
  \begin{enumerate}[label=(S\arabic*), ref = (S\arabic*),
    leftmargin=*]
  \item \label{item:7} all its singular orbits are non-degenerate
    without hyperbolic blocks;
  \item \label{item:8} the first component $J$ is a proper moment map
    of an effective Hamiltonian $S^1$-action.
  \end{enumerate}
\end{defi}

Properness of $J$ above implies that of $F$, thus making semi-toric
systems into examples of almost-toric ones (see Definition
\ref{defn:at}). 


\begin{exm}\label{exm:st_as_semi-toric}
  Four dimensional toric integrable systems on closed manifolds (see Example
  \ref{exm:symp_toric}) are examples of semi-toric systems without
  focus-focus leaves. More generally, four-dimensional 
  toric integrable systems whose moment map has a proper first component are
  semi-toric.
\end{exm}

\begin{exm}\label{exm:coupled_spin}
  For all but two values of the parameter $t$ in the family of
  integrable systems given by the coupled angular momenta on
  $S^2 \times S^2$ (see Example \ref{exm:cou_ang_mom}), the
  corresponding integrable systems are semi-toric (cf. \cite[Example
  6.1]{san_polygon}). This family of examples can be adapted to
  construct examples of semi-toric systems on non-compact manifolds by
  linearizing one of the spheres at one of the poles: this gives rise
  to the so-called coupled spin oscillator on $S^2 \times \R^2$
  (cf. \cite[Example 6.2]{san_polygon}).
\end{exm}

\begin{rema}\label{rmk:ham_spaces}
  Let $\left(M,\omega,F=(J,H)\right)$ be a semi-toric system; the
  triple $\left(M,\omega,J\right)$ obtained by `forgetting' $H$
  encodes an effective Hamiltonian $S^1$-action on $(M,\omega)$ one of
  whose moment maps $J : (M,\omega) \to \R$ is proper. We say that
  $\left(M,\omega,J\right)$ {\em underlies}
  $\left(M,\omega,F=(J,H)\right)$. In this case, the fibers of $J$ are
  connected (cf. \cite{atiyah,gui_ste_conv,lmtw}); moreover, if
  $(M,\omega)$ is closed, a classification of such triples, known as
  {\em Hamiltonian $S^1$-spaces} is achieved in \cite{karshon}.
\end{rema}

The first fundamental result in the study of semi-toric system is a
{\em connectedness} result for the fibers of its moment map: this can
be seen as a generalization of connectedness of the fibers of
(four-dimensional) symplectic toric manifolds whose moment maps are
proper (cf. \cite{atiyah,gui_ste_conv,lmtw}). Its proof, which is
omitted here, uses connectedness of the fibers of the first component
(see Remark \ref{rmk:ham_spaces}), together with the control on the
singular fibers which arises from restricting the types of singular
orbits that can arise (see property \ref{item:7} in Definition
\ref{defn:semi-toric}).

\begin{theo}[Theorem 3.4 in \cite{san_polygon}]\label{thm:connected}
  The fibers of a semi-toric system are connected.
\end{theo}

An immediate consequence of Theorem \ref{thm:connected} is that the
leaf space of a semi-toric system is homeomorphic to its moment map
image endowed with the subspace topology (cf. \cite{hsss}). In fact, the
moment map image of a semi-toric system and the associated {\em
  bifurcation diagram} ({\it i.e.} the set of singular values of the
moment map) satisfy the following properties (cf. \cite[Proposition
2.9 and Theorem 3.4]{san_polygon}). Let $\left(M,\omega,F\right)$ be a
semi-toric system and set $B = F(M)$; then:
\begin{itemize}[leftmargin=*]
\item $B$ is contractible;
\item $\partial B \subset B$  consists of the image of purely
  elliptic orbits;
\item the set of {\em focus-focus values}, {\it i.e.} the image of
  focus-focus fibers, is discrete in $B$ and, hence, countable. Denote
  it by $B_{\mathrm{ff}} := \{c_i\}_{i \in I}$;
\item the subset of $B_{\mathrm{reg}} \subset B$ consisting of regular
  values equals $\mathrm{Int}(B) \smallsetminus B_{\mathrm{ff}}$. This
  implies that $B_{\mathrm{lt}}$, the locally toric leaf space of $\left(M,\omega,F\right)$ equals
  $B \smallsetminus B_{\mathrm{ff}}$.
\end{itemize}

In fact, slightly more is true (cf. \cite[Corollary 5.10]{san_polygon}
for a slightly different argument).

\begin{coro}\label{cor:finite_ff}
  Let $\left(M,\omega,F\right)$ be a semi-toric system and let
  $B_{\mathrm{ff}}$ denote its set of focus-focus values. Then the
  cardinality of $B_{\mathrm{ff}}$ is finite.
\end{coro}

\begin{proof}[Sketch of proof]
  The idea is to use the {\em Duistermaat-Heckman} measure associated
  to the $S^1$-action one of whose moment map is $J$
  (cf. \cite{dui_heck}). This is a non-negative function whose value
  at a point $x \in J(M)$ is the symplectic volume of the symplectic
  reduction $J^{-1}(x)/S^1$; in this case, it turns out to be a
  piecewise linear function, whose changes in slope depend on isolated
  singular points of the $S^1$-action and their {\em isotropy
    weights}, {\it i.e.} a pair of coprime integers which can be
  naturally associated to the linearized $S^1$-action near an isolated
  fixed point (cf. \cite[Section 3]{gls} and \cite[Lemma
  2.12]{karshon}). The result follows by observing that the isotropy
  weights of the $S^1$-action at the focus-focus points need be
  $\pm 1$ (cf. \cite[Theorem 1.2]{zung_another}), and the formula of
  \cite[Section 3]{gls} and \cite[Lemma 2.12]{karshon} for the
  Duistermaat-Heckman function implies that there can be at most
  finitely many such isolated fixed points.
\end{proof}

The above properties of semi-toric systems would justify the use of
strong symplectic equivalence as `the' notion of equivalence for
semi-toric systems; however, the property of being semi-toric is {\em
  not} invariant under such isomorphisms. Seeing as the $S^1$-action
plays an important role in establishing the basic properties of
semi-toric systems, the following stricter notion of equivalence is
considered.

\begin{defi}\label{defn:equiv_st}
  Two semi-toric systems $\left(M_1,\omega_1,F_1\right)$,
  $\left(M_2,\omega_2,F_2\right)$ are said to be {\em isomorphic} if
  they are strongly symplectically equivalent via a pair
  $\left(\varphi,g\right)$, where
  $g(x,y) = \left(x, g^{(2)}(x,y)\right)$ for some smooth function
  $g^{(2)} : B_1 \to \R$ satisfying $\frac{\partial g^{(2)}}{\partial y}>0$.
\end{defi}

\begin{rema}\label{rmk:iso}
  While the notion of isomorphism introduced in~\cite[Section
  2]{pel_san_inv} uses the more general condition
  $\frac{\partial g^{(2)}}{\partial y}\neq 0$, the rest of the
  discussion in that article implicitly assumes that $g$ preserves
  orientation, which \emph{in fine} coincides with our definition here
  (see, for instance, \cite[Section 4]{pel_san_acta} and
  \cite[Definition 1.5]{pel_rat_san_affine}).
\end{rema}

\subsection{Invariants of semi-toric
  systems}\label{sec:invar-semi-toric}

The aim of this section is to construct data associated to a semi-toric
system that is invariant under isomorphisms, and can therefore be
used to classify these systems (see Theorem
\ref{thm:classification}). Henceforth, any semi-toric system $\left(M,\omega, F=(J,H)\right)$ is
assumed to satisfy the following extra condition:

\begin{enumerate}[label=(S\arabic*), ref = (S\arabic*), start=3,
  leftmargin=*]
\item \label{item:10} There is at most one focus-focus point on any
  given level set of the function $J$.
\end{enumerate}

\begin{rema}\label{rmk:generic}
  Property \ref{item:10} is equivalent to demanding that any
  focus-focus fiber has multiplicity one and that any level set of $J$
  contains at most one focus-focus fiber. Many works in the literature
  require the former (however, cf.
  \cite{san_invt,san_polygon,pel_rat_san_affine,pel-tang-ff} for
  results without this restriction); in fact, \cite{zung-st1} states
  that this condition is `generic'. On the other hand, imposing the
  latter merely simplifies the exposition below.
\end{rema}

In what follows, fix a semi-toric system $\left(M,\omega,
  F\right)$ whose moment map image is denoted by $B$ and whose
set of focus-focus values is denoted by $B_{\mathrm{ff}}$. By
Corollary \ref{cor:finite_ff}, there exists a non-negative integer
$m_{\mathrm{ff}} \in  \Z_{\geq 0}$ which equals the cardinality of
$B_{\mathrm{ff}}$. Let $c_1 =\left(x_1,y_1\right), \ldots ,
c_{m_{\mathrm{ff}}} =
\left(x_{m_{\mathrm{ff}}},y_{m_{\mathrm{ff}}}\right)$ denote the
elements of $B_{\mathrm{ff}}$ ordered so that $i < j$ implies $x_i < x_j$.

\subsubsection{The number of focus-focus values and the Taylor
  series invariant}\label{sec:card-b_mathrmff-tayl}
The cardinality $m_{\mathrm{ff}}$ of $B_{\mathrm{ff}}$ is the simplest
datum that can be associated to $\left(M,\omega,
  F=(J,H)\right)$ and is invariant under isomorphism; it is
henceforth referred to as the {\em number of focus-focus
  points}.  \\

The next invariant of $\left(M,\omega,
  F = (J,H)\right)$ comes from a symplectic invariant associated to (the
isomorphism class of the germ at) each focus-focus fiber, and can be
constructed using the ideas of Section
\ref{sec:monodr-focus-focus}. For
each $i=1,\ldots,m_{\mathrm{ff}}$, let $V_i \subset
\mathrm{Int}\left(B\right)$ be an open, connected neighborhood of $c_i$ which
contains precisely one singular value of $F$. Then the subsystem of $\left(M,\omega,
  F\right)$ relative to $F^{-1}\left(V_i\right)$ is a neighborhood of
a focus-focus fiber of multiplicity one (see Definition
\ref{defn:sat_neigh_ff}). However, it has a special property: by
construction, the
restriction of $J$ to $F^{-1}\left(V_i\right)$ is the moment map of an
effective Hamiltonian $S^1$-action. 

\subsubsection*{Intermezzo: Vertical neighborhoods of focus-focus
  fibers of multiplicity one}

A subsystem of a semi-toric system is in general not semi-toric
anymore, because the restriction of $J$ to the subsystem need not be proper. However, it still retains the
fundamental property that its first component generates an effective
Hamiltonian $S^1$-action (provided that the subsystem is $S^1$-invariant). Therefore we take the liberty to employ
Definition~\ref{defn:equiv_st} for such subsystems. This remark is at
the heart of the recently introduced notion of {\em vertical
  almost-toric systems} in \cite{hsss}.

The terminology \emph{vertical neighborhood of a focus-focus fiber of
  multiplicity one} will hereinafter denote a neighborhood of a
focus-focus fiber of multiplicity one for which the first component of
the moment map is a moment map for an effective $S^1$ action. This is
the case, for instance, for a \emph{normalized} neighborhood of a
focus-focus fiber of multiplicity one, see
Definition~\ref{defn:normalized} and Lemma~\ref{lemma:norm}. From the
uniqueness (up to sign) of the $S^1$-action, see Corollary
\ref{cor:unique_sys_pres}, we obtain the following result.
\begin{lemm}
\label{lemm:vertical}
Let $(M,\omega,F=(J,H))$ be a semi-toric system. Then, any focus-focus
fiber of this system, with critical point $p_0$, admits a normalized
neighborhood of the form $(M'\subset M, \omega|_{M'}, g\circ F|_{M'})$
where the Eliasson diffeomorphism $g$ has the form
  \begin{equation}
    g\left(x,y\right) = \left(x - x_0 ,
      g^{(2)}\left(x,y\right)\right),
    \label{equ:vertical} 
  \end{equation}
  where $x_0=J(p_0)$, and $\frac{\partial g^{(2)}}{\partial y}>0$.
\end{lemm}
\begin{proof}
  Let $(\tilde\phy:U\to V,\tilde g=(\tilde g^{(1)}, \tilde g^{(2)}))$
  be the Eliasson isomorphism used in Lemma~\ref{lemma:nice_rep} to
  construct the normalized neighborhood.  Since
  $\tilde g^{(1)}(p_0)=0$, Corollary \ref{cor:unique_sys_pres} entails
  that there exists $\epsilon=\pm 1$ such that
  $\epsilon \tilde g^{(1)}\circ F=J-J(p_0)$, and hence
  $\tilde g(x,y)=(\epsilon x-\epsilon x_0, \tilde g^{(2)}(x,y))$ for
  all $(x,y)\in F(U)$.  Finally, using the $K_4$-action on
  $\mathrm{Aut}\left((\R^4,0),\omega_{\mathrm{can}},q\right)$ of
  Section \ref{sec:taylor-series-orbit}, we construct a new Eliasson
  isomorphism of the form
  $(\phy,g)=(A_j\circ\tilde\phy, g_j\circ \tilde g)$ for which both
  $\deriv{g^{(1)}}{x}>0$ and $\deriv{g^{(2)}}{y}>0$, which gives the
  desired result.
\end{proof}
If $g$ and $\tilde g$ are two such Eliasson diffeomorphisms, then
$\tilde g g^{-1}$ belongs to the group $G$ defined
in~\eqref{equ:g-group}, and, more precisely, because of the special
form of these diffeomorphisms, we have
$\tilde g g^{-1} \in G_0 =\ker \vec\epsilon \subset G$ (see Remark
\ref{rema:orientation-preserving}). Thus, if, in the discussion of
Section~\ref{sec:monodr-focus-focus}, we assume that all Eliasson
diffeomorphisms are of the form given by Lemma~\ref{lemm:vertical},
the action of the group $K_4$ has to be replaced by the action of the
subgroup of $K_4$ that preserves $G_0$, and this subgroup is simply
the identity. In particular, the Taylor series orbit of a vertical
neighborhood of a focus-focus fiber of multiplicity one reduces to a
single Taylor series, which is precisely the one obtained from
a regularized action of a normalized neighborhood of the form given by
Lemma~\ref{lemm:vertical}.

In order to adhere to the notation used in~\cite{pel_san_inv}, let
$\RM\formel{a,b}_0\subset \RM\formel{a,b}$ be the subset consisting of
formal power series whose constant term vanishes and whose
$a$-coefficient lies in $\left[0,1\right[$. (This set is in
natural bijection with $\RM\formel{a,b}_\textup{ff}$.)

The above discussion motivates introducing the following notion.
\begin{defi}\label{defn:taylor}
  Given the germ at the focus-focus fiber of
  $\left(M',\omega',F'\right)$, a vertical almost-toric neighborhood
  of a focus-focus fiber of multiplicity one, its {\em associated Taylor series}
  $ T \left(M',\omega', F'\right) \in  \R\llbracket a,b
  \rrbracket_{0}$ is the Taylor series at the origin of the
  regularized action of any normalized vertical almost-toric
  neighborhood with germ at the focus-focus fiber equal to that of $\left(M',\omega',F'\right)$.
\end{defi}
An important consequence is that Theorem~\ref{thm:taylor_sharp},
restated in the vertical category, asserts that this associated Taylor
series completely classifies a vertical neighborhood of the
focus-focus fiber, up to vertical isomorphisms.


\subsubsection*{The Taylor series invariant of a semi-toric system}
Using the above intermezzo, we can introduce the desired invariant of
semi-toric systems. 

\begin{defi}\label{defn:taylor_invariant_st}
  Let $(M,\omega,F)$ be a semi-toric system with $m_{\mathrm{ff}}$
  focus-focus points.  For each $i=1,\ldots, m_{\mathrm{ff}}$, the
  power series $T_i \in \R \llbracket a,b \rrbracket_0$ associated
  with the focus-focus critical value $c_i$ via
  Definition~\ref{defn:taylor} is said to be the {\em Taylor series
    invariant} at $c_i$ of $\left(M,\omega, F\right)$, while the
  ordered collection
  $\mathbf{T} := \left(T_i,\ldots,T_{m_{\mathrm{ff}}}\right) \in
  \left(\R \llbracket a,b \rrbracket_0\right)^{m_{\mathrm{ff}}}$
  is called the {\em Taylor series invariant} of
  $\left(M,\omega, F\right)$.
\end{defi}

The following result shows that the Taylor series invariant of
Definition \ref{defn:taylor_invariant_st} does not depend on the
choice of representative in the isomorphism class of a semi-toric
system. 

\begin{lemm}\label{lemma:necessary}
  If two semi-toric systems are isomorphic, their Taylor series
  invariants are equal.
\end{lemm}
\begin{proof}
  Suppose that $\is$ and $\left(M',\omega', F'\right)$ are isomorphic
  via $\left(\phy,g\right)$. Then they have an equal number of
  focus-focus points, {\it i.e.}
  $m_{\mathrm{ff}} = m_{\mathrm{ff}}' = m$. If this number is zero,
  there is nothing to prove. Suppose that $m \geq 1$. For
  $i=1,\ldots,m$, let $c_i$ and $c_i'$ denote the $i$th focus-focus
  value of $\is$ and $\left(M',\omega', F'\right)$ respectively,
  ordered as stated at the beginning of Section
  \ref{sec:invar-semi-toric}. Since $\left(\phy,g\right)$ is a strong
  symplectic equivalence, it sends focus-focus points to focus-focus
  points. This fact, together with the special form of $g$ (see
  Definition \ref{defn:equiv_st}), implies that for all
  $i=1,\ldots, m$, $g(c_i) = c'_i$. Fix $i=1,\ldots,m$.

  If $V_i \subset B_i$ denotes an open neighborhood of $c_i$
  containing no other singular value of $F$, then
  $V_i':=g\left(V_i\right)$ is an open neighborhood of $c_i'$
  satisfying an analogous property for $F'$. Moreover, the restriction
  of $\left(\phy,g\right)$ to
  $\left(F^{-1}\left(V_i\right),V_i\right)$ gives a vertical
  isomorphism between the subsystems of $\is$ and
  $\left(M',\omega',F'\right)$ relative to $F^{-1}\left(V_i\right)$
  and to $\left(F'\right)^{-1}\left(V'_i\right)$. 
  It follows that their Taylor series $T_i$ and $ T'_i$ are equal.
\end{proof}

\subsubsection{Cartographic invariant}\label{sec:st_polygons}
The last remaining invariant of semi-toric systems generalizes the
moment map image for (four dimensional) toric integrable systems with
proper moment map (cf. \cite[Proposition 6.3]{kar_ler}). Loosely
speaking, this invariant encodes the integral affine structure
$\mathcal{A}$ on 
the locally toric leaf space
$B_{\mathrm{lt}} = B \smallsetminus B_{\mathrm{ff}} \subset \R^2$
(cf.  Corollary
\ref{cor:ia_reg_ls} and Remark
\ref{rmk:purely_elliptic}). To make the above precise, we need to
recall the notion of {\em developing map} of an (integral)
affine structure\footnote{The discussion holds {\em mutatis
    mutandis} in the more general context of
  $\left(G,X\right)$-structures in the sense of \cite{smillie}.}
(cf. \cite{aus_mar} for further details).

\subsubsection*{Intermezzo: developing maps for (integral) affine manifolds}
Let $\left(N,\mathcal{A}\right)$ be an $n$-dimensional integral affine
manifold and let $\tilde{N}$ denote its universal cover. The universal
covering map $q : \tilde{N} \to N$ induces an integral affine
structure $\tilde{\mathcal{A}}$ on $\tilde{N}$ by pulling back
$\mathcal{A}$. Upon a choice
of basepoint $x_0 \in N$ and of an integral affine coordinate chart
$\chi_0$ defined near $x_0$, there is a local diffeomorphism
$\mathrm{dev}_{x_0,\chi_0} : \tilde{N} \to \R^n$ which is a {\em
  global} integral affine coordinate chart for $\tilde{\mathcal{A}}$
(cf. \cite{aus_mar} for the explicit construction). 
 
\begin{defi}\label{defi:dev_map}
  The map $\mathrm{dev}_{x_0,\chi_0} : \tilde{N} \to \R^n$ is called
  the {\em developing map} of $\left(N,\mathcal{A}\right)$ (relative
  to the choices of $x_0$ and of $\chi_0$).
\end{defi}

\begin{rema}\label{rmk:choices_dev_map}
  If $x'_0 \in N$ and $\chi'_0$ are different choices of
  basepoint and of integral affine coordinate map respectively, then
  there exists an element $\rho \in \mathrm{AGL}\left(n;\Z\right)$ such
  that
  $\mathrm{dev}_{x'_0, \chi'_0} = \rho \circ
  \mathrm{dev}_{x_0,\chi_0}$.
  Conversely, for any $\rho \in \mathrm{AGL}\left(n;\Z\right)$, the map
  $\rho \circ \mathrm{dev}_{x_0,\chi_0}$ is a developing map.
\end{rema}

Henceforth, fix choices $x_0 \in N$ and $\chi_0$ and, to simplify notation, denote the
resulting developing map by
$\mathrm{dev}$. The action of the fundamental group
$\pi_1(N) = \pi_1(N,x_0)$ on $\tilde{N}$ by deck transformations is
via integral affine diffeomorphisms, {\it i.e.} there is a group homomorphism
$\mathfrak{a} : \pi_1(N) \to \mathrm{AGL}\left(n;\Z\right)$ which,
for any $[\gamma] \in \pi_1(N)$, makes the following diagram commute
\begin{equation*}
  \xymatrix{\tilde{N} \ar[r]^-{\mathrm{dev}} \ar[d]_-{\cdot
      [\gamma]} & \R^n \\
    \tilde{N} \ar[r]_-{\mathrm{dev}} & \R^n \ar[u]_-{\mathfrak{a}\left([\gamma]\right)},}
\end{equation*}
where $\cdot [\gamma] : \tilde{N} \to \tilde{N}$ denotes the
diffeomorphism induced by acting by $[\gamma]$. 



\begin{defi}\label{defn:affine_holo}
  The homomorphism
  $\mathfrak{a} : \pi_1(N) \to \mathrm{AGL}\left(n;\Z\right)$ is the {\em affine holonomy}
  of $\left(N,\mathcal{A}\right)$. Its
  composite with the natural projection
  $\mathrm{Lin}: \mathrm{AGL}\left(n;\Z\right) \to \mathrm{GL}\left(n;\Z\right)$ is
  denoted by
  $\mathfrak{l} : \pi_1(N) \to \mathrm{GL}\left(n;\Z\right)$ and is
  the {\em linear holonomy} of
  $\left(N,\mathcal{A}\right)$.
\end{defi}

\begin{exm}\label{exm:ff_lin_hol}
  Unraveling the above constructions and the proof of Corollary
  \ref{cor:monodromy}, we obtain that equation \eqref{eq:11} is
  nothing but the calculation of the linear holonomy of the integral
  affine structure induced by an almost-toric system as in Section
  \ref{sec:monodr-focus-focus} near a focus-focus value. Exactness of
  the symplectic form in a neighborhood of the focus-focus fiber
  implies, in fact, that the linear and affine holonomies coincide for
  such integral affine structures.
\end{exm}

\begin{rema}\label{rmk:corners_dev}
  If $\left(N,\mathcal{A}\right)$ is an $n$-dimensional integral
  affine manifold with corners, the above discussion constructs a
  developing map and an affine holonomy representation for
  $\left(N,\mathcal{A}\right)$. The types of integral affine manifolds
  with corners that we deal with when studying semi-toric systems (or,
  more generally, almost-toric systems) are, in fact, special, for the
  facets and corners satisfy a {\em unimodularity} condition
  (cf. Definition \ref{defn:polygon} below). This condition can be described as
  follows: the image of any codimension $k = 1, 2$ face of $\tilde{N}$
  under the developing map is the intersection of $k$ linear
  hyperplanes in $\R^2$ whose normals can be chosen to span a
  $k$-dimensional unimodular sublattice of $\Z^2$, {\it i.e.}  the
  quotient of $\Z^2$ by the lattice has no torsion. This is a
  consequence of Theorem \ref{thm:toric} which provides integral
  affine coordinate maps near the boundary of the locally toric leaf
  space (cf. Remark \ref{rmk:purely_elliptic}).
\end{rema}

\subsubsection*{The cartographic invariant in the case $m_{\mathrm{ff}}=0$}\label{sec:cart-invar-case}
Going back to the semi-toric system $\left(M,\omega,F\right)$, the aim
is to encode the integral affine structure $\mathcal{A}$ on
$B_{\mathrm{lt}}$; intuitively, the idea is to construct a map defined
on the whole of $B$ which plays the role of the developing map for the
`{\em singular} integral affine structure on $B$' and whose image
satisfies familiar properties, which are recalled below.

\begin{defi}\label{defn:polygon}
  A {\em polygon} is a closed subset of $\R^2$
  whose boundary is a piece-wise linear curve with finitely many
  vertices contained in any compact subset of $\R^2$. Each linear
  piece of the boundary is called an {\em edge}, while a {\em vertex}
  is a point at which the boundary fails to be
  differentiable. A polygon is said to be
  \begin{itemize}[leftmargin=*]
  \item {\em convex} if it is the convex hull of isolated points in
    $\R^2$;
  \item {\em simple} if there are exactly two edges incident to any
    vertex;
  \item {\em rational} if the slope of any edge is a rational number.
  \end{itemize}
  A vertex $v$ of a simple, rational polygon is said to be {\em
    smooth} (or {\em unimodular}) if the normal vectors to the edges
  incident to $v$, each chosen so that its components are coprime,
  generate $\Z^2$. A simple, rational polygon is said to be {\em
    smooth} (or {\em unimodular}) if all its vertices are. The set of
  convex, simple, rational (and smooth) polygons is denoted by
  $\mathrm{RPol}\left(\R^2\right)$ (respectively $\mathrm{DPol}\left(\R^2\right)$).
\end{defi}

\begin{rema}\label{rmk:polys_act}
  Clearly, the following chain of inclusions holds
  $\mathrm{DPol}\left(\R^2\right) \subset \mathrm{RPol}
  \left(\R^2\right) \subset \mathrm{Pol}\left(\R^2\right)$, where
  $\mathrm{Pol}\left(\R^2\right)$ is the set of all polygons in $\R^2$. Moreover,
  the natural $\mathrm{AGL}\left(2;\Z\right)$-action on $\R^2$
  defines a $\mathrm{AGL}\left(2;\Z\right)$-action on
  $\mathrm{Pol}\left(\R^2\right)$ which leaves both
  $\mathrm{DPol}\left(\R^2\right)$ and $ \mathrm{RPol}
  \left(\R^2\right)$ invariant. 
\end{rema}

Suppose first that $m_{\mathrm{ff}}=0$; then the following result holds.

\begin{prop}\label{prop:toric_polygons_st}
  Let $\left(M,\omega, F=(J,H)\right)$ be a semi-toric system
  with $m_{\mathrm{ff}}=0$. There exists a smooth map
  $f : B = F(M) \to \R^2$ of the form
  $f(x,y) := \left(x,f^{(2)}(x,y)\right)$, where $f^{(2)}: B \to \R$
  is a smooth function with $\frac{\partial f^{(2)}}{\partial y} >0$, such that
  \begin{itemize}[leftmargin=*]
  \item $f$ is a diffeomorphism onto its image;
  \item the composite $f \circ F : \left(M,\omega\right) \to \R^2$ is
    the moment map of an effective Hamiltonian
    $\mathbb{T}^2$-action\footnote{Throughout we fix an identification
      $\mathrm{Lie}\left(\mathbb{T}^2 \right) \cong \R^2$.}.
  \end{itemize}
  In particular, $f(B)$ is a convex, rational, simple, smooth polygon,
  {\it i.e.} $f(B) \in \mathrm{DPol}\left(\R^2\right)$.
\end{prop}

\begin{proof}[Sketch of proof]
  The idea is to choose an appropriate developing map for the integral
  affine manifold $\left(B,\mathcal{A}\right)$. Since $B$ is
  contractible, upon a choice of basepoint and integral affine
  coordinate map, there
  exists a developing map $\mathrm{dev} : B \to \R^2$. By definition
  of the integral affine structure on $B$ (cf. Corollary
  \ref{cor:ia_reg_ls} and Remark \ref{rmk:purely_elliptic}) and by
  property \ref{item:8}, it is
  possible to fix the above choices so that
  $\mathrm{dev}(x,y) = \left(x,\mathrm{dev}^{(2)}(x,y)\right)$ and
  $\mathrm{dev}$ is orientation-preserving. Set
  $f:= \mathrm{dev}$; since $\mathrm{dev}$ is a local diffeomorphism,
  the above form implies that $\mathrm{dev}$ is
  injective. This proves the first item. By
  definition of the integral affine structure $\mathcal{A}$, $f \circ F$ is the moment map of an effective Hamiltonian
  $\mathbb{T}^2$-action. Moreover, the first component of $f \circ F$
  is equal to $J$, which is proper by assumption; therefore
  $f \circ F$ is proper, which implies that $f(B)$ is a convex,
  rational, simple, smooth polygon by \cite[Theorem 1.1]{lmtw} and
  \cite{delz}.
\end{proof}

\begin{defi}\label{defn:carto_toric}
  Given a semi-toric system $\left(M,\omega,F\right)$ with $m_{\mathrm{ff}}=0$, any map $f: B \to \R^2$ as in Proposition
  \ref{prop:toric_polygons_st} is said to be a {\em cartographic diffeomorphism}.
\end{defi}

\begin{rema}\label{rmk:rephrase}
  A consequence of Proposition \ref{prop:toric_polygons_st} is that
  any semi-toric system without focus-focus points is isomorphic to a toric integrable system.
\end{rema}

\begin{defi}\label{defn:dec_carto_no_ff}
  Given a semi-toric system $\left(M,\omega,F\right)$ with
  $m_{\mathrm{ff}}=0$ and a cartographic diffeomorphism $f$, the {\em
    decorated semi-toric polygon associated to $f$} is 
  $f\left(B\right) \in \times
  \mathrm{DPol}\left(\R^2\right)$.
\end{defi}

The cartographic invariant of $\left(M,\omega,F\right)$ with
$m_{\mathrm{ff}} = 0$ is the collection of {\em all} decorated
semi-toric polygons of {\em all} semi-toric systems isomorphic to
$\left(M,\omega,F\right)$. In fact, these can be described in terms of
some group actions; to this end, we
introduce the subgroup of $\mathrm{AGL}(2;\Z)$ of elements which are
orientation-preserving and fixes 
vertical lines, and denote it by $\mathcal{V}$. Explicitly, 
\[
\mathcal{V} = \left\{\left(
    \begin{pmatrix}
      1 & 0 \\
      k & 1
    \end{pmatrix},
    \begin{pmatrix}
      a \\
      b
    \end{pmatrix}\right) \bigg| \, k \in \ZM \text{ and } a,b \in \R
\right\}.
\]
Combining Remark \ref{rmk:choices_dev_map} with the proof of
Proposition \ref{prop:toric_polygons_st}, we obtain the following
characterization.

\begin{coro}\label{cor:family_no_ff}
  Let $\left(M,\omega, F\right)$ be a semi-toric system with
  $m_{\mathrm{ff}} = 0$, and let $f: B
  \to \R^2$ be a cartographic diffeomorphism. Then $f' : B \to \R^2$ is a
  cartographic diffeomorphism if and only if
  there exists $\rho \in \mathcal{V}$ such that
  $f'= \rho \circ f$.
\end{coro}


There is a natural
action of $\mathcal{V}$ on $
\mathrm{DPol}\left(\R^2\right)$, given by $H \cdot \Delta := H\left(\Delta\right)$. 
  

\begin{defi}\label{defn:carto_no_ff}
  Let $\left(M,\omega,F\right)$ be a semi-toric system with $m_{\mathrm{ff}}=0$. The $\mathcal{V}$-orbit of the decorated semi-toric polygon associated to
  any of its cartographic diffeomorphisms is called the {\em
  cartographic invariant} of $\left(M,\omega,F\right)$.
\end{defi}

\subsubsection*{Cartographic homeomorphisms in the case
  $m_{\mathrm{ff}} > 0$}

Our next aim is to generalize Proposition \ref{prop:toric_polygons_st}
to the case $m_{\mathrm{ff}} \geq 1$. Henceforth, fix
$m_{\mathrm{ff}} > 0$, which implies that $B_{\mathrm{lt}}$ is not
simply connected; the main insight of \cite{san_polygon} (which also
appears, with fewer details, in \cite{symington}) is that, by choosing
a suitable open, simply connected subset of $B_{\mathrm{lt}}$, it is
possible to define a homeomorphism of $B$ onto its image which
restricts to a developing map for the above chosen simply connected
subset of $B_{\mathrm{lt}}$! The idea behind constructing these simply
connected subsets comes from an understanding of the affine holonomy
of the integral affine structure near focus-focus fibers
(cf. Corollary \ref{cor:monodromy} and Example \ref{exm:ff_lin_hol}),
which, in suitable coordinates, leaves a vertical line invariant.

We show how to construct these subsets. For each $i =1,\ldots, m_{\mathrm{ff}}$,
choose a sign $\epsilon_i \in \{+1,-1\}$; this choice is henceforth encoded in
the vector
$\boldsymbol{\epsilon} =
\left(\epsilon_1,\ldots,\epsilon_{m_{\mathrm{ff}}}\right) \in
\{+1,-1\}^{m_{\mathrm{ff}}}$.
For each $i=1,\ldots,m_{\mathrm{ff}}$, consider the {\em (vertical)
  $\epsilon_i$-cut at $c_i$},
\[
l^{\epsilon_i} := \left\{ (x,y) \in B \mid x = x_i, \, \epsilon_i y
  \geq \epsilon_i y_i \right\},
\]
and set
$l^{\boldsymbol{\epsilon}} := \bigcup\limits_{i=1}^{m_{\mathrm{ff}}}
l^{\epsilon_i}$,
which denotes the union of all the cuts associated to
$\boldsymbol{\epsilon}$. The following result, stated below without
proof, characterizes the complement of the cuts.

\begin{lemm}\label{lemma:compl_cuts}
  For any $\boldsymbol{\epsilon} \in \{+1,-1\}^{m_{\mathrm{ff}}}$, the
  complement of the cuts $B \smallsetminus l^{\boldsymbol{\epsilon}}$
  is open and dense in $B$, and is simply connected.
\end{lemm}

For any
$\left(\mathsf{x},\mathsf{y}\right) \in B_{\mathrm{lt}}$ and any
choice $\boldsymbol{\epsilon} \in \{+1,-1\}^{m_{\mathrm{ff}}}$, define
\[
j_{\boldsymbol{\epsilon}} \left(\mathsf{x},\mathsf{y}\right) :=
\sum\limits_{\left\{ i = 1,\ldots, m_{\mathrm{ff}} \mid
    \left(\mathsf{x},\mathsf{y}\right) \in l^{\epsilon_i} \right\}}
\epsilon_i,
\]
where the convention is that, if
$\left(\mathsf{x},\mathsf{y}\right) \in B \smallsetminus
l^{\boldsymbol{\epsilon}}$,
then
$ j_{\boldsymbol{\epsilon}} \left(\mathsf{x},\mathsf{y}\right) = 0$.
The next result, whose proof is only sketched, establishes the
existence of the required `singular' developing maps, thus
generalizing Proposition \ref{prop:toric_polygons_st}
(cf. \cite{pel_rat_san_affine,san_polygon} for details).

\begin{theo}[Theorem 3.8 in \cite{san_polygon}]\label{thm:polygons}
  Let $\left(M,\omega,F\right)$ be a semi-toric system with
  $m_{\mathrm{ff}}>0$. For any
  $\boldsymbol{\epsilon} \in \{+1,-1\}^{m_{\mathrm{ff}}}$, there
  exists an orientation-preserving $f_{\boldsymbol{\epsilon}} : B \to \R^2$ of
  the form
  $f_{\boldsymbol{\epsilon}}(x,y) :=
  \left(x,f_{\boldsymbol{\epsilon}}^{(2)}(x,y)\right)$ such that
  \begin{itemize}[leftmargin=*]
  \item
    $f_{\boldsymbol{\epsilon}}|_{B \smallsetminus
      l^{\boldsymbol{\epsilon}}}$
    is a developing map for the restriction of $\mathcal{A}$ to
    $B \smallsetminus l^{\boldsymbol{\epsilon}}$;
  \item $f$ is a homeomorphism onto its image;
  \item for any
    $\left(\mathsf{x},\mathsf{y}\right) \in B_{\mathrm{lt}}$,
    \begin{equation}
      \label{eq:13}
      \lim\limits_{\stackrel{(x,y) \to (\mathsf{x},\mathsf{y})}{x <
          \mathsf{x}}} \DD f_{\boldsymbol{\epsilon}} (x,y) =
      \begin{pmatrix}
        1 & 0 \\
        j_{\boldsymbol{\epsilon}}(\mathsf{x},\mathsf{y})
        & 1
      \end{pmatrix}
      \lim\limits_{\stackrel{(x,y) \to (\mathsf{x},\mathsf{y})}{x >
          \mathsf{x}}} \DD f_{\boldsymbol{\epsilon}} (x,y).
    \end{equation}
  \end{itemize}
  In particular, $f_{\boldsymbol{\epsilon}}(B)$ is a convex, rational,
  simple polygon, {\it i.e.} $f_{\boldsymbol{\epsilon}}(B) \in \mathrm{RPol}\left(\R^2\right)$. Finally, $f_{\boldsymbol{\epsilon}}$ is unique up to
  composition on the left by an element of $\mathcal{V}$.
\end{theo}

\begin{proof}[Sketch of proof]
  Fix $\boldsymbol{\epsilon} \in \{+1,-1\}^{m_{\mathrm{ff}}}$, which,
  in turn, fixes the vertical cuts. Since
  $B \smallsetminus l^{\boldsymbol{\epsilon}}$ is open in
  $B_{\mathrm{lt}}$, it inherits an integral affine structure; since
  it is simply connected, this integral affine structure can be
  developed. Using the arguments of the proof of Proposition
  \ref{prop:toric_polygons_st}, the developing map can be taken to be
  orientation-preserving and
  of the form $(x,y) \mapsto \left(x, a(x,y)\right)$ for some smooth
  function $a$. The idea is to show that the above map extends to a
  continuous map on $B$ (which is unique, since
  $B \smallsetminus l^{\boldsymbol{\epsilon}}$ is dense), which is
  denoted by $f_{\boldsymbol{\epsilon}}$; to prove that
  $f_{\boldsymbol{\epsilon}}$ exists an understanding of the integral
  affine structure near the focus-focus values is needed (cf. Lemma
  \ref{lemma:period_ff}, as well as \cite[Step 4 of the proof of
  Theorem 3.8]{san_polygon} and \cite[Step 4 of the proof of Theorem
  B]{pel_rat_san_affine}). The map $f_{\boldsymbol{\epsilon}}$ has the
  required form and satisfies the first item; moreover, \cite[Step 4
  of the proof of Theorem B]{pel_rat_san_affine} shows that it is a
  homeomorphism onto its image. 
  Equation \eqref{eq:13} also follows from the way in
  which existence of $f_{\boldsymbol{\epsilon}}$ is shown
  (cf. \cite[Step 4 of the proof of Theorem 3.8]{san_polygon} for the
  case $\mathrm{sgn}(f_{\boldsymbol{\epsilon}}) = +1$ and \cite{hsss}
  for the general case). Equation \eqref{eq:13}, together with the
  definition of the integral affine structure on
  $B \smallsetminus l^{\boldsymbol{\epsilon}}$, implies that
  $f_{\boldsymbol{\epsilon}}(B)$ is a convex, rational, simple polygon
  (cf. \cite[Step 6 of the proof of Theorem 3.8]{san_polygon}). To
  complete the proof, observe that the above choice of developing map
  for the induced integral affine structure on
  $B \smallsetminus l^{\boldsymbol{\epsilon}}$ is unique up to
  composition on the left by an element of $\mathcal{V}$
  (cf. Corollary \ref{cor:family_no_ff}); since
  $B \smallsetminus l^{\boldsymbol{\epsilon}} \subset B$ is dense,
  this shows that $f_{\boldsymbol{\epsilon}}$ also satisfies the same
  property.
\end{proof}

\begin{rema}\label{rmk:not_smooth}
  It follows from the above proof that the only vertices of
  $f_{\boldsymbol{\epsilon}}(B)$ that may fail to be smooth in the
  sense of Definition \ref{defn:polygon} are the ones which belong to
  the image of vertical cuts. As a
  consequence, if a vertex $v$ of $f_{\boldsymbol{\epsilon}}(B)$ is not
  smooth, then we can conclude that there is a focus-focus value whose
  first coordinate equals that of $v$. Equation \eqref{eq:13} can be used to
  `measure' the failure of any such vertex to be smooth
  (cf. \cite[Lemma 2.28]{hss} for a precise statement). 
\end{rema}

\begin{rema}\label{rmk:true_more_genearl}
  The above sketch of proof uses assumption \ref{item:10}, for, in
  this case, for any choice of $\boldsymbol{\epsilon}$, $B
  \smallsetminus l^{\boldsymbol{\epsilon}}$ is simply connected. In
  fact, Theorem \ref{thm:polygons} holds for any semi-toric system
  (cf. \cite[Theorem 3.8]{san_polygon}, \cite[Theorem
  B]{pel_rat_san_affine} and \cite[Theorem 4.24]{hsss} for proofs in
  the more general case).
\end{rema}


\begin{defi}\label{defn:st_polygons}
  Given a semi-toric system $\left(M,\omega,F\right)$ with
  $m_{\mathrm{ff}} > 0$ and any
  $\boldsymbol{\epsilon} \in \{+1,-1\}^{m_{\mathrm{ff}}}$, a map
  $f_{\boldsymbol{\epsilon}} : B \to \R^2$ as in Theorem \ref{thm:polygons}
  is said to be
  a {\em cartographic homeomorphism} of
  $\left(M,\omega,F\right)$  (relative to
  $\boldsymbol{\epsilon}$). Its image
  $\Delta_{\boldsymbol{\epsilon}}:= f_{\boldsymbol{\epsilon}}(B)$ is a {\em semi-toric
    polygon} associated to $\left(M,\omega,F\right)$ (relative
  to $\boldsymbol{\epsilon}$).
\end{defi}

\begin{rema}\label{rmk:from_st_ham}
  Any semi-toric polygon associated to a semi-toric system defined on a
  closed manifold can be used to recover the invariants of the
  underlying Hamiltonian $S^1$-space (cf. Remark
  \ref{rmk:ham_spaces} and \cite{hss}). The forthcoming
  \cite{hsss_converse} shows that, given a Hamiltonian $S^1$-space 
  which satisfies some necessary conditions (called {\em thinness} in
  \cite{hsss_converse}), there exists a semi-toric system whose
  underlying Hamiltonian $S^1$-space is isomorphic to the original
  one. 
\end{rema}

Our next aim is to encode information of a cartographic homeomorphism
$f_{\boldsymbol{\epsilon}}$ so as to generalize Definition
\ref{defn:dec_carto_no_ff}. To this end, we first need to introduce an invariant of a cartographic homeomorphism (cf. \cite[Section
5.2]{pel_san_inv} for details) which relies on properties of the group
$\mathcal{V}$. 

\subsubsection*{Twisting indices of a cartographic homeomorphism} 
Let $\kappa: \mathcal{V} \to \Z$ be the homomorphism
obtained by composing the restriction of $\mathrm{Lin}: \mathrm{AGL}\left(2;\Z\right) \to \mathrm{GL}\left(2;\Z\right)$
with the isomorphism $\mathrm{Lin}\left(\mathcal{V}\right) \cong \Z$
given by  $\left( \begin{smallmatrix}
    1 & 0 \\
    k & 1
  \end{smallmatrix}\right) \mapsto k$. We refer to $\kappa$ as {\em twisting cocycle} of
$\mathcal{V}$ and  can be used to associate {\em
  twisting indices} to a cartographic homeomorphism as follows. 



As in Section
  \ref{sec:card-b_mathrmff-tayl}, for 
$i=1,\ldots,m_{\mathrm{ff}}$, let $V_i \subset \mathrm{Int}(B)$ denote
an open neighborhood of $c_i$ which contains precisely one singular
value. For each $i = 1, \ldots, m_{\mathrm{ff}}$, there is a
unique (up to sign) Hamiltonian vector
field $\ham{i}$ defined in $F^{-1}\left(V_i \smallsetminus l^{\epsilon_i} \right)$
which is `radial' (this is the vector field constructed in \cite[Step
2 of Section 5.2]{pel_san_inv}). 
In fact, there
exists a map $\nu_{\epsilon_i} : V_i \to \R^2$ which is a cartographic
homeomorphism relative to $\epsilon_i$ for the subsystem of
$\left(M,\omega,F\right)$ relative to $F^{-1}\left(V_i\right)$, such
that
\begin{itemize}[leftmargin=*]
\item $\ham{\nu^{(2)}_{\epsilon_i}} =
\ham{i}$, where $\nu_{\epsilon_i}(x,y) = \left(x,\nu_{\epsilon_i}^{(2)}(x,y)\right)$
(cf. \cite[Lemma 5.6]{pel_san_inv}). 
\end{itemize}

\begin{defi}\label{defn:priv_cart}
  For each $i=1,\ldots, m_{\mathrm{ff}}$, the map $\nu_{\epsilon_i}$ is a {\em privileged} cartographic
  homeomorphism for $c_i$ relative to
  $\epsilon_i$. The collection $\boldsymbol{\nu}_{\boldsymbol{\epsilon}} :=
  (\nu_{\epsilon_1},\ldots,\nu_{\epsilon_{m_{\mathrm{ff}}}})$
  is a choice of {\em privileged} cartographic
  homeomorphisms for $\left(c_1,\ldots,c_{m_{\mathrm{ff}}}\right)$ relative to $\boldsymbol{\epsilon} = \left(\epsilon_1,\ldots,\epsilon_{m_{\mathrm{ff}}}\right)$.
\end{defi}

\begin{rema}\label{rmk:priv}
  For given $i = 1,\ldots, m_{\mathrm{ff}}$, $\epsilon_i$ and
  privileged cartographic homeomorphism $\nu_i$, $\nu_i'$ is a
  privileged cartographic homeomorphism for $c_i$ relative to
  $\epsilon_i$ if and only if there exists $\tau_i \in \mathcal{T}$ with $\nu_{\epsilon_i}' =
  \tau_i\circ \nu_{\epsilon_i}$, where $\mathcal{T} \subset \mathcal{V}$
  is the subgroup of vertical translations.
\end{rema}

Fix a
cartographic homeomorphism $f_{\boldsymbol{\epsilon}}$ and privileged
cartographic homeomorphisms $\boldsymbol{\nu}_{\epsilon}$. Since $f_{\boldsymbol{\epsilon}}|_{V_i}$
and $\nu_{\epsilon_i}$ are cartographic homeomorphisms relative to $\epsilon_i$ for the subsystem of
$\left(M,\omega,F\right)$ relative to $F^{-1}\left(V_i\right)$, there exists $\rho_i\left(f_{\boldsymbol{\epsilon}},\nu_{\epsilon_i}\right) \in \mathcal{V}$ such that
$f_{\boldsymbol{\epsilon}}|_{V_i} =
\rho_i\left(f_{\boldsymbol{\epsilon}},\nu_{\epsilon_i}\right) \circ
\nu_{\epsilon_i}$. (Observe that $\rho_i\left(f_{\boldsymbol{\epsilon}},\nu_{\epsilon_i}\right)$
does not depend on the choice of neighborhood $V_i$.)

\begin{defi}\label{defn:twist_index_ff}
  The {\em twisting index} of the cartographic homeomorphism
  $f_{\boldsymbol{\epsilon}}$ with respect to the privileged cartographic
  homeomorphisms $\boldsymbol{\nu}_{\boldsymbol{\epsilon}}$ is
  $$\boldsymbol{\kappa}\left(f_{\boldsymbol{\epsilon}},\boldsymbol{\nu}_{\boldsymbol{\epsilon}}\right):=\left(\kappa\left(\rho_1\left(f_{\boldsymbol{\epsilon}},\nu_{\epsilon_1}\right)\right),\ldots,
  \kappa\left(\rho_{m_{\mathrm{ff}}}\left(f_{\boldsymbol{\epsilon}},\nu_{\epsilon_{m_{\mathrm{ff}}}}\right)\right)\right)\in \Z^{m_{\mathrm{ff}}}.$$
\end{defi}

In fact, the twisting index of
any cartographic homeomorphism relative to $\boldsymbol{\epsilon}$ is independent
of the choice of privileged cartographic homeomorphisms relative to
$\boldsymbol{\epsilon}$. 

\begin{coro}\label{cor:indep}
  For $\boldsymbol{\epsilon} \in
  \left\{+1,-1\right\}^{m_{\mathrm{ff}}}$, let $f_{\boldsymbol{\epsilon}}$ be a cartographic homeomorphism and
  let
  $\boldsymbol{\nu}_{\boldsymbol{\epsilon}},\boldsymbol{\nu}'_{\boldsymbol{\epsilon}}$
  be privileged cartographic homeomorphisms relative to
  $\boldsymbol{\epsilon}$. Then
  $\boldsymbol{\kappa}\left(f_{\boldsymbol{\epsilon}},\boldsymbol{\nu}_{\boldsymbol{\epsilon}}\right)
  = \boldsymbol{\kappa}\left(f_{\boldsymbol{\epsilon}},\boldsymbol{\nu}'_{\boldsymbol{\epsilon}}\right)$.
\end{coro}
\begin{proof}
  It suffices to check that, for each $i=1,\ldots,m_{\mathrm{ff}}$,
  $\kappa\left(\rho_i\left(f_{\boldsymbol{\epsilon}},\nu_{\epsilon_i}\right)\right)
  =
  \kappa\left(\rho_i\left(f_{\boldsymbol{\epsilon}},\nu'_{\epsilon_i}\right)\right)$. By
  Remark \ref{rmk:priv}, there exists $\tau_i \in \mathcal{T}$ such
  that $\rho_i\left(f_{\boldsymbol{\epsilon}},\nu'_{\epsilon_i}\right) =
  \tau_i \circ
  \rho_i\left(f_{\boldsymbol{\epsilon}},\nu_{\epsilon_i}\right)$. Then,
  using that $\kappa$ is a homomorphism, 
  \[
  \kappa\left(
    \rho_i\left(f_{\boldsymbol{\epsilon}},\nu'_{\epsilon_i}\right)\right)
  = \kappa(\tau_i) +
  \kappa\left(\rho_i\left(f_{\boldsymbol{\epsilon}},\nu_{\epsilon_i}\right)\right)
  =
  \kappa\left(\rho_i\left(f_{\boldsymbol{\epsilon}},\nu_{\epsilon_i}\right)\right),
  \]
  since $\kappa|_{\mathcal{T}} \equiv 0 $.
\end{proof}

In light of Corollary \ref{cor:indep}, the following notion makes sense.

\begin{defi}\label{defn:twi_carto}
  Given a cartographic homeomorphism $f_{\boldsymbol{\epsilon}}$
  relative to $\boldsymbol{\epsilon}\in
  \left\{+1,-1\right\}^{m_{\mathrm{ff}}}$, the collection of {\em twisting indices} of
  $f_{\boldsymbol{\epsilon}}$ is
  $\boldsymbol{\kappa}\left(f_{\boldsymbol{\epsilon}}\right):=
  \boldsymbol{\kappa}\left(f_{\boldsymbol{\epsilon}},\boldsymbol{\nu}_{\boldsymbol{\epsilon}}\right)$,
  where $\nu_{\boldsymbol{\epsilon}}$ is any choice of privileged
  cartographic homeomorphisms relative to $\boldsymbol{\epsilon}$. 
\end{defi}

\subsubsection*{Decorated semi-toric polygons and cartographic
  invariants in the case $m_{\mathrm{ff}} > 0$}
With the twisting indices of a cartographic homeomorphism at hand, it
is possible to define an invariant of a cartographic homeomorphism,
which generalizes Definition \ref{defn:dec_carto_no_ff}.

\begin{defi}\label{defn:dec_st_poly}
  Given a semi-toric system $\left(M,\omega,F\right)$ with
  $m_{\mathrm{ff}} \geq 1$ and a cartographic homeomorphism
  $f_{\boldsymbol{\epsilon}} : B \to \R^2$ relative to
  $\boldsymbol{\epsilon}$, the {\em decorated semi-toric polygon
    associated to $f_{\boldsymbol{\epsilon}}$} is
  \[
  \left(\boldsymbol{\epsilon},
    f_{\boldsymbol{\epsilon}}\left(B\right), \left(\left(\mathsf{c}_1,
        \kappa_1\left(f_{\boldsymbol{\epsilon}}\right)\right),\ldots,\left(\mathsf{c}_{m_{\mathrm{ff}}},\kappa_{m_{\mathrm{ff}}}\left(f_{\boldsymbol{\epsilon}}\right)\right)\right)
  \right)
  \]
  where for each $i=1,\ldots,m_{\mathrm{ff}}$,
  $\mathsf{c}_i = f_{\boldsymbol{\epsilon}}(c_i)$, and
  $\boldsymbol{\kappa}\left(f_{\boldsymbol{\epsilon}}\right)= (\kappa
  _1\left(f_{\boldsymbol{\epsilon}}\right),\ldots, \kappa
  _{m_{\mathrm{ff}}}\left(f_{\boldsymbol{\epsilon}}\right))$
  is as in Definition \ref{defn:twi_carto}.
\end{defi}
By Theorem \ref{thm:polygons}, the decorated semi-toric polygon
associated to a cartographic homeomorphism is an element of
$\{+1,-1\}^{m_{\mathrm{ff}}} \times \mathrm{RPol}\left(\R^2\right)
\times \left(\R^2 \times \Z\right)^{m_{\mathrm{ff}}}$,
and it should be clear how Definition \ref{defn:dec_st_poly}
generalizes Definition \ref{defn:dec_carto_no_ff}.

To make the above notion of decorated semi-toric polygons into an
invariant of the isomorphism class of $\left(M,\omega,F\right)$, the
idea is to consider the possible decorated semi-toric polygons of {\em
  all} cartographic homeomorphisms of {\em all} semi-toric systems
isomorphic to $\left(M,\omega,F\right)$. To describe this family
explicitly, we view a decorated semi-toric polygon as an element of
$$ \{+1,-1\}^{m_{\mathrm{ff}}} \times
\mathrm{Pol}\left(\R^2\right) \times \left(\R^2 \times
  \Z\right)^{m_{\mathrm{ff}}},$$
\noindent
and define a
$\mathcal{V} \times \left\{+1,-1\right\}^{m_{\mathrm{ff}}}$-action on
this set as follows. First, we define a $\mathcal{V}$-action on
$ \{+1,-1\}^{m_{\mathrm{ff}}} \times \mathrm{Pol}\left(\R^2\right)
\times \left(\R^2 \times \Z\right)^{m_{\mathrm{ff}}}$ by setting
\begin{equation}
  \label{eq:14}
  \begin{split}
    \rho \cdot (\boldsymbol{\epsilon},    \Delta,&((\mathsf{c}_1,\kappa_1)\ldots,(\mathsf{c}_{m_{\mathrm{ff}}},\kappa_{\mathrm{m_{\mathrm{ff}}}})
    ) : = (\boldsymbol{\epsilon},
      \rho(\Delta), \\
      & ((\rho(\mathsf{c}_1), \kappa(\rho) +
          \kappa_1),\ldots,(\rho(\mathsf{c}_{m_{\mathrm{ff}}}),\kappa(\rho) +
          \kappa_{m_{\mathrm{ff}}}))),
    \end{split}
\end{equation}
for $\rho \in \times \mathcal{V}$,
$\left(\boldsymbol{\epsilon}, \Delta,
  \left(\left(\mathsf{c}_1,\kappa_1\right),\ldots,\left(\mathsf{c}_{m_{\mathrm{ff}}},\kappa_{m_{\mathrm{ff}}}\right)\right)
\right) \in \{+1,-1\}^{m_{\mathrm{ff}}} \times
\mathrm{Pol}\left(\R^2\right) \times \left(\R^2 \times
  \Z\right)^{m_{\mathrm{ff}}}$.
Second, we define a $\left\{+1,-1\right\}^{m_{\mathrm{ff}}}$-action
which illustrates the fact that semi-toric polygons corresponding to
different signs are related by {\em piece-wise} integral affine
diffeomorphisms (cf. \cite[Section 4]{san_polygon} for details). For
any $x_0 \in \R$, let $\mathsf{l}_{x_0} : \R^2 \to \R^2$ be the
piece-wise integral affine diffeomorphism which is the identity on
$\left\{(x,y) \mid x \leq x_0\right\}$ and acts as the shear
$ \left(\begin{smallmatrix}
    1 & 0 \\
    1 & 1
  \end{smallmatrix}\right)$ on $\left\{(x,y)
  \mid x > x_0\right\}$. Fix
$\mathbf{x}=\left(x_1,\ldots,x_{N}\right) \in \R^{N}$; for any
$\mathbf{k} = \left(k_1,\ldots,k_N\right) \in \Z^{N}$, set 
$\mathsf{l}^{\mathbf{k}}_{\mathbf{x}} := \mathsf{l}^{k_1}_{x_1}
\circ \ldots \circ \mathsf{l}^{k_N}_{x_N}$. These transformations
allow to state the following lemma, whose proof is left to the reader.

\begin{lemm}\label{lemma:action}
  For any fixed integer $m_{\mathrm{ff}} \geq 1$, the formula
  \begin{gather}
    \label{eq:15}
    \boldsymbol{\epsilon}' \star \left(\boldsymbol{\epsilon},
      \Delta,
      \left(\left(\mathsf{c}_1,\kappa_1\right),\ldots,\left(\mathsf{c}_{m_{\mathrm{ff}}},\kappa_{m_{\mathrm{ff}}}\right)\right)
    \right) : =\\
    \left(\boldsymbol{\epsilon}'\bullet \boldsymbol{\epsilon},
      \mathsf{l}^{\mathbf{k}\left(\boldsymbol{\epsilon},\boldsymbol{\epsilon}'\right)}_{\mathbf{x}}\left(\Delta\right),
      \left(\left(\mathsf{c}_1
          \kappa_1\right),\ldots,\left(\mathsf{c}_{m_{\mathrm{ff}}},\kappa_{m_{\mathrm{ff}}}\right)\right)
    \right),
  \end{gather}
  where
  $\boldsymbol{\epsilon}'\bullet\boldsymbol{\epsilon} :=
  \left(\epsilon_1'\epsilon_1,\ldots,
    \epsilon_{m_{\mathrm{ff}}}'\epsilon_{m_{\mathrm{ff}}}\right)$,
  $\mathbf{k}\left(\boldsymbol{\epsilon},\boldsymbol{\epsilon}'\right):=
  \boldsymbol{\epsilon}\bullet \frac{\boldsymbol{1} -
    \boldsymbol{\epsilon}'}{2}$,
  and $\mathbf{x}$ is the obtained by taking the first components of
  $\mathsf{c}_1,\ldots,\mathsf{c}_{m_{\mathrm{ff}}}$, defines an
  action of $\{+1,-1\}^{m_{\mathrm{ff}}}$ on
  $\{+1,-1\}^{m_{\mathrm{ff}}} \times \mathrm{Pol}\left(\R^2\right)
  \times \left(\R^2 \times \Z\right)^{m_{\mathrm{ff}}}$
  which commutes with the $\mathcal{V}$-action given by equation
  \eqref{eq:14}.
\end{lemm}

\begin{rema}\label{rmk:why_not_rat}
  The above
  $\left\{+1,-1\right\}^{m_{\mathrm{ff}}}$-action does not leave $ \{+1,-1\}^{m_{\mathrm{ff}}} \times
  \mathrm{RPol}\left(\R^2\right) \times
  \left(\R^2\right)^{m_{\mathrm{ff}}}$ invariant, as not all elements of a given
  orbit are necessarily convex (cf. \cite[Section
  2.2]{pel_san_acta}). On the other hand, the $\mathcal{V}$-action of
  equation \eqref{eq:14} restricts to an action on $ \{+1,-1\}^{m_{\mathrm{ff}}} \times
  \mathrm{RPol}\left(\R^2\right) \times
  \left(\R^2\right)^{m_{\mathrm{ff}}}$.
\end{rema}

Lemma \ref{lemma:action} yields an action of $ \mathcal{V}
\times \left\{+1,-1\right\}^{m_{\mathrm{ff}}}$ on $\{+1,-1\}^{m_{\mathrm{ff}}} \times
\mathrm{Pol}\left(\R^2\right) \times
\left(\R^2 \times \Z\right)^{m_{\mathrm{ff}}}$, which can be used to
introduce the following data attached to a semi-toric system (which is
well-defined by \cite[Proposition 4.1]{san_polygon} and
\cite[Lemma 5.1 and Proposition 5.8]{pel_san_inv}). 

\begin{defi}\label{defn:carto_invt}
  Let $\left(M,\omega,F\right)$ be a semi-toric system with
  $m_{\mathrm{ff}} \geq 1$ focus-focus values. The {\em cartographic invariant} of
  $\left(M,\omega,F\right)$ is the $\mathcal{V} \times
  \left\{+1,-1 \right\}^{m_{\mathrm{ff}}}$-orbit of the decorated
  semi-toric polygon associated to any cartographic homeomorphism.
\end{defi}

\begin{rema}\label{rmk:contrast}
  The cartographic invariant encodes three of the invariants of
  semi-toric systems as defined in \cite{pel_san_inv}, namely the {\em
    semi-toric polygon invariant}, the {\em volume invariant} and the
  {\em twisting index invariant} (cf.  \cite[Definition 4.5,
  Definition 5.2 and Definition
  5.9]{pel_san_inv}). 
\end{rema}

The cartographic invariant of a semi-toric polygon is an invariant of
its isomorphism class by \cite[Lemmata 4.6, 5.3 and 5.10]{pel_san_inv}.

\begin{coro}\label{cor:poly_invt}
  If two semi-toric systems are isomorphic, then they have equal
  cartographic invariants.
\end{coro}


\subsection{Classification of semi-toric systems}\label{sec:classification}

Section \ref{sec:invar-semi-toric} provides all invariants
of semi-toric systems needed to determine their isomorphism type
(cf. Theorem \ref{thm:classification}).

\begin{defi}\label{defn:complete}
  Given a semi-toric system $\left(M,\omega,F=(J,H)\right)$, its {\em
    complete set of invariants} is given by
  \begin{enumerate}[label=(\arabic*), ref=(\arabic*), leftmargin=*]
  \item the number of focus-focus values;
  \item its Taylor series invariant 
    (cf. Definition \ref{defn:taylor_invariant_st});
  \item its cartographic invariant 
    (cf. Definition \ref{defn:carto_invt}).
  \end{enumerate}
\end{defi}

The main result of \cite{pel_san_inv}, stated below without proof,
shows that the invariants of Definition \ref{defn:complete} completely
determine a semi-toric system up to isomorphism (cf. \cite[Theorem
6.2]{pel_san_inv}).

\begin{theo}\label{thm:classification}
  Two semi-toric systems are isomorphic if and only if their complete
  sets of invariants are equal.
\end{theo}

To complete the classification of semi-toric systems, it suffices to
determine which abstract data corresponds to the complete set of
invariants of some semi-toric system. Without going into details, the
possible cartographic invariants with twisting indices are restricted:
for instance, the elements of $\left(\R^2\right)^{m_{\mathrm{ff}}}$
appearing in any such invariant lie in the interior of the
corresponding polygon\footnote{This corresponds to the fact that the
  {\em volume invariants} of \cite{pel_san_inv,pel_san_acta} are
  necessarily positive.}. Moreover, the polygon {\em and} the elements
of $\left(\R^2\right)^{m_{\mathrm{ff}}}$ are constrained
(cf. \cite[Section 4]{pel_san_acta} for details). Being loose, we
refer to any
$ \mathcal{V} \times \left\{+1,-1\right\}^{m_{\mathrm{ff}}}$- orbit in
$\{+1,-1\}^{m_{\mathrm{ff}}} \times \mathrm{Pol}\left(\R^2\right)
\times \left(\R^2 \times \Z\right)^{m_{\mathrm{ff}}}$
which satisfies the necessary conditions of \cite{pel_san_acta} as
being {\em admissible}. Thus we can state the last result of this
section, which completes the classification of semi-toric systems
(cf. \cite[Theorem 4.6]{pel_san_acta}).

\begin{theo}\label{thm:real}
  For any non-negative integer $m_{\mathrm{ff}}$, any element of $ \left(\R
    \llbracket a,b \rrbracket_0 \right)^{m_{\mathrm{ff}}}$ and any
  admissible $\mathcal{V}
  \times \left\{+1,-1\right\}^{m_{\mathrm{ff}}}$- orbit in $\{+1,-1\}^{m_{\mathrm{ff}}} \times
  \mathrm{Pol}\left(\R^2\right) \times
  \left(\R^2 \times \Z\right)^{m_{\mathrm{ff}}}$, there exists a
  semi-toric system whose complete set of invariants equals the
  above data.
\end{theo}


\section{Quantum systems and the inverse problem}
\label{sec:quantum}

\subsection{The joint spectrum}

Recall from Definition~\ref{defn:ci} that a classical completely
integrable system is the data of $n$ Poisson-commuting independent
functions $f_1,\dots,f_n$ on a $2n$-dimensional symplectic
manifold. These functions are typically `classical quantities',
{\it e.g.} energy, angular momentum, etc. In the quantum world, observables
are linear operators acting on a Hilbert space, and usually bear the
same name (quantum energy, quantum angular momentum, etc.). According
to the correspondence principle, that we discuss below, the Poisson
bracket is the classical limit of the operator bracket. This leads to
the following general definition of a quantum integrable system.
\begin{defi}
  \label{defi:quantum-ci}
  Let $\mathcal{H}$ be a Hilbert space `quantizing' the symplectic
  manifold $M^{2n}$. A {\em quantum completely integrable system} is
  an $n$-tuple $\left(T_1,\ldots,T_n\right)$ pairwise commuting, `independent' selfadjoint operators
  acting on $\mathcal{H}$:
  \begin{equation}
    \forall i,j, \qquad [T_j,T_j] =0.\label{equ:quantum-commute}
  \end{equation}
\end{defi}
In case the $T_j$ are not necessarily bounded, the commutation
property~\eqref{equ:quantum-commute} is taken in the strong sense: the
spectral measure (obtained via the spectral theorem as a
projector-valued measure) of $T_i$ and $T_j$ commute.

We need, of course, to define the terms in quotation marks, see
Section \ref{sec:corr-princ-semicl} below. Assuming this, we can
introduce the most important object for us, which is the quantum
analogue of the image of the moment map $F=(f_1,\dots,f_n)$: namely
the \emph{joint spectrum} of $(T_1,\dots, T_n)$.

Recall that the point spectrum of a (possibly unbounded) operator $T$
acting on a Hilbert space $\mathcal{H}$ is the set
\[
\sigma_{p}(T) = \{ \lambda\in\CM \quad | \quad \ker (T-\lambda I) \neq
\{0\}. \}
\]
An element of $\sigma_p$ is called an eigenvalue of $T$. More
generally, the spectrum $\sigma(T)$ of $T$ is by definition the set of
$\lambda\in\CM$ such that $(T-\lambda I)$ does not admit a bounded
inverse; thus it contains $\sigma_p$. The spectrum of $T$ is called
\emph{discrete} when it consists of isolated eigenvalues of finite
algebraic multiplicity.

\begin{defi}
  \label{defi:joint-spectrum}
  Let $T_1,\dots, T_n$ be pairwise commuting operators; the
  \emph{(discrete) joint spectrum} of $\left(T_1,\ldots,T_n\right)$ is the set of simultaneous
  eigenvalues of the operators $T_j$, $j=1,\ldots,n$, \emph{i.e.}
  \[
  \Sigma(T_1,\dots, T_n) := \{(\lambda_1,\dots,\lambda_n)\in\RM^n |
  \quad \bigcap_{j=1}^n \ker(T_j-\lambda_j I) \neq \{0\} \}.
  \]
\end{defi}
A more general, not necessarily discrete, notion of joint spectrum can
be obtained by considering the support of the joint spectral measure;
see~\cite{charbonnel}. In this text we only consider the discrete
case, which physically speaking corresponds to the existence of common
localized quantum states for $T_1,\dots,T_n$.

Our goal is to relate $F(M)$ and $\Sigma(T_1,\dots, T_n)$. Part of the
question is hidden in the meaning of `$\mathcal{H}$ quantizing $M$',
but it is not limited to this. The operators $T_j$ themselves must possess a good
semiclassical limit.

\subsection{The correspondence principle and the semiclassical limit}
\label{sec:corr-princ-semicl}

What is the relation between the symplectic manifold $(M,\omega)$ and
the Hilbert space $\mathcal{H}$? How can a quantum observable (an operator
$T$) correspond to a classical observable (a function
$f\in\Cinf(M)$)? The way to go from the classical setting to the
quantum setting is called \emph{quantization} by mathematicians; the
other direction, which often makes more sense from the point of view
of quantum mechanics, is called the \emph{semiclassical limit}. It is
out of the scope of this text to explain the various mathematical
answers to these questions. We instead try to convey the general
idea, and then propose a simple axiomatization that will be enough to
prove some non-trivial results concerning the relationship between the
joint spectrum and the image of the moment map.

The traditional, naive, approach to the correspondence bewteen
classical and quantum mechanics is to consider polynomials in
canonical (Darboux) coordinates. Assume that, in some open set of $M$,
$\omega=\sum_{j=1}^n d\xi_j\wedge dx_j$. Set
$\xi=(\xi_1,\dots,\xi_n)\in\RM^n$ (momentum coordinates) and
$x=(x_1,\dots, x_n)\in\RM^n$ (position coordinates). The corresponding
`quantizing' Hilbert space is the space $L^2(\RM^n)$ of
square-integrable functions in the $x$ variable. The Dirac rule
asserts on the one hand that the `quantum position operator'
associated with the variable $x_j$ is the operator of multiplication
by $x_j$:
\begin{equation}
  L^2(\RM^n) \ni u \mapsto x_j u \in L^2(\RM^n).
  \label{equ:quantize-x}
\end{equation}
Notice that this operator (that we simply denote by $x_j$) is
unbounded, \emph{i.e.} only defined on a dense subset of $L^2(\RM^n)$,
which can be taken to be the set $\Cinf_0$ of compacty supported
smooth functions, or the Schwartz space $\mathscr{S}(\RM^n)$.

On the other hand, the `quantum momentum operator' associated with
$\xi_j$ is the differentiation operator:
\begin{equation}
  L^2(\RM^n) \ni u \mapsto \frac{\h}{i}\deriv{u}{x_j} \in L^2(\RM^n),
  \label{equ:quantize-xi}
\end{equation}
which is equally unbounded. While the Dirac rules are very natural,
the difficulty arises as soon as one wants to quantize polynomials in
$(x,\xi)$. Indeed, there is a choice to make, as position and momentum
do not commute:
\begin{equation}
  \left[\frac{\h}{i}\deriv{}{x_j}\,,\, x_j\right] = \frac{\h}{i}.
  \label{equ:uncertainty}
\end{equation}

Hence both operators $\frac{\h}{i}\deriv{}{x_j} \circ x_j$ and
$x_j \circ \frac{\h}{i}\deriv{}{x_j}$ have the same classical limit
$x_j \xi_j$. This property is called the \emph{uncertainty principle}
because it makes precise the fact that the operation of measuring the
position and the momentum (or speed) of a quantum particle will yield
different results depending on the order with which they are
performed. This uncertainty vanishes in the \emph{semiclassical limit}
$\h\to 0$, and this simple observation goes a long way into building
semiclassical theories where the knowledge of classical mechanics will
give relevant information on the quantum spectrum, \emph{provided}
$\h$ is small enough.

In the mathematics literature, there are two widely used semiclassical
theories:

\begin{enumerate}[leftmargin=*]
\item Weyl quantization on $\RM^n$, or more generally
  pseudo-differential quantization on $M=T^*X$, where $X$ is a smooth
  manifold of dimension $n$
  (see~\cite{guillemin-sternberg-semiclassical}
  or~\cite{zworski-book-12});
\item Berezin-Toeplitz quantization on a prequantizable compact Kähler
  manifold, or the more general version on symplectic manifolds,
  see~\cite{charles-toeplitz} or~\cite{ma-marinescu}.
\end{enumerate}

In both cases, any smooth function on $M$ can be quantized to an
operator on $\mathcal{H}$. However, in the case of Berezin-Toeplitz
quantization, due to the compactness of $M$, the set of admissible
values of $\h$ is quantized ($\h=1/k$, with $k\in\NM^*$) and the
finite dimensional Hilbert space $\mathcal{H}_\h$ must depend on $\h$.

\subsection{The spherical pendulum and the spin-oscillator}

\subsubsection*{Spherical pendulum} 
Recall from Example~\ref{exm:sp} that the spherical pendulum is an
integrable system on $TS^2\simeq T^* S^2\subset T^*\RM^3$ given by the
commuting functions
\[
J(x,y) = x_1y_2 - x_2 y_1; \quad H(x,y) = \frac{1}{2}\lVert y \rVert^2
+ x_3.
\]
It is a good exercise to show that the integrable system $(J,H)$ is
\emph{almost-toric} (see Section~\ref{sec:almost-toric-sing});
see~\cite[Exercise 4]{lf-p-s-vn_exercises} and~\cite{san-panoramas}.

The functions $H$ and $J$ can be viewed as restrictions to $T^*S^2$ of
functions on $T^*\RM^3$, and as such, can be quantized using (for
instance) the Weyl quantization rule (see for instance~\cite[Section
9.6]{guillemin-sternberg-semiclassical}), which in this case is a
direct application of the correspondence
principle~\eqref{equ:quantize-x},~\eqref{equ:quantize-xi}, yielding
the following differential operators acting on functions of
$(x_1,x_2,x_3)$:
\[
\hat J = \frac{\h}{i}\left(x_1\deriv{}{x_2} - x_2\deriv{}{x_1}\right);
\quad \hat H = -\frac{\h^2}{2}\Delta + x_3,
\]
where
$\Delta = \frac{\partial^2}{\partial_{x_1}^2} +
\frac{\partial^2}{\partial_{x_2}^2} +
\frac{\partial^2}{\partial_{x_3}^2}$
is the Laplacian, and $x_3$ in the formula for $\hat H$ stands for
$x_3 \textup{Id}$, \emph{i.e.}  the operator of multiplication by
$x_3$. It is a special property of Weyl's quantization that, since $H$
(or $J$) is a quadratic function, and $\{J,H\}=0$, we get
\[
[\hat H, \hat J] = 0
\]
as a differential operator. Thus, this commutation property remains
when we restrict these operators to functions on $S^2\subset \RM^3$.
The corresponding restricted operators (which we continue to denote by
$\hat J$ and $\hat H$) form a quantum integrable system in the sense
of Definition~\ref{defi:quantum-ci}, where the Hilbert space is the
Lebesgue space $\mathcal{H}=L^2(S^2)$ (the sphere $S^2$ is equipped
with the Euclidean density inherited from $\RM^3$).

Since $S^2$ is a closed manifold, the spectral theory of $\hat H$ is
relatively easy. Note that the restriction of $\Delta$ to smooth
functions on $S^2$ is nothing by the Laplace-Beltrami operator on the
Riemannian $S^2$. It is a general fact that the Laplace-Beltrami
operator on a closed Riemannian manifold is essentially selfadjoint
(see for instance~\cite [Chapter 8]{taylor-II}), and its closure is a
selfadjoint operator with compact resolvent. Since $x_3$ is a bounded
operator, the same conclusion continues to hold for $\hat H$. Thus,
$\hat H$ has a discrete spectrum. For each eigenvalue
$\lambda\in\sigma(\hat H)$, the eigenspace $\ker(\hat H - \lambda)$ is
finite dimensional and, by the ellipticity of $\hat H$, consists of
smooth functions. Since $[\hat J, \hat H]=0$, one can then restrict
the angular momentum operator $\hat J$ to this eigenspace, which gives
a Hermitian matrix, and compute its eigenvalues $\mu_1,\dots, \mu_N$.
The set of all such $(\mu_j,\lambda)\in\RM^2$ constitute the joint
spectrum of $(\hat J, \hat H)$.  A numerical approximation of it (with
$\h=0.1$) is depicted in Figure~\ref{fig:quantum-pendulum}. A striking
fact is that this joint spectrum perfectly fits within the image of
the classical moment map $(J,H)$. That this is so, at least for
small values of $\h$, can be proven in a very general setting, see
Theorem~\ref{theo:quantum1}.

\begin{figure}[h]
  \centering
  \includegraphics[width=0.6\linewidth]{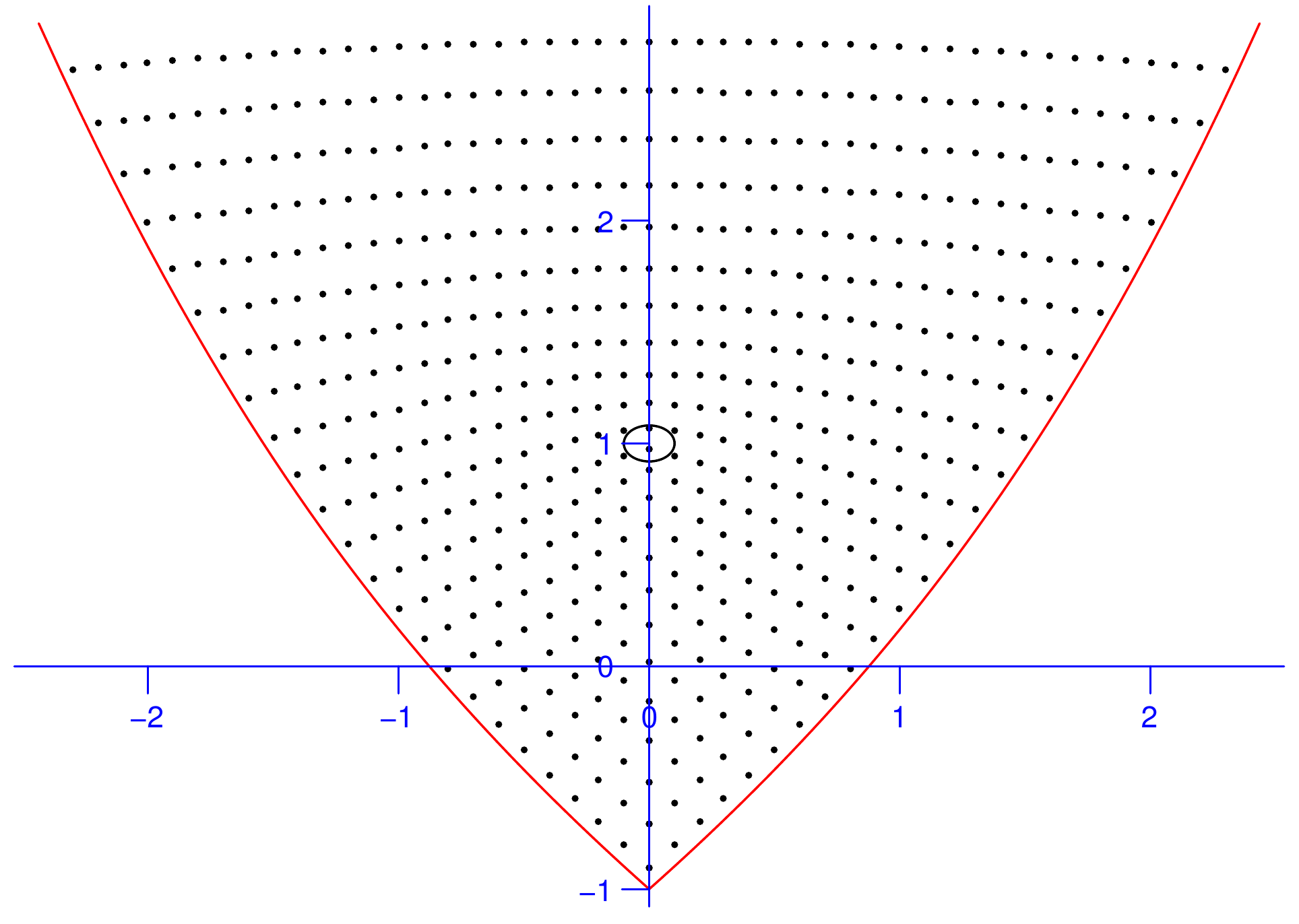}
  \caption{Joint spectrum of the quantum spherical pendulum (black
    dots) for $\h=0.1$. The circle indicates the position of the
    focus-focus critical value, and the red curve consists of all
    other critical values of the moment map $(J,H)$. Compare
    with~\autoref{fig:classical-pendulum}.}
  \label{fig:quantum-pendulum}
\end{figure}

Starting from the seminal articles \cite{cushman-duist,dui},
the spherical pendulum has been an inspiring toy-model for the
understanding of the many links between classical and quantum
integrable systems. The $J$-action is a global $S^1$ symmetry, however
the spherical pendulum is not strictly speaking a semi-toric system
(Definition~\ref{defn:semi-toric}), because the function $J$ is not
proper. A quantum manifestation of this non-properness can be seen on
Figure~\ref{fig:quantum-pendulum}: for each fixed eigenvalue $\mu$ of
$\hat J$, the corresponding eigenspace is infinite dimensional (and
hence there is an infinite number of joint eigenvalues which project
down onto $\mu$). Generalized semi-toric systems with non-proper $J$
can behave in many pathological ways, and their classification is
still open, see~\cite{pel_rat_san_affine}.

\subsubsection*{Spin-oscillator (or Jaynes-Cummings)}
It turns out that there is another simple integrable system whose
local properties are essentially similar to the spherical pendulum,
with the crucial difference that it is a genuine semi-toric system:
the so-called spin-oscillator coupling~\cite{pel_san_spin}, or, in the
physics terminology, the Jaynes-Cummings
system~\cite{jaynes_cummings}. It is very similar to the coupled
angular momenta of Example~\ref{exm:cou_ang_mom}, but it enjoys the
additional property, like the spherical pendulum does, to have a
non-compact phase space.

Again, we let $S^2$ be the unit sphere in $\R^3$ with coordinates
$(x_1,\,x_2,\,x_3)$, and let $\R^2$ be equipped with coordinates
$(u,\, v)$.  Let $M$ be the product manifold $S^2\times\R^2$ equipped
with the product symplectic structure $\omega_{S^2} \oplus \omega_{\mathrm{can}}$,
where $\omega_{S^2}$ is the standard symplectic form on the sphere and
$\omega_{\mathrm{can}}$ is the canonical symplectic form on $\R^2$.  Let
$J,\,H \colon M \to \R$ be the smooth maps defined by
$$J := (u^2+v^2)/2 +  z \,\, \textup{and}\,\, H :=
\frac{1}{2} \, (ux+vy).
 $$ 
 The \emph{coupled spin\--oscillator} is the $4$\--dimensional
 integrable system given by
 $(M,\, \omega_{S^2} \oplus \omega_{\mathrm{can}},\, (J,\,H))$.  As for the
 spherical pendulum, the spin-oscillator is an almost-toric system,
 and it is in fact a semi-toric system
 (Definition~\ref{defn:semi-toric}).  Its bifurcation diagram is very
 similar to that of the spherical pendulum: one isolated focus-focus
 critical value, and two branches of elliptic-regular values connected
 to each other at an elliptic-elliptic value.

 However, from the quantum viewpoint, the Jaynes-Cummings model is
 very different the quantum spherical pendulum: because its phase
 space is not a cotangent bundle, it cannot be quantized using
 (pseudo)differential operators. Moreover, it contains a compact,
 invariant symplectic manifold $S^2\times \{0\}$, and hence can be
 quantized only for a discrete set of values of $\h\in(0,1]$. There
 are two natural ways of obtaining commuting operators for this
 system. One is to view the sphere $S^2$ as a symplectic reduction of
 $\CM^2$ and use invariant differential operators on $\RM^2$,
 see~\cite{pel_san_spin}; another possibility is to perform
 Berezin-Toeplitz quantization of the $S^2$,
 see~\cite{san-pelayo-polterovich}. Figure~\ref{fig:joint-spectrum-spin}
 shows the joint spectrum of the Jaynes-Cummings model which, as was
 the case with the spherical pendulum, nicely fits within the image of
 the classical moment map.
 \begin{figure}[h]
   \centering
   \includegraphics[width=0.7\linewidth]{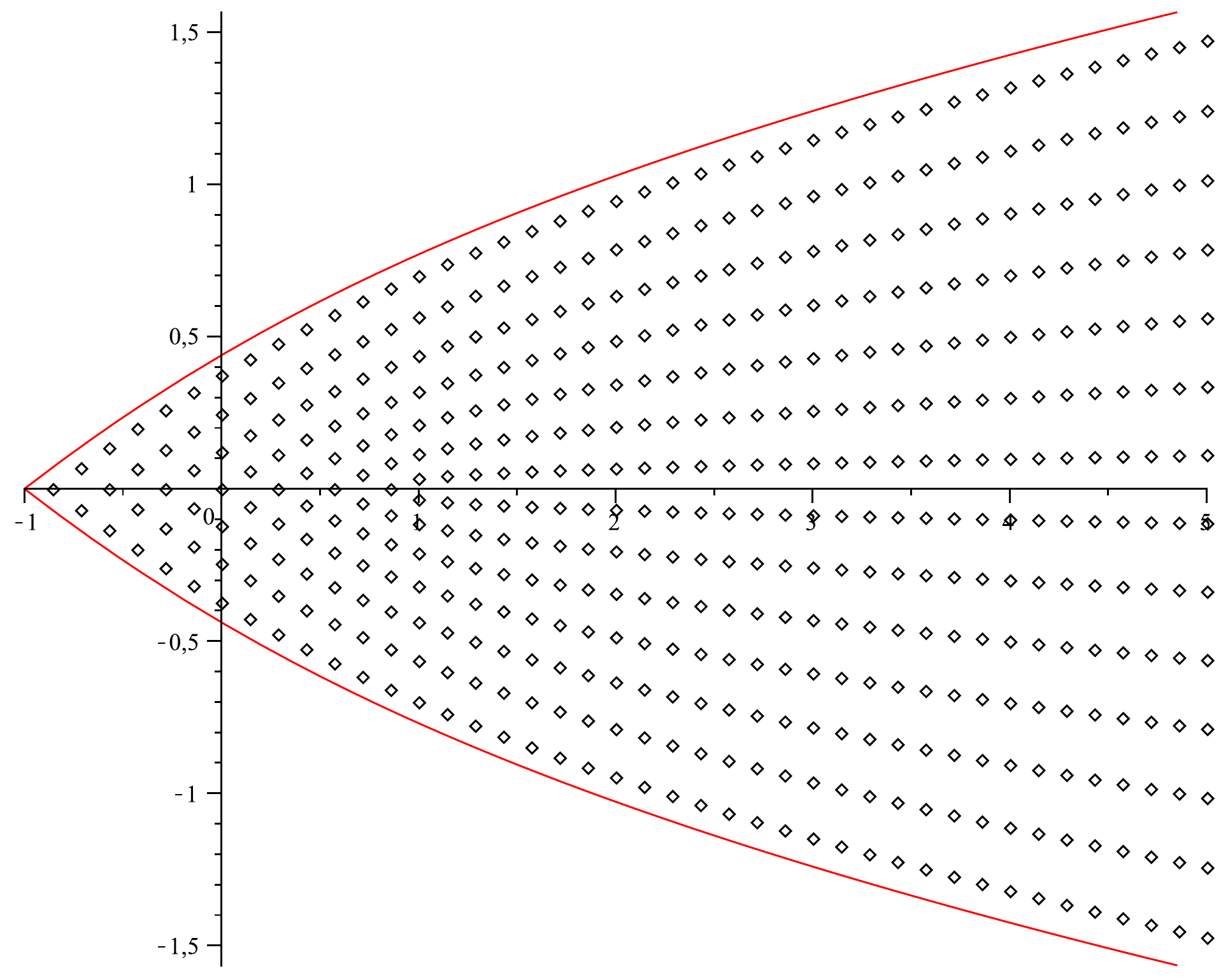}
   \caption{The joint spectrum of the Jaynes-Cummings model
     (cf. \cite[Figure 6]{pel_san_spin}).}
   \label{fig:joint-spectrum-spin}
 \end{figure}

\subsection{Semiclassical quantization}

Following~\cite{san-pelayo-polterovich}, we shall not give the
technical details of the Weyl or Berezin-Toeplitz quantization, but
instead introduce a minimal set of simple axioms, that are satisfied
by these quantizations, and sufficient to understand how to obtain the
semiclassical limit of a joint spectrum.

Let $M$ be a connected manifold (either closed or open).  Let
$\mathcal{A}_0$ be a subalgebra of $\Cinf(M;\RM)$ containing the
constants and all compactly supported functions. We fix a subset
$I \subset (0,1]$ that accumulates at $0$.  If $\mathcal{H}$ is a
complex Hilbert space, we denote by $\mathcal{L}(\mathcal{H})$ the set
of linear (possibly unbounded) selfadjoint operators on
$\mathcal{H}$. By a slight abuse of notation, we write $\norm{T}$ for
the \emph{operator norm} of an operator, and $\norm{f}$ for the
\emph{uniform norm} of a function on $M$.

\begin{defi} \label{defi:semiclassical} A \emph{semiclassical
    quantization} of $(M,\mathcal{A}_0)$ consists of a family of
  complex Hilbert spaces $\mathcal{H}_{\hbar},\; \hbar\in I$, and a
  family of $\RM$\--linear maps
  $\OP \colon \mathcal{A}_0 \to \mathcal{L}(\mathcal{H}_{\hbar})$
  satisfying the following properties, where $f$ and $g$ are in
  $\mathcal{A}_0$:
  \begin{enumerate}[leftmargin=*]
  \item \label{item:one}
    $\norm{\OP (1) - {\rm Id}} = \mathcal{O}(\hbar)$ {\bf
      (normalization)};
  \item \label{item:garding} for all $f \geq 0$ there exists a
    constant $C_f$ such that $\OP (f) \geq -C_f \hbar $ {\bf
      (quasi-positivity)}; (this means
    $\pscal{\OP(f)u}{u}\geq -C\h\norm{u}^2$ for all
    $u\in\mathcal{H}_\h$)
  \item \label{item:norm} let $f\in\mathcal{A}_0$ such that $f\neq 0$
    and has compact support, then
  $$\liminf_{\hbar\to 0}
  \norm{\OP(f)}>0$$ {\bf (non-degeneracy)};
\item \label{item:symbolic} if $g$ has compact support, then for all
  $f$, $\OP(f) \circ \OP(g)$ is bounded, and we have
  \[ \norm{\OP(f) \circ \OP(g) - \OP(fg)} = \mathcal{O}(\hbar),
  \] {\bf (product formula)}.
\end{enumerate}
A \emph{quantizable} manifold is a manifold for which there exists a
semiclassical quantization.
\end{defi}

\medskip
\noindent We shall often use the following consequence of these
axioms: for a bounded function $f$, the operator $\OP(f)$ is
bounded. Indeed, if $c_1 \leq f \leq c_2$ for some $c_1,c_2 \in \RM$,
normalization and quasi-positivity yield
\begin{equation}\label{eq-gard-cor}
  c_1\cdot\textup{Id} - \mathcal{O}(\hbar) \leq \OP(f) \leq c_2\cdot \textup{Id} +  \mathcal{O}(\hbar).
\end{equation}
Since our operators are selfadjoint, this is enough to obtain
\begin{equation} \label{eq-norm-bounds-vsp} \norm{\OP(f)} \leq
  \norm{f} + \cO(\hbar),
\end{equation}
see Lemma~\ref{lemm:rayleigh} below:

\begin{lemm}
  \label{lemm:rayleigh}
  Let $T$ be a (not necessarily bounded) selfadjoint operator on a
  Hilbert space with a dense domain and with spectrum $\sigma(T)$, we
  have
  \begin{equation}
    \sup \sigma(T) = \sup_{u\neq 0} \frac{\pscal{Tu}{u}}{\pscal{u}{u}}\;.
    \label{equ:rayleigh}
  \end{equation}
  In particular,
  \begin{equation}
    \sup \{|s|: s \in \sigma(T)\} = \sup_{u\neq 0} \frac{|\pscal{Tu}{u}|}{\pscal{u}{u}} = \norm{T}\leq \infty ;.
    \label{equ:rayleigh-1}
  \end{equation}
\end{lemm}
\begin{proof}
  If $\lambda\in\sigma(T)$, then by the Weyl criterion, there exists a
  sequence $(u_n)$ with $\norm{u_n}=1$ such that
  \[
  \lim_{n\to\infty} \norm{(T-\lambda {\rm Id})u_n}=0.
  \]
  Therefore, $\lim_{n\to\infty}\pscal{Tu_n}{u_n} = \lambda$, which
  implies that
  \[
  \sup_{\norm{u}=1} \pscal{Tu}{u} \geq \sup \sigma(T).
  \]
  Conversely, if $\sigma(T)$ lies in $(-\infty,c]$, then by the
  spectral theorem $T \leq c \cdot \textup{Id}$ which yields
  \[
  \sup_{\norm{u}=1} \pscal{Tu}{u} \leq \sup \sigma(T).
  \]
\end{proof}

\subsection{Semiclassical operators}

Consider the algebra $\mathcal{A}_I$ whose elements are collections
$\vec{f} =(f_\hbar)_{\hbar \in I}$, $f_\hbar \in \mathcal{A}_0$ with
the following property: for each $\f$ there exists
$f_0 \in \mathcal{A}_0$ so that
\begin{equation}\label{eq-vector}
  f_\hbar = f_0+ \hbar f_{1,\hbar}\;,
\end{equation}
where the sequence $f_{1,\hbar}$ is uniformly bounded in $\hbar$ and
supported in the same compact set $K=K(\f) \subset M$.  The function
$f_0$ is called the {\it principal part} of $\f$. If $f_0$ is
compactly supported as well, we say that $\f$ is \emph{compactly
  supported}.

\begin{defi}
  We define a map
  $$\text{Op}: \mathcal{A}_I \to \prod_{\hbar \in I}
  \mathcal{L}(\mathcal{H}_\hbar),\;\; \f=(f_\hbar) \mapsto
  (\OP(f_\hbar))\;.$$
  {\it A semiclassical operator} is an element in the image of
  $\text{Op}$. Given $\f \in \mathcal{A}_I$, the function
  $f_0 \in \mathcal{A}_0$ defined by \eqref{eq-vector} is called {\it
    the principal symbol} of $\text{Op}(\f)$.
\end{defi}

By \eqref{eq-norm-bounds-vsp}
\begin{equation}\label{eq-cO}
  \OP(f_\hbar) = \OP(f_0)+ \cO(\hbar)\;.
\end{equation}
This together with the product formula readily yields that for every
$\vec{g}$ with compact support and every $\f$,
\begin{equation}
  \label{eq-symbolic-prime}
  \norm{\OP(f_\hbar) \circ \OP(g_\hbar) -
    \OP(f_\hbar g_\hbar)} = \mathcal{O}(\hbar)\;.
\end{equation}

Now we are ready to show that {\it the principal symbol of a
  semiclassical operator is unique}.  Indeed, if $\text{Op}(\f)=0$,
then for any compactly supported function $\chi$, we get by
\eqref{eq-symbolic-prime}
\[
\OP(f_\hbar\chi) = \OP(f_\hbar)\OP(\chi) + \mathcal{O}(\hbar) =
\mathcal{O}(\hbar),
\]
and then by \eqref{eq-cO}, $\OP(f_0\chi) = \mathcal{O}(\hbar)$. By the
normalization axiom, we conclude that $f_0\chi=0$. Since $\chi$ is
arbitrary, $f_0=0$.

\medskip

\begin{rema}
  It is interesting to notice that this abstract semiclassical
  quantization does not use the uncertainty
  principle~\eqref{equ:uncertainty}; and in fact, we don't even
  require $M$ to be symplectic! This, of course, is necessary for
  obtaining finer results, see~\cite{san-panoramas}.
\end{rema}

\subsection{Convergence of the joint spectrum for semiclassical
  integrable systems}

Let $M$ be a quantizable manifold in the sense of
Definition~\ref{defi:semiclassical}. Following Definition~\ref{defi:quantum-ci}, we can say that a semiclassical
integrable system on $M$ is a set of independent commuting
semiclassical operators $(T_1(\h),\dots,T_n(\h))$. Here, `commuting'
has to be understood for any fixed value of $\h$. Let $f_1,\dots,f_n$
be the principal symbols of $T_1(\h),\dots,T_n(\h)$,
respectively. Then, by definition, the term `independent' means that
the differentials $df_1,\dots,df_n$ must be almost everywhere linearly
independent, as in Definition~\ref{defi:CI}.

Because we have not taken into account, in this weak version of
quantization, the Poisson bracket and the uncertainty principle, we
cannot relate the commutation $[T_i(\h),T_j(\h)]=0$ to a classical
property. In fact, in the very general theorem below, we don't even
need the independence of $T_1(\h),\dots, T_n(\h)$. In any case, we may
define the joint spectrum $\Sigma_\h (T_1,\dots,T_d)$, see
Definition~\ref{defi:joint-spectrum}.

\begin{theo}[\cite{san-pelayo-polterovich}]
  \label{theo:quantum1} Let $M$ be a quantizable manifold.  Let
  $d\geq 1$ and let $(T_1,\dots T_d)$ be pairwise commuting
  semiclassical operators on $M$.  Let
  $F=(f_1,\dots,f_n) : M\to\RM^n$, where $f_j$ is the principal symbol
  of $T_j$. Let $J$ be a subset of $I$ that accumulates at $0$.  Then
  from the family
  \[
  \Big\{ \Sigma_\h (T_1,\dots,T_d) \Big\}_{\hbar \in J}
  \]
  one can recover the closed convex hull of $F(M)$.
\end{theo}

The theorem is in fact constructive, in the sense that the convex hull
of $\Sigma_\h (T_1,\dots,T_d)$ converges locally, in the Hausdorff
sense, to $F(M)$; more precisely, we have:
\begin{theo}[\cite{san-pelayo-polterovich}]
  Under the hypothesis of Theorem~\ref{theo:quantum1}, if the
  operators $T_j$ are uniformly bounded as $\h\to 0$, then the closed
  convex hull of $\Sigma_\h (T_1,\dots,T_d) $ converges in the
  Hausdorff metric, as $\h\to 0$, to the closed convex hull of $F(M)$.
\end{theo}

\begin{proof}
  We restrict here to the one-dimensional case ($n=1$); the general
  case can be recovered by using linear combinations of the form
  $T_\xi:=\sum \xi_j T_j$ (see~\cite{san-pelayo-polterovich}). If
  $n=1$, the statement is quite easy to write down:
  \begin{quote}
    «Prove that
    \[ [\inf \sigma(T), \sup \sigma(T)] \to [\inf f, \sup f] \quad \text{ as }
    {\h\to 0},
    \]
    where $T$ is a bounded semiclassical operator, and $f$ its
    principal symbol.»
  \end{quote}
Let $\epsilon>0$ be fixed, independent on $\h$. The easy part is to prove that, when $\h$ is small enough,
\[
\sup \sigma(T) \leq \sup f + \epsilon,
\]
as this is a direct consequence
of~\eqref{eq-norm-bounds-vsp},~\eqref{equ:rayleigh},
and~\eqref{equ:rayleigh-1}. Now let us show that, conversely,
\[
\sup \sigma(T) \geq \sup f - 2\epsilon.
\]
The strategy is to construct a good `test function' $u$ that is close
to realizing the $\sup$ in~\eqref{equ:rayleigh}. Let
$F_\epsilon:=\sup f - \epsilon$; since $f$ is continuous, there exists
a connected open set $B\subset M$ where $f\geq F_\epsilon$. Let
$\chi\geq 0$ with $\chi\in\Cinf_0(B)$, \emph{i.e.} $\chi$ is smooth
function on $M$ with compact support $K\subset B$. Thus
$(f-F_\epsilon)\chi\geq 0$ on $M$ and $(f-F_\epsilon)\chi=0$ outside
of $K$ (see \autoref{fig:quasimode}). Let $\tilde\chi$ be
another cut-off function, with $\tilde\chi\geq 0$, $\tilde\chi = 1$ on
a open set $\tilde B$ whose closure is contained in $B$, and
$\tilde\chi\in\Cinf_0(K)$, so that $\tilde \chi \chi = \tilde \chi$,
and hence, by the product formula,
\begin{equation}
  \OP\chi\circ\OP\tilde\chi = \OP \chi\tilde\chi + \O(\h) = \OP
  \tilde\chi + \O(\h).
  \label{equ:chi}
\end{equation}
\begin{figure}[h]
  \centering
  \includegraphics[width=0.5\linewidth]{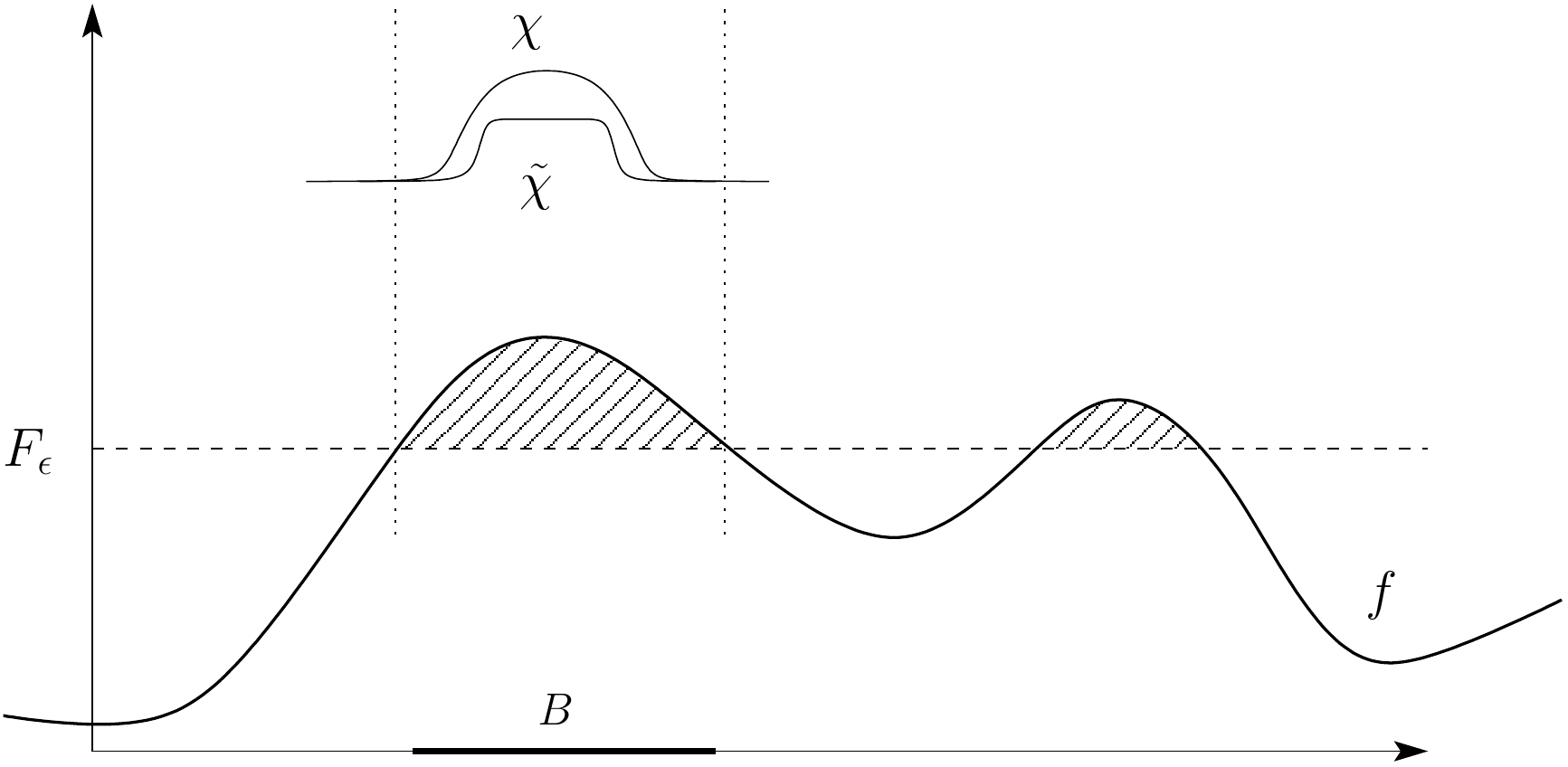}
  \caption{Construction of the `quasimode' $w_\h$}
  \label{fig:quasimode}
\end{figure}

The non-degeneracy axiom,
together with Lemma~\ref{lemm:rayleigh}, implies that one can find a
vector $v_\h\in\mathcal{H}_\h$ such that
\[
\norm{(\OP\tilde\chi)v_\h} > c/2, \qquad c=\norm{\OP\tilde\chi},
\]
uniformly in $\h \leq 1$. Let 
\[
u_\h := \frac{(\OP\tilde\chi) v_\h}{\norm{(\OP\tilde\chi) v_\h}};
\]
then, using~\eqref{equ:chi}, we obtain
\[
(\OP\chi) u_\h = u_\h + \O(\h).
\]
Finally, let $w_\h:=(\OP\chi) u_\h$; we have
\[
\OP(f-F_\epsilon) w_\h = \OP((f-F_\epsilon)\chi) u_\h + \O(\h).
\]
Hence, the quasi-positivity axiom implies
\[
\pscal{\OP(f-F_\epsilon) w_\h}{u_\h} \geq -C_\epsilon \h \norm{u_\h},
\]
for some constant $C_\epsilon$ depending on $\epsilon$.
Since $\norm{u_\h}=1$ and $\norm{w_\h}=1+\O(\h)$, we get, for a possibly different constant $\tilde C_\epsilon$, 
\[
\pscal{\OP(f-F_\epsilon) w_\h}{w_\h} \geq -\tilde C_\epsilon \h
\norm{u_\h},
\]
which implies, due to Lemma~\ref{lemm:rayleigh}, that $\sup\sigma(T)\geq F_\epsilon - \tilde C_\epsilon\h$. Thus, when $\h$ is small enough, we can reach
\[
\sup\sigma(T)\geq F_\epsilon - \epsilon,
\]
as required.
\end{proof}

\subsection{The inverse problem in the toric case}

Theorem~\ref{theo:quantum1} has a rather spectacular consequence in
the case of toric systems.  Recall that an integrable system is called
\emph{toric} if the Hamiltonian flow of each function $f_1,\dots,f_n$
is $2\pi$-periodic, and the corresponding $\T^n$-action is effective,
see Example~\ref{exm:symp_toric}. Accordingly, a set of commuting
semiclassical operators $T_1,\dots,T_n$ will, by definition,
constitute a \emph{quantum toric system} if the principal symbols
$f_1,\dots,f_n$ form a toric system.

Such systems have been studied in details in~\cite{san-charles-pelayo}
in the framework of Berezin-Toeplitz quantization. It was proven that
the joint spectrum of such a system is a regular deformation (in the
semiclassical parameter $\h$) of the set of $\h$-integral points of
the image of the moment map $\mu:=(f_1,\dots,f_n)$; precisely:
\[
\Sigma_\h(T_1,\dots,T_n) = g_\h(\mu(M)\cap (v+\h\ZM^n)) +
\O(\h^\infty),
\]
where $v$ is any vertex of the polytope $\mu(M)$, and $g_\h$ is a
deformation of the identity:
\[
g_h = \textup{Id} + \h g_1 + \h^2 g_2 + \cdots,
\]
in the sense of an asymptotic expansion in the $\Cinf(\RM^n)$
topology.

As a consequence, the inverse problem was solved; in fact, this also
follows directly from Theorem~\ref{theo:quantum1} for a general
semiclassical quantization:

\begin{theo}[\cite{san-charles-pelayo,san-pelayo-polterovich}]
\label{theo:q-toric}
  In the class of quantum toric systems (on a compact symplectic
  manifold $M$), the asymptotics of the joint spectrum completely
  determines the symplectic manifold $M$ and the toric moment map
  $\mu$.
\end{theo}
\begin{proof}
  The proof is a simple application of Theorem~\ref{theo:quantum1}, in
  view of the Delzant classification~\cite{delz}. Indeed, the
  asymptotics of the joint spectrum determine the closed convex hull
  of $\mu(M)$. But we know that in the class of toric systems, the
  image $\mu(M)$ is closed and convex. Thus, we may recover $\mu(M)$
  from the joint spectrum. By the Delzant result, $\mu(M)$ in turn
  completely determines $(M,\mu)$.
\end{proof}

\subsection{The general inverse problem}
In view of Theorem~\ref{theo:q-toric} above, it is tempting to state a
general conjecture (which is a slight refinement of the statements
in~\cite[Conjecture 9.1]{san-pelayo-bams11}, \cite[Conjecture
3.4]{san-alvaro-first-steps}), as follows:
\begin{conj}
  Let $\mathscr{N\!D}_n$ be the class of completely integrable systems
  on a $2n$ symplectic manifold with non-degenerate singularities (see
  Definition~\ref{defn:non-deg}). Let $Q(\mathscr{N\!D}_n)$ be the
  class of $n$-uples $T=(T_1,\dots,T_n)$ of commuting operators whose
  principal symbols form an element in $\mathscr{N\!D}_n$. Then the
  asymptotics of the joint spectrum of an element in
  $T\in Q(\mathscr{N\!D}_n)$ completely determines the symplectic
  manifold and the principal symbols of $T$.
\end{conj}
At the time of writing, this conjecture is open. However, preliminary
results can shed some light. In one degree of freedom, the conjecture
was shown to hold in a generic (and large) subset of
$\mathscr{N\!D}_1$~\cite{san-inverse}. Genericity was necessary
because of possible symmetries that generate eigenvalues with
multiplicities. Hence, for a general statement, it should be clear
that `joint spectrum' is understood as a spectrum with multiplicities.

Theorem \ref{theo:q-toric} states that the conjecture holds when
$\mathscr{N\!D_n}$ is replaced by the set of toric systems. Recently,
several papers have been trying to attack the semi-toric case (which
was already formulated as a conjecture in~\cite{san-pelayo-bams11},
cf. \cite{san-alvaro-first-steps,san-lefloch-pelayo:jc} and references
therein). In~\cite{san-lefloch-pelayo:jc}, the authors prove the
following result:
\begin{theo}[\cite{san-lefloch-pelayo:jc}]
  For quantum semi-toric systems which are either semiclassical
  pseudo-differential operators, or semiclassical Berezin-Toeplitz
  operators, the joint spectrum (modulo $\O(\h^2)$) determines the
  following invariants:
\begin{enumerate}[leftmargin=*]
  \item[{\rm (1)}] the number $m_{f}$ of focus-focus values,
  \item[{\rm (2)}] the Taylor series invariant (see
    Definition~\ref{defn:taylor_invariant_st}),
  \item[{\rm (3)}] the volume invariant associated with each
    focus-focus value,
  \item[{\rm (4)}] the polygonal invariant of the system.
  \end{enumerate}
\end{theo}
The last two ingredients are part of what we describe in this work as
the \emph{cartographic invariant}, see
Section~\ref{sec:cart-invar-case}.  In order to obtain a full answer
in the semi-toric case, it remains on the one hand to be able to
detect the full cartographic invariant (which means detecting the
twisting cocycle), and on the other hand to recover exactly the
principal symbols, not only up to isomorphism (in other words, if two
systems have the same joint spectrum, can we prove that the map $g$ in
Definition~\ref{defn:equiv_st} is the identity?).

In comparison to semi-toric systems, more general integrable systems
in $\mathscr{N\!D}_n$ become quickly much more delicate to analyze,
due to the presence of \emph{hyperbolic singularities}, which allow
for non-connected fibers of the moment map. It is currently not
known how to obtain a tractable classification of integrable systems
with hyperbolic singularities; however a solid topological foundation
was laid out in the book~\cite{bolsinov-fomenko-book}. A reasonable
approach would be to first consider the case of hyperbolic
singularities in the presence of a global $S^1$-action.

Finally, we conclude with two open questions closely related to the
conjecture.
\begin{enumerate}[leftmargin=*]
\item Can one detect from the joint spectrum of a system in
  $Q(\mathscr{N\!D}_n)$ whether the system is semi-toric?
\item Can one tell from the joint spectrum of a general quantum
  integrable system whether it belongs to $Q(\mathscr{N\!D}_n)$?
\end{enumerate}

\bibliographystyle{abbrv}%
\bibliography{bibli}
\end{document}